\documentclass[12pt]{amsart}
\usepackage{amsfonts,latexsym,amsmath, amssymb}
\usepackage{mathrsfs,MnSymbol}
\usepackage{url,color}
\usepackage{upgreek}
\usepackage{fancyhdr}
\usepackage{hyperref}
\usepackage{times}
\usepackage{scalefnt}
\usepackage{url}
\usepackage{cite}

\newcommand{\bea}{\begin{eqnarray}}
\newcommand{\eea}{\end{eqnarray}}
\def\beaa{\begin{eqnarray*}}
\def\eeaa{\end{eqnarray*}}
\def\ba{\begin{array}}
\def\ea{\end{array}}
\def\be#1{\begin{equation} \label{#1}}
\def \eeq{\end{equation}}

\def\be{{\beta}}

\def\ep{\epsilon}

\newcommand{\<}{  \langle   }
\renewcommand{\>}{  \rangle   }

\newtheorem{theorem}{Theorem}[section]
\newtheorem{lemma}[theorem]{Lemma}
\newtheorem{proposition}[theorem]{Proposition}
\newtheorem{corollary}[theorem]{Corollary}
\newtheorem{definition}[theorem]{Definition}
\newtheorem{remark}[theorem]{Remark}

\numberwithin{equation}{section}

\numberwithin{equation}{section}

\topmargin - .23 in

\begin{document}


\title{Global well-posedness and scattering for mass-critical, defocusing, infinite dimensional vector-valued resonant nonlinear Schr\"odinger system  }
\author{Kailong Yang \and Lifeng Zhao}


\thanks{ Wu Wen-Tsun Key Laboratory of Mathematics, Chinese Academy of Sciences
 and Department of Mathematics, University of Science and Technology of China, Hefei 230026, \ Anhui, \ China. \texttt{ykailong@mail.ustc.edu.cn, zhaolf@ustc.edu.cn} }



\begin{abstract}

	In this article, we consider the infinite dimensional vector-valued resonant nonlinear Schr\"odinger system, which arises from the study of the asymptotic behavior of the defocusing  nonlinear Schr\"{o}dinger equation on ``wave guide" manifolds like $\mathbb{R}^2\times \mathbb{T}$ in \cite{CHENG}. We show global well-posedness and scattering for this system by long time Strichartz estimates and frequency localized interaction  Morawetz estimates. As a by-product, our results  make the arguments of scattering theory in \cite{CHENG} closed as crucial ingredients for compactness of the critical elements.
\bigskip

\bigskip


\end{abstract}

\maketitle

\setcounter{tocdepth}{2}
\pagenumbering{roman} \tableofcontents \newpage \pagenumbering{arabic}
\section{Introduction}
In this paper we study the initial value problem of the infinite dimensional vector-valued  resonant nonlinear Schr\"{o}dinger system,
\begin{equation}\label{eq-y8.1'}
\begin{cases}
i\partial_t u_j + \Delta u_j =\sum\limits_{\mathcal{R}(j)}u_{j_1}\bar{u}_{j_2}u_{j_3},\\
u_j(0) = u_{0,j},
\end{cases}
\end{equation}
with unknown $\vec{u}=\{u_j\}_{j\in\mathbb{Z}}$, where $u_j: \mathbb{R}^2_x \times \mathbb{R}_t \rightarrow \mathbb{C}$,
$$\mathcal{R}(j)=\{(j_1,j_2,j_3)\in \mathbb{Z}^3,\ \ j_1-j_2+j_3 = j,\ \ |j_1|^2-|j_2|^2 +|j_3|^2 = |j|^2\}.$$
 In the special case when $u_j=0$ for $j \neq 0$, we recover the cubic Schr\"{o}dinger equation
\begin{equation}\label{eq-y8.2'}
i\partial_t u+ \Delta u =|u|^2u.
\end{equation}
Recently, there has been increasing interest in infinite dimensional vector-valued resonant nonlinear Schr\"{o}dinger system. Such equations have also been extensively studied in applied sciences on various backgrounds.  These finite or infinite systems of nonlinear Schr\"{o}dinger equations arise independently in the study of nonlinear optics in wave guides and are the object of several previous studies (see, e.g.,\cite{JH}, \cite{VS}, \cite{JDA}, the books\cite{APT}, \cite{SS}  and references therein). We should point out that their study on ``wave guide" manifolds  seems to be of particular interest, especially in nonlinear optics of telecommunications \cite{DML,QJP,T1}. In addition, the infinite dimensional vector-valued  resonant nonlinear Schr\"{o}dinger systems have also appeared in the study of the asymptotic behavior of the defocusing nonlinear Schr\"{o}dinger equation on ``wave guide" manifolds which are partially compact like $\mathbb{R}^2\times \mathbb{T}$ in \cite{CHENG} and $\mathbb{R}\times \mathbb{T}^2$ in \cite{HP}. In fact, from  Hani and Pausader \cite{HP} and Cheng, Guo, Yang and Zhao \cite{CHENG}, the  infinite dimensional vector-valued resonant system \eqref{eq-y8.1'} is derived during the construction of profile decomposition, which is an important step to get scattering of the defocusing  nonlinear Schr\"{o}dinger equation on corresponding ``wave guide" manifolds. Moreover, the scattering of the  infinite dimensional vector-valued resonant system \eqref{eq-y8.1'} is equivalent to that of the  Schr\"{o}dinger equation on corresponding ``wave guide" manifolds $\mathbb{R}^2\times \mathbb{T}$ ($\mathbb{R}\times \mathbb{T}^2$). While the study of the asymptotic behavior of nonlinear Schr\"{o}dinger equation on compact or partially compact manifolds has a great influence on better understanding the broad question of the effect of the geometry of the domain on the asymptotic behavior of solutions, especially large solutions to nonlinear dispersive equations. The study of solutions of the nonlinear Schr\"{o}dinger equation on compact or partially compact domains has a long history. We will not attempt here to make exhaustive list of those works, We just refer to \cite{VRT,NN} and the references therein for previous works on the relation between scattering and geometry.    This in turn illustrates the importance of the study of the infinite dimensional vector-valued resonant system \eqref{eq-y8.1'}.

Defocusing Schr\"{o}dinger equations have been widely studied for many years. The scattering in energy space for energy-subcritical and mass-supercritical nonlinear  Schr\"odinger equations in the defocusing case was proved in  Ginibre and  Velo \cite{JG} by exploiting Morawetz estimates and approximate finite speed of propagation.  The scattering of mass critical Schr\"odinger equations in the radial case was studied by  Killip,  Tao,  Visan and  Zhang in \cite{KV2, KTV2, TVZ1, TVZ2} when $d\geq2$. Recently the radial assumption was removed in all dimensions by  Dodson \cite{D, D1, D3, D4},  who proved scattering for initial data of finite mass in defocusing case and the dichotomy below the ground state mass in the focusing case. The infinite dimensional vector-valued resonant Schr\"{o}dinger system \eqref{eq-y8.1'} retains many properties of the defocusing, mass-critical Schr\"{o}dinger equation for $d=2$. For instance, the system \eqref{eq-y8.1'} possesses mass conservation and energy conservation (see Section 2), the norm of the solution to system \eqref{eq-y8.1'} is scale invariant under the mass-critical space $L_x^2h^1$, which is defined by $\|\vec{u}\|_{L^2_xh^1}=\Big\|\big(\sum\limits_{j\in\mathbb{Z}} \langle j\rangle^{2}|u_j|^2 \ \big)^\frac12 \Big\|_{L_x^2}$. In this paper, we aim to establish global wellposedness and scattering of the infinite dimensional vector-valued resonant system \eqref{eq-y8.1'} by Dodson's argument in \cite{D}.

In \cite{D},  Dodson established scattering theory for mass-critical, defocusing Schr\"{o}dinger equation for $d=2$  by  ``concentration compactness" argument. This so-called `` concentration compactness" method has been utilized since at least the 1980's (\cite{BL}, \cite{BC}) in a wide variety of partial differential equations, including elliptic, hyperbolic, parabolic and geometric partial differential equations. Historically, such progress has gone energy-critical wave \cite{KM1}, energy-critical Schr\"{o}dinger equations \cite{D5,KM,KV1} and then mass-critical Schr\"{o}dinger equations \cite{D, D1, D3, D4}, although there are exceptions to this. The essential ingredients in \cite{D} are long time Strichartz estimate and frequency localized interaction Morawetz estimate. The long time Strichartz estimate was first derived by Dodson in \cite{D3} to exclude the rapid frequency cascade and the quasi-soliton two scenarios, and then was developed into more complicated versions based on $U^p_{\triangle},V^p_{\triangle}$ spaces in \cite{D} and \cite{D1} to deal with the low dimension case $d=1,2$. Frequency localized interaction Morawetz estimate were first formulated by  Colliander, Keel, Staffilani, Takaoka and Tao in \cite{CKSTT}  to exclude minimal energy blow-up solutions to energy-critical Schr\"{o}dinger equations. Later, it was developed by Dodson to exclude minimal mass blow-up solutions to  mass-critical Schr\"{o}dinger equations in \cite{D, D1, D3, D4}.  Theses techniques with the compact contradiction arguments make it possible to exclude the rapid frequency cascade and the quasi-soliton two scenarios, which yields the non-existence of critical element, thus proving the scattering.

In this paper, we prove the scattering for the infinite dimensional vector-valued resonant system \eqref{eq-y8.1'}. This result is crucial  for compactness of the critical elements and in turn make the arguments of scattering theory in \cite{CHENG} closed. The main theorem is as follows.
\begin{theorem}\label{th-y8.1'}
The initial value problem \eqref{eq-y8.1'} is globally well-posed and scattering for any $\vec{u}_0=\{u_{0,j}\}_{j\in \mathbb{Z}} \in L^2_xh^1 (\mathbb{R}^2\times \mathbb{Z})$. Here we say the solution $\vec{u}(t,x)$ to \eqref{eq-y8.1'} scattering means that there exist  $\vec{u}^{\pm}\in L_x^2h^1$  such that
\begin{equation}
  \lim_{t\rightarrow \pm\infty} \|e^{it\triangle}\vec{u}^{\pm}-\vec{u}(t)\|_{L_x^2h^1}=0.
\end{equation}
\end{theorem}

In the proof of Theorem \ref{th-y8.1'}, there are two ingredients, one is the long time Strichartz estimate
\begin{theorem}[Long time Strichartz estimate]\label{th-y1.2}
Suppose $\vec{u}(t)$ is a minimal mass blowup solution to \eqref{eq-y8.1'},
for $\vec{u}_{0}=\{u_{0,j}\}_{j\in\mathbb{Z}}\in L^2_xh^1(\mathbb{R}^2\times\mathbb{Z}).$
Then there exists a constant $C>0$ (only depending on $\vec{u}$), such that for any $M=2^{k_0}$, $0<\ep_3 \ll \ep_2\ll \ep_1<1$ satisfying \eqref{eq-y5.1},  \eqref{eq-y5.2} (see Section 5), $\|\vec{u}\|^4_{L^4_{t,x}l^{2}([0,T])}=M$ and $\int_0^T N(t)^3 dt=\ep_3M$,
$$\|\vec{u}\|_{\tilde{X}_{k_0}([0,T])}\leq C.$$
\end{theorem}
\noindent The other is the frequency localized interaction Morawetz estimate
\begin{theorem}[Frequency localized interaction Morawetz estimate]\label{th-y1.1}
 Suppose $\vec{u}(t,x)$ is a minimal mass blowup solution to \eqref{eq-y1} on $[0,T]$ with
$\int_0^T N(t)^3 dt=K$.  Then
\begin{equation}
 \|\sum_{j\in\mathbb{Z}}|\nabla|^{1/2}|P_{\leq \frac{10K}{\ep_1}}u_j(t,x)|^2\|^2_{L^2_{t,x}([0,T]\times \mathbb{R}^2)} \lesssim o(K),
\end{equation}
where $o(K)$ is a quantity such that $\frac{o(K)}{K}\rightarrow0 $ as $K\rightarrow \infty$.
\end{theorem}
\begin{remark}
Notice that  both our long time Strichartz estimate and frequency localized interaction Morawetz estimate are $l^{2}$ level at ``$j$" direction (compared with the initial data $\vec{u}_{0}$ in $h^1$ at ``$j$" direction), which allow us to obtain $\|\vec{u}\|_{L^4_{t,x}l^{2}([0,\infty))}< \infty$ by compact contradiction arguments, then we can use the classical method of higher regularity to get $\|\vec{u}\|_{L^4_{t,x}h^{1}([0,\infty))}< \infty$, this implies Theorem \ref{th-y8.1'}.
\end{remark}
The proof of the Theorem \ref{th-y1.2} relies on three bilinear Strichartz  estimates (see Theorem \ref{th-y5.16}, \ref{th-y5.17} and \ref{th-y5.19}), whose proof rely on the interaction Morawetz estimate of \cite{FL}. Such estimates will give a logarithmic improvement over what would be obtained from \eqref{eq-y4.2'} directly. This improvement is quite helpful to the proof.

Although the bulk of our arguments is similar to Dodson's \cite{D} and \cite{D1}, there are still some differences. First, we give the description of almost periodic solution in Hilbert space $L^2_xh^1$ (see Proposition \ref{pr-y4.8} below) and then construct $U^p_{\triangle}(h^1),V^p_{\triangle}(h^1)$ spaces based on the Hilbert space $L^2_xh^1$; Second,  there are some difficulties due to the asymmetry of nonlinear term. This is also why we construct our long time Strichartz estimate in $l^{2}$ level at ``$j$" direction. For example,  when we utilize the interaction Morawetz estimate to prove Theorem \ref{th-y5.17},  we need to estimate the nonlinear terms. If we take $h^1$ norm at ``$j$" direction, this would destroy the symmetry of nonlinear term.  To overcome these difficulties, we need take $l^2$ norm at ``$j$" direction. Combing the specific nonlinear summation,
\begin{equation}\label{1.1a}
\sum_{\mathcal{R}(j)}u_{j_1} \bar{u}_{j_2} u_{j_3}= |u_{j}|^2 u_{j} + 2u_{j}\sum_{j_1\neq j} |u_{j_1}|^2.
\end{equation}
we obtain
\begin{equation}
\left(\sum\limits_{j\in\mathbb{Z}} |\sum\limits_{\mathcal{R}(j)} u_{j_1} \bar{u}_{j_2} u_{j_3}|^2\right)^\frac{1}{2} \lesssim\left\|\vec{u}\right\|^3_{l^{2}},
\end{equation} then $\|\vec{u}\|_{l^{2}}$ can be regarded  as a space-time function.  Although this makes our calculation more complicated, we have still completed our estimates.

The paper is organized as follows: Section 2  contains the local well-posedness and small-data scattering along with the associated stability theory for the resonant system \eqref{eq-y8.1'}; in Section 3 we reduce the minimal mass blowup solution to almost periodic solution and give a description of almost periodic solutions in $L_x^2h^1$; in Section 4 we construct $U^p_{\triangle}$,  $V^p_{\triangle}$ space and bilinear Strichartz estimate we need for the rest of our work; in Section 5 and Section 6 we calculate long time Strichartz estimates and frequency localized Morawetz estimates for $d=2$ specifically and then use these estimate to complete the proof of Theorem \ref{th-y8.1'}; 
\vspace*{0.2cm}

{\bf Notation and Preliminaries}
We will use the notation $X\lesssim Y$ whenever there exists some constant $C>0$ so that $X \le C Y$. Similarly, we will use $X \sim Y$ if
$X\lesssim Y \lesssim X$.

In the following we let $\mathcal{R}(j)=\{(j_1,j_2,j_3)\in \mathbb{Z}^3,\ \ j_1-j_2+j_3 = j,\ \ |j_1|^2-|j_2|^2 +|j_3|^2 = |j|^2\}$, $\vec{u} = \{u_j\}_{j\in \mathbb{Z}}$, and $\mathbf{F}(\vec{u}):=\{F_j(\vec{u})\}_{j\in\mathbb{Z}}:=
\{\sum\limits_{\mathcal{R}(j)}u_{j_1}\bar{u}_{j_2}u_{j_3}\}_{j\in\mathbb{Z}}
=:\sum^{\rightarrow}\limits_{\mathcal{R}(j)}u_{j_1}\bar{u}_{j_2}u_{j_3}$.

Define the norm
$\|\vec{u}\|_{L_t^p L_x^q h^{a}(I\times \mathbb{R}^2\times \mathbb{Z})} = \left\|\Big\|\big(\sum\limits_{j\in\mathbb{Z}} \langle j\rangle^{2a}|u_j|^2 \ \big)^\frac12 \Big\|_{L_x^q}\right\|_{L_t^p}$, $\|\vec{u}\|_{L^p_xh^a}=\Big\|\big(\sum\limits_{j\in\mathbb{Z}} \langle j\rangle^{2a}|u_j|^2 \ \big)^\frac12 \Big\|_{L_x^p}$ and $\|\vec{u}\|_{h^{a}}=\big(\sum\limits_{j\in\mathbb{Z}} \langle j\rangle^{2a}|u_j|^2 \ \big)^\frac12$, $a\geq 0$. Note that $\|\vec{u}\|_{h^0}=\|\vec{u}\|_{l^{2}}$.

Let $\phi \in C^{\infty}_0(\mathbb{R}^2)$ be a radial, decreasing function
\begin{equation*}\phi(x):=
\left\{
  \begin{array}{ll}
    1, & \hbox{$|x|\leq 1$,} \\
    0, & \hbox{$|x|>2$.}
  \end{array}
\right.
\end{equation*}
Define the partition of unity
\begin{equation*}
1=\phi(x)+\sum_{j=1}^{\infty}[\phi(2^{-j}x)-\phi(2^{-j+1}x)]=\psi_0(x)+\sum_{j=1}^{\infty}\psi_j(x).
\end{equation*}
For any integer $j\geq 0$, let
\begin{equation*}
  P_{j}f= \mathcal{F}^{-1}(\psi_j(\xi)\hat{f}(\xi))=\int K_j(x-y)f(y)dy,
\end{equation*}
where $K_j$ is an $L^1$ -kernel. When $j$ is an integer less than zero, let $P_j f = 0$. Finally let
\begin{equation*}
  P_{j_1\leq\cdot\leq j_2}f=\sum_{j_1\leq\j\leq j_2}P_j f.
\end{equation*}
We also define the frequency truncation
\begin{equation*}
  P_{\leq j}f= \mathcal{F}^{-1}(\phi(2^{-j}\xi)\hat{f}(\xi)).
\end{equation*}
Let $\xi_0\in \mathbb{R}^2$, then define $P_{\xi_0,j}u=e^{ix\cdot\xi_0}P_j(e^{-ix\cdot\xi_0})u$ and $P_{\xi_0,j}\vec{u}=\{P_{\xi_0,j}u_k\}_{k\in\mathbb{Z}}$. Similarly we can define $P_{\xi_0,\leq j}\vec{u}$ and $P_{\xi_0,\geq j}\vec{u}$. Finally we point out  $P_{\xi_0,j}\vec{\bar{u}}=\{\overline{P_{\xi_0,j}u_k}\}_{k\in\mathbb{Z}}$.

\section{Local wellposedness and small data scattering}

Let's recall the mass critical resonant system
\begin{equation}\label{eq-y1}
\begin{cases}
i\partial_t u_j + \Delta_{\mathbb{R}^2} u_j = \sum\limits_{\mathcal{R}(j)} u_{j_1} \bar{u}_{j_2} u_{j_3},\\
u_j(0) = u_{0,j},
\end{cases}
\end{equation}
where $\mathcal{R}(j)=\{(j_1,j_2,j_3)\in \mathbb{Z}^3,\ \ j_1-j_2+j_3 = j,\ \ |j_1|^2-|j_2|^2 +|j_3|^2 = |j|^2\}$. Denote $\vec{u} = \{u_j\}_{j\in \mathbb{Z}}$,  $\vec{u}_{0} = \{u_{0,j}\}_{j\in \mathbb{Z}}$.

As showed in \cite{HP} and \cite{CHENG}, the system \eqref{eq-y8.1'} possesses conserved mass
$$M(\vec{u}(t))=\int_{\mathbb{R}^{2}}\sum\limits_{j\in\mathbb{Z}}g(j)|u_{j}(t,x)|^{2}\mathrm{d}x=M(\vec{u}(0)),$$
where $g(j)=a+\<b,j\>+c|j|^2$ with $a,c\geq0$ and $b\in \mathbb{Z}$ (we always take $g(j)=1$ or $g(j)=1+|j|^2:=\<j\>^2$),\\
and energy
$$E(\vec{u}(t))=\frac{1}{2}\int_{\mathbb{R}^{2}}\sum\limits_{j\in\mathbb{Z}}|\nabla u_{j}(t,x)|^{2}\mathrm{d}x+\frac{1}{4}\int_{\mathbb{R}^{2}}\sum\limits_{j\in\mathbb{Z}} \bar{u}_{j} \sum\limits_{\mathcal{R}(j)}u_{j_1}\bar{u}_{j_2}u_{j_3}\mathrm{d}x=E(\vec{u}(0)).$$
It's worthing noticing that the second term in energy is nonnegative. We first have
\begin{lemma}[Littlewood-Paley characterization]\label{le-y2.1}
For $1<p<\infty$, $a\in\{+1,0\}$,
\begin{equation}\label{45986}
\|\vec{f}\|_{L^p_xh^a}:=\Big\|\big(\sum\limits_{j\in\mathbb{Z}} \langle j\rangle^{2a}|f_j|^2 \ \big)^\frac12 \Big\|_{L_x^p}\sim \Big\|\big(\sum\limits_{j\in\mathbb{Z}} \langle j\rangle^{2a}\sum_{k}|P_{k}f_j|^2 \ \big)^\frac12 \Big\|_{L_x^p}=:\left\|\big(\sum_{k}|P_{k}\vec{f}|^2\big)^{1/2}\right\|_{L^p_xh^a}.
\end{equation}
\end{lemma}
\begin{proof}
See Proposition 5.1.4 of Chapter 5 in \cite{LG1} and Section 5.3 of Chapter 4 in \cite{EM1}.
\end{proof}
\noindent We also have the following Strichartz estimate.

\begin{proposition}[Strichartz estimate]\label{pr-y3.1}
\begin{align}
\|e^{it\Delta} \vec{f}\|_{L_t^p L_x^q h^a} \lesssim \|\vec{f}\|_{L_{x}^2 h^a(\mathbb{R}^2\times \mathbb{Z})}, \label{eq-y2.1} \\
\|\int_0^t e^{i(t-s)\Delta} \vec{F}(s,x)\,\mathrm{d}s\|_{L_t^p L_x^q h^a} \lesssim \|\vec{F}\|_{L_t^{p'} L_x^{q'} h^a}, \label{eq-y2.1'}
\end{align}
where $a\in \{+1,0\}$, $(p,q)$ is an admissible pair, that is $\frac2p + \frac2q = 1$ and $2 < p \le \infty$.
\end{proposition}
\begin{proof}
The conclusion follows from the Minkowski inequality and usual Strichartz estimate.
\end{proof}
\begin{proposition}[nonlinear estimate]\label{prop3.1'} Let $\mathbf{F}(\vec{u}):=\{F_j(\vec{u})\}_{j\in\mathbb{Z}}:=
\{\sum\limits_{\mathcal{R}(j)}u_{j_1}\bar{u}_{j_2}u_{j_3}\}_{j\in\mathbb{Z}}=:\sum^{\rightarrow}\limits_{\mathcal{R}(j)}u_{j_1}\bar{u}_{j_2}u_{j_3}$, We have the following nonlinear estimate,
\begin{equation}\label{eq-y2.6}
\left\|\mathbf{F}(\vec{u})\right\|_{h^a}\lesssim\left\|\vec{u}\right\|^3_{h^a},
\end{equation}
where $a\in \{+1,0\}$ and recalling $\left\|\vec{u}\right\|_{h^a}:=\left(\sum\limits_{j\in\mathbb{Z}} \langle j\rangle^{2a}\left|u_{j}\right|^2\right)^\frac12.$
As a by-product, if $\vec{u} = \{u_j\}_{j\in \mathbb{Z}}$ is a solution to the resonant system \eqref{eq-y1}, then
\begin{equation}\label{eq-y3.1}
\|\vec{u}\|_{L_t^p L_x^q h^a(I\times \mathbb{R}^2\times \mathbb{Z})}\lesssim \|\vec{u}_{0}\|_{L_x^2 h^a(\mathbb{R}^2\times \mathbb{Z})}+\|\vec{u}\|^3_{L_t^{\bar{p}} L_x^{\bar{q}} h^a(I\times \mathbb{R}^2\times \mathbb{Z})}.
\end{equation}
where $(p,q)$, $(\bar{p},\bar{q})$  are admissible pairs and $2 < p,\bar{p} \le \infty$.
\end{proposition}
\begin{proof}
We only consider the case $a=1$ here. When $a=0$, by \eqref{1.1a} and the fact $l^2\subset l^{\infty}$ and Cauchy-Schwarz, \eqref{eq-y2.6} can easily be proven,
\begin{align*}
&\left\|\mathbf{F}(\vec{u})\right\|_{h^1}\\
=&\left(\sum\limits_{j\in\mathbb{Z}} \langle j\rangle^2\left|\sum\limits_{\mathcal{R}(j)} u_{j_1} \bar{u}_{j_2} u_{j_3}\right|^2
\right)^\frac{1}{2}\\
=&\left(\sum\limits_{j\in\mathbb{Z}} \langle j\rangle^2\left|\sum\limits_{\mathcal{R}(j)}\left(\langle j_{1}\rangle u_{j_1} \langle j_{2}\rangle\bar{u}_{j_2}\langle j_{3}\rangle u_{j_3}\right)\left(\langle j_{1}\rangle^{-1}\langle j_{2}\rangle^{-1}\langle j_{3}\rangle^{-1}\right)\right|^2 \right)^\frac12\\
\le & C\left(\sum\limits_{j\in\mathbb{Z}} \langle j\rangle^2\left(\sum\limits_{\mathcal{R}(j)}\langle j_{1}\rangle^{2} |u_{j_1}|^{2} \langle j_{2}\rangle^{2}|{u_{j_2}}|^{2}\langle j_{3}\rangle^{2} |u_{j_3}|^{2}\sum\limits_{\mathcal{R}(j)}\langle j_{1}\rangle^{-2}\langle j_{2}\rangle^{-2}\langle j_{3}\rangle^{-2}\right) \right)^\frac12\\
\le & C\left(\sum\limits_{j\in\mathbb{Z}}\left(\sum\limits_{\mathcal{R}(j)}\langle j_{1}\rangle^{2} |u_{j_1}|^{2} \langle j_{2}\rangle^{2}|{u_{j_2}}|^{2}\langle j_{3}\rangle^{2} |u_{j_3}|^{2}\sup\limits_{j\in \mathbb{Z}}\left(\langle j\rangle^2\sum\limits_{\mathcal{R}(j)}\langle j_{1}\rangle^{-2}\langle j_{2}\rangle^{-2}\langle j_{3}\rangle^{-2}\right)\right) \right)^\frac12\\
\le &\|\vec{u}\|^3_{h^1},
\end{align*}
the last inequality comes from the Lemma A.2 in \cite{HP}, $$\sup\limits_{j\in \mathbb{Z}}\left\{\langle j\rangle^2\sum\limits_{\mathcal{R}(j)}\langle j_{1}\rangle^{-2}\langle j_{2}\rangle^{-2}\langle j_{3}\rangle^{-2}\right\}\lesssim 1,$$
\end{proof}

Following the standard arguments, we can use Proposition \ref{pr-y3.1} and Proposition \ref{prop3.1'} to get the local wellposedness, small data scattering and stability theory for the resonant system \eqref{eq-y1}.
\begin{theorem}\label{th-y3.3} The resonant system \eqref{eq-y1} has the following properties: \\
For $a\in \{+1,0\}$,
\begin{enumerate}
  \item (Local wellposedness) Suppose that $\|\vec{u}_{0}\|_{L^2 h^1} \le E$,  then the resonant system \eqref{eq-y1} has a unique strong solution $\vec{u}\in C_t^0\left((-T,T);L^{2}h^{1}\right)\bigcap L_{t}^{4}L_{x}^{4}h^{a}\left((-T,T)\times \mathbb{R}^2\times \mathbb{Z}\right)$  for some $T > 0$, satisfying $\vec{u}(0)=\vec{u}_{0}$;
  \item (Small data scattering) There exists sufficiently small $\delta > 0$, if $\|\vec{u}_{0}\|_{L^2 h^1} \le \delta$,  then \eqref{eq-y1} has an unique global solution
$\vec{u}\in L_{t}^{\infty}L_{x}^{2}h^{1}\left(\mathbb{R}\times \mathbb{R}^2\times \mathbb{Z}\right)\bigcap L_{t}^{4}L_{x}^{4}h^{a}\left(\mathbb{R}\times \mathbb{R}^2\times \mathbb{Z}\right)$ with initial data $\vec{u}(0)=\vec{u}_{0}$, Moreover, there exist $\vec{u}^{\pm} \in L_{x}^{2}h^{1}(\mathbb{R}^2 \times \mathbb{Z})$ such that
\begin{equation} \label{eq1.3}
\|\vec{u}(t)- e^{it\Delta} \vec{u}^{\pm}\|_{L_{x}^{2}h^{a}} \to 0, \ \text{ as } t\to \pm \infty.
\end{equation}
\end{enumerate}
\end{theorem}
\begin{remark}\label{remark2.5}
Actually whenever $a=1$ or $a=0$, \eqref{eq1.3} implies $\|\vec{u}(t)- e^{it\Delta} \vec{u}^{\pm}\|_{L_{x}^{2}h^{1}} \to 0$,  as  $t\to \pm \infty$. Indeed, \eqref{eq1.3} is equivalent to $\|\vec{u}\|_{L_{t}^{4}L_{x}^{4}h^{a}(\mathbb{R}\times \mathbb{R}^2\times \mathbb{Z})}<\infty$. Then we can divide the time interval into finite segments $\mathbb{R}=\cup_{j=1}^M I_j$, such that $\|\vec{u}\|_{L_{t}^{4}L_{x}^{4}h^{a}(I_j\times \mathbb{R}^2\times \mathbb{Z})}<\delta$. Then by Proposition \ref{pr-y3.1}, \eqref{1.1a} and the fact $l^2\subset l^{\infty}$ and Cauchy-Schwarz,
\begin{equation*}
\begin{split}
\|\vec{u}\|_{L_{t}^{4}L_{x}^{4}h^{1}(I_j\times \mathbb{R}^2\times \mathbb{Z})}&\lesssim \|\vec{u}_0\|_{L_x^2 h^1} +\left\|\mathbf{F}(\vec{u})\right\|_{L_{t}^{4/3}L_{x}^{4/3}h^{1}(I_j\times \mathbb{R}^2\times \mathbb{Z})}\\
&\lesssim \|\vec{u}_0\|_{L_x^2 h^1} + \|\vec{u}\|^2_{L_{t}^{4}L_{x}^{4}l^2(I_j\times \mathbb{R}^2\times \mathbb{Z})}\|\vec{u}\|_{L_{t}^{4}L_{x}^{4}h^{1}(I_j\times \mathbb{R}^2\times \mathbb{Z})}\\
&\lesssim \|\vec{u}_0\|_{L_x^2 h^1} + \delta^2 \|\vec{u}\|_{L_{t}^{4}L_{x}^{4}h^{1}(I_j\times \mathbb{R}^2\times \mathbb{Z})}.
\end{split}
\end{equation*}
This implies $\|\vec{u}\|_{L_{t}^{4}L_{x}^{4}h^{1}(I_j\times \mathbb{R}^2\times \mathbb{Z})}\lesssim \|\vec{u}_0\|_{L_x^2 h^1}$, so $\|\vec{u}\|_{L_{t}^{4}L_{x}^{4}h^{1}(\mathbb{R}\times \mathbb{R}^2\times \mathbb{Z})}<\infty$, this in turn shows $\|\vec{u}(t)- e^{it\Delta} \vec{u}^{\pm}\|_{L_{x}^{2}h^{1}} \to 0$,  as  $t\to \pm \infty$.
\end{remark}
\begin{theorem}[Stability]\label{th-y3.6}
 For $a\in \{+1,0\}$, let I be a compact interval and $\vec{e}=\{e_{j}\}_{j\in \mathbb{Z}}$, $e_{j}=i\partial_t u_j + \Delta u_j - \sum\limits_{\mathcal{R}(j)} u_{j_1} \bar{u}_{j_2} u_{j_3}$, assume $\|\vec{u}\|_{L_{t}^{4}L_{x}^{4}h^{a}\left(I\times \mathbb{R}^2\times \mathbb{Z}\right)}\leq A$ for some $A>0$,  then for $\forall\epsilon>0,\exists \delta>0$,  such that if $\|\vec{e}\|_{L_{t}^{\frac{4}{3}}L_{x}^{\frac{4}{3}}h^{a}\left(I\times \mathbb{R}^2\times \mathbb{Z}\right)}\leq\delta,\|\vec{u}(t_{0})-\vec{v}_{0}\|_{L^{2}h^{1}}\leq\delta,$
then the resonant system \eqref{eq-y1} has a solution $\vec{v}\in L_{t}^{\infty}L_{x}^{2}h^{1}(I)\cap L_{t}^{4}L_{x}^{4}h^{a}(I)$ with initial data $\vec{v}(t_{0})=\vec{v}_0$, Moreover,
$$\|\vec{u}-\vec{v}\|_{L_{t}^{4}L_{x}^{4}h^{a}\left(I\times \mathbb{R}^2\times \mathbb{Z}\right)} +\|\vec{u}-\vec{v}\|_{L_{t}^{\infty}L_{x}^{2}h^{1}\left(I\times \mathbb{R}^2\times \mathbb{Z}\right)}\leq\epsilon.$$
\end{theorem}
We also have the persistence of regularity:
\begin{theorem}[Persistence of regularity]\label{co6.533}
Suppose $\vec{u}_0 \in L_x^2 h^1$ and $\vec{u} \in C_t^0 L_x^2 h^1(\mathbb{R} \times \mathbb{R}^2 \times \mathbb{Z})$ is the solution of \eqref{eq-y1}.
Suppose also that $\vec{v}_0 \in H^4_x h^5$ satisfies
\begin{align*}
\|\vec{u}_0 - \vec{v}_0 \|_{L_x^2 h^1} \lesssim \epsilon,
\end{align*}
and that $\vec{v}$ is the solution to \eqref{eq-y1} with initial data $\vec{v}(0) = \vec{v}_0$. Then it holds that
\begin{align*}
\|\vec{v}\|_{L_t^\infty H_x^4 h^5} + \bigg\| \Big( \sum_{p\in \mathbb{Z} } \langle p \rangle^{10}  \big|  (\langle \nabla \rangle^4 v_p)(t,x)\big|^2 \Big)^\frac12   \bigg\|_{L_{x,t}^4(\mathbb{R} \times \mathbb{R}^2)} & \lesssim \|\vec{v}_0\|_{H^4_x h^5},\\
\|\vec{u} - \vec{v} \|_{L_t^\infty L_x^2 h^1 \left(\mathbb{R}\times \mathbb{R}^2\times \mathbb{Z}\right)}+ \|\vec{u}-\vec{v}\|_{L_{t}^{4}L_{x}^{4}h^{1}\left(\mathbb{R}\times \mathbb{R}^2\times \mathbb{Z}\right)} & \lesssim  \epsilon,
\end{align*}
and there exists $\vec{w}^\pm \in H^4_x h^5$ such that
\begin{align*}
\bigg\| \Big(\sum_{p\in \mathbb{Z}^2} \langle p\rangle^2 \big|(v_p(t)- e^{it\Delta_{\mathbb{R}^2}} w_p^\pm)(x)\big|^2\Big)^\frac12 \bigg\|_{L_x^2(\mathbb{R}^2)} \to 0, as\,\, t\to \pm\infty.
\end{align*}

\end{theorem}

\section{Reduction to almost periodic solutions and the existence of critical element}

We let $G$ be the  group  generated by phase rotations, Galilean transforms, translations and dilations. We let $G\backslash L^{2}h^{1}(\mathbb{R}^2\times\mathbb{Z})$ be the modulo space of $G$-orbits $G\vec{f}:=\{g\vec{f}:g\in G\}$ endowed with the usual quotient topology. Define
$$(T_{g_{\theta,\xi_0,x_0,\lambda}}\vec{u})(t,x):=\frac{1}{\lambda}e^{i\theta}e^{ix\cdot\xi_0}e^{-it|\xi_0|^2}
\vec{u}\left(\frac{t}{\lambda^2},\frac{x-x_0-2\xi_0t}{\lambda}\right),$$
 then the map $g\mapsto T_g$ is a group action of $G$.
\begin{proposition}[Linear profile decomposition in $L_x^2 h^1(\mathbb{R}^2 \times \mathbb{Z})$]  \label{pr6.640}
Let $\{\vec{u}_{n}\}$ be a bounded sequence in $L_x^2 h^1(\mathbb{R}^2 \times \mathbb{Z})$. Then (after passing to a subsequence if necessary) there
exists $J\in  \{0,1, \cdots \} \cup \{\infty\}$, functions $\{\vec{\phi}^{j}\}_{j=1}^{J} \subseteq L_x^2 h^1$, group elements $\{g_n^j\}_{j=1}^{J} \subseteq G$,
and times $\{t_n^j\}_{j=1}^{J} \subseteq \mathbb{R}$ so that defining $\vec{w}_{n}^J$ by
\begin{align*}
\vec{u}_n(x) = &  \sum_{j=1}^J g_n^j e^{it_n^j \Delta_{\mathbb{R}^2}} \vec{\phi}^j + \vec{w}_n^J(x) \\
         := & \sum_{j=1}^J \frac1{\lambda_n^j} e^{ix\xi_n^j} (e^{it_n^j \Delta_{\mathbb{R}^2} } \vec{\phi}^j)\left(\frac{x-x_n^j}{\lambda_n^j} \right) + \vec{w}_n^J(x ),
\end{align*}
we have the following properties:
\begin{align*}
\limsup_{n\to \infty} \|e^{it\Delta_{\mathbb{R}^2}} \vec{w}_n^J\|_{L_{t,x}^4 l^4(\mathbb{R} \times \mathbb{R}^2 \times \mathbb{Z} )} \to 0, \ \text{ as } J\to \infty, \\
e^{-it_n^j \Delta_{\mathbb{R}^2}} (g_n^j)^{-1} \vec{w}_n^J \rightharpoonup 0  \text{ in } L_x^2 h^1,  \text{ as  } n\to \infty,\text{ for each }  j\le J,\\
\sup_{J} \lim_{n\to \infty} \left(\|\vec{u}_n\|_{L_x^2 h^1}^2 - \sum_{j=1}^J \|\vec{\phi}^j\|_{L_x^2 h^1}^2 - \|\vec{w}_n^J\|_{L_x^2 h^1}^2\right) = 0,
\end{align*}
and lastly, for $j\ne j'$, and $n\to \infty$,
\begin{align*}
\frac{\lambda_n^j}{\lambda_n^{j'}} + \frac{\lambda_n^{j'}}{\lambda_n^j} + \lambda_n^j \lambda_n^{j'} |\xi_n^j - \xi_n^{j'}|^2 + \frac{|x_n^j-x_n^{j'}|^2 } {\lambda_n^j \lambda_n^{j'}}  + \frac{|(\lambda_n^j)^2 t_n^j -(\lambda_n^{j'})^2 t_n^{j'}|}{\lambda_n^j \lambda_n^{j'}} \to \infty.
\end{align*}%
\end{proposition}
\begin{proof}the proof is similar to Proposition 3.1 in \cite{CHENG}, so we won't reiterate them here.
\end{proof}

\begin{remark}
By using interpolation, the H\"older inequality and Proposition \ref{pr6.640}, for $0<\epsilon_0\leq 1$,  we have
\begin{align*}
  \limsup_{n\to \infty} \|e^{it\Delta_{\mathbb{R}^2}} \vec{w}_n^J\|_{L_{t,x}^4 h^{1-\epsilon_0}}
\lesssim &  \limsup_{n\to \infty} \|e^{it\Delta_{\mathbb{R}^2}} \vec{w}_n^J\|_{L_{t,x}^4 h^{1 }}^{1-\epsilon_0} \|e^{it\Delta_{\mathbb{R}^2}} \vec{w}_n^J\|_{L_{t,x}^4 l^2}^{\epsilon_0} \\
\lesssim  & \limsup_{n\to \infty} \|\vec{w}_n^J\|_{L_x^2 h^1}^{1-\epsilon_0} \|e^{it\Delta_{\mathbb{R}^2}} \vec{w}_n^J\|_{L_{t,x}^4 l^1} ^{\frac{ \epsilon_0}{3}}  \|e^{it\Delta_{\mathbb{R}^2}} \vec{w}_n^J\|_{L_{t,x}^4 l^4}^{\frac{2\epsilon_0}{3}} \\
\lesssim & \limsup_{n\to \infty} \|\vec{w}_n^J\|_{L_x^2 h^1}^{1-\epsilon_0} \|e^{it\Delta_{\mathbb{R}^2}} \vec{w}_n^J\|_{L_{t,x}^4 h^1} ^{\frac{ \epsilon_0}{3}}  \|e^{it\Delta_{\mathbb{R}^2}} \vec{w}_n^J\|_{L_{t,x}^4 l^4}^{\frac{2\epsilon_0}{3}}\\
\lesssim & \limsup_{n\to \infty} \|\vec{w}_n^J\|_{L_x^2 h^1}^{1-\frac23 \epsilon_0}\|e^{it\Delta_{\mathbb{R}^2}} \vec{w}_n^J\|_{L_{t,x}^4 l^4}^{\frac{2\epsilon_0}{3}}
  \to 0, \text{ as }  J\to \infty.
\end{align*}

\end{remark}
To prove resonant system \eqref{eq-y1} is globally well-posed and scattering it suffices to prove that if $\vec{u}$ is a solution to \eqref{eq-y1}, then
$$\|\vec{u}\|_{L_{t,x}^{4}h^1(\mathbb{R}\times\mathbb{R}^{2}\times\mathbb{Z})}<\infty,$$ for all $\vec{u}_{0}\in L^2h^1(\mathbb{R}^2\times\mathbb{Z}).$
Actually from Remark \ref{remark2.5}, we just need to show $\|\vec{u}\|_{L_{t,x}^{4}l^2(\mathbb{R}\times\mathbb{R}^{2}\times\mathbb{Z})}<\infty$ for all $\vec{u}_{0}\in L^2h^1(\mathbb{R}^2\times\mathbb{Z}).$

For $\vec{u}$ solving \eqref{eq-y1} with maximal lifespan interval $I$, we define the function
$$A(m)=\sup\{\|\vec{u}\|_{L_{t,x}^{4}l^2(I\times\mathbb{R}^{2}\times\mathbb{Z})}: \|\vec{u}(0)\|_{L^2 h^1(\mathbb{R}^2 \times\mathbb{Z})}\leq m\},$$ and $$m_{0}=\sup\{m: A(m')<+\infty, \forall m'<m\}.$$
If we can prove $m_{0}=+\infty$, then global wellposedness and scattering are estabilished.

\begin{theorem}[Reduction to almost periodic solutions]\label{th-y4.3}
Assume $m_{0}<+\infty$. Then there exits a solution (calling critical element) $\vec{u}\in C_{t}^{0}L_{x}^{2}h^{1}\left(I\times \mathbb{R}^2\times \mathbb{Z}\right)\bigcap L_{t}^{4}L_{x}^{4}l^2\left(I\times \mathbb{R}^2\times \mathbb{Z}\right)$ to \eqref{eq-y1} with $I$ the maximal lifespan interval such that
\begin{enumerate}
    \item $M(\vec{u})=m_{0}.$
    \item $\vec{u}$ blows up at both directions in time, i.e. $\|\vec{u}\|_{L_{t,x}^{4}l^2(I\cap(-\infty,t_{0}))}=
        \|\vec{u}\|_{L_{t,x}^{4}l^2(I\cap(\widetilde{t}_{0},+\infty))}=+\infty$, for some $t_{0},\widetilde{t}_{0}\in I.$
    \item $\vec{u}$ is an almost periodic solution modulo $G.$
\end{enumerate}
where $\vec{u}$ is called an almost periodic solution modulo $G$ if the quotiented orbit $\{G\vec{u}:t\in I\}$ is a precompact subset of $G\backslash L^{2}h^{1}(\mathbb{R}^2\times\mathbb{Z}).$
\end{theorem}

In order to prove this theorem we need the following proposition whose proof will be postponed.

\begin{proposition}\label{pr-y4.4}
Let $u_{n,p}(t,x)$ be defined on time interval $I_{n}.$
Assume $m_{0}<+\infty$,  $\vec{u}_{n}:=\left\{u_{n,p}\right\}_{p\in\mathbb{Z}}, n=1,2,\cdot\cdot\cdot$ is a sequence of solutions to \eqref{eq-y1} satisfying $$\limsup\limits_{n\rightarrow\infty}M\left(\vec{u}_{n}\right)=m_{0},$$ $$\lim\limits_{n\rightarrow\infty}\left\|\vec{u}_{n}\right\|_{L_{t,x}^{4}l^2(I_{n}\cap(-\infty,t_{n}))}=
\lim\limits_{n\rightarrow\infty}\left\|\vec{u}_{n}\right\|_{L_{t,x}^{4}l^2(I_{n}\cap(t_{n},+\infty))}=+\infty, \,\,for\,\, some \,\,t_{n}\in I_{n}.$$
Then $G\vec{u}_{n}(t_{n})$ converges (up to subsequence) in $G\backslash L^{2}h^{1}(\mathbb{R}^2\times\mathbb{Z}).$
\end{proposition}
\begin{proof}[Proposition \ref{pr-y4.4}$\Rightarrow$Theorem \ref{th-y4.3}]
By the definition of $m_{0}$,  if $m_{0}<+\infty$, then we can find a sequence  $\vec{u}_{n}:=\left\{u_{n,j}\right\}_{j\in\mathbb{Z}}, n=1,2,\cdot\cdot\cdot$, where $u_{n,j}(t,x)$ is defined on the maximal lifespan time interval $I_{n}$ such that  $$\|\vec{u}^{(n)}\|_{L_{t,x}^{4}l^2(I_{n}\times\mathbb{R}^{2}\times\mathbb{Z})}\nearrow \infty,$$ $$\limsup\limits_{n\rightarrow\infty}M\left(\vec{u}^{(n)}\right)=m_{0},$$ then $\exists t_n \in I_n$ such that
\begin{equation}\label{eq-y4.1}
\lim\limits_{n\rightarrow\infty}\left\|\vec{u}_{n}\right\|_{L_{t,x}^{4}l^2(I_{n}\cap(-\infty,t_{n}))}
=\lim\limits_{n\rightarrow\infty}\left\|\vec{u}_{n}\right\|_{L_{t,x}^{4}l^2(I_{n}\cap(t_{n},+\infty))}=+\infty.
\end{equation}
By time translation invariance we may take $t_n \equiv0$. Proposition \ref{pr-y4.4} guarantees that there is a subsequence of $G\vec{u}_{n}(0)$ (still denoted as $G\vec{u}_{n}(0)$) converging  in $G\backslash L^{2}h^{1}$, i.e. $\exists \vec{u}_{0}\in L^{2}h^{1},g_{n}\in G$,  such that  $\lim\limits_{n\rightarrow\infty}g_{n}\vec{u}_{n}(0)=g^{(0)}\vec{u}_{0}$ in $L^{2}h^{1}$. We might as well suppose $g_{n}\equiv \mathrm{I}d$ by the group's action, then $\lim\limits_{n\rightarrow\infty}\vec{u}_{n}(0)=\vec{u}_{0}$ in $L^{2}h^{1}$, this implies $M\left(\vec{u}_{0}\right)\leq m_{0}.$

Let $\vec{u}(t,x)$ be the solution to \eqref{eq-y1} on the maximal lifespan interval $I$ with initial data $\vec{u}(0,x)=\vec{u}_{0}$, then we claim that $\vec{u}$ blows up at both directions in time. Indeed, if $\vec{u}$ does not blow up forward in time (say), then $[0,+\infty)\subset I$ and $\|\vec{u}\|_{L_{t,x}^{4}l^2(0,+\infty)}<+\infty$, by the stability theorem \ref{th-y3.6}, for sufficiently large $n$,  we have $[0,+\infty)\subset I_n$ and $\limsup\limits_{n\rightarrow\infty}\left\|\vec{u}_{n}\right\|_{L_{t,x}^{4}l^2(I_{n}\cap(t_{n},+\infty))}<\infty$,  contradicting \eqref{eq-y4.1}. Similarly we get $\vec{u}$ blows up backward in time. By the construction of $m_0$ this forces $M\left(\vec{u}_{0}\right)\geq m_{0}$ and hence $M\left(\vec{u}_{0}\right)= m_{0}$ exactly.

It remains to show that  $\vec{u}$  is an almost periodic solution modulo $G$. Consider an arbitrary sequence $G\vec{u}(t_{n}')$ in $\{G\vec{u}(t):t\in I\}$. Since $\vec{u}$ blows up at both directions in time, but is locally in $L_{t,x}^{4}l^2$, we have
$$\left\|\vec{u}\right\|_{L_{t,x}^{4}l^2(I\cap(-\infty,t_{n}'))}
=\left\|\vec{u}\right\|_{L_{t,x}^{4}l^2(I\cap(t_{n}',+\infty))}=+\infty.$$
Applying Proposition \ref{pr-y4.4} once again we see that $G\vec{u}(t_{n}')$ does have a convergent sequence in $G\backslash L^{2}h^{1}(\mathbb{R}^2\times\mathbb{Z})$. Thus, the orbit $\{G\vec{u}(t):t\in I\}$ is precompact in $G\backslash L^{2}h^{1}(\mathbb{R}^2\times\mathbb{Z})$ as desired.
\end{proof}

\begin{proof}[Proof of Proposition \ref{pr-y4.4}]
 By translating $\vec{u}_{n}$ (and $I_n$ ) in time, we may take
$t_n = 0$ for all $n$; thus,
$$\lim\limits_{n\rightarrow\infty}\left\|\vec{u}_{n}\right\|_{L_{t,x}^{4}l^2(I\cap(-\infty,0))}=
\lim\limits_{n\rightarrow\infty}\left\|\vec{u}^{(n)}\right\|_{L_{t,x}^{4}l^2(I\cap(0,+\infty))}=+\infty.$$
$\limsup\limits_{n\rightarrow\infty}M\left(\vec{u}_{n}\right)=m_{0}$ implies that $\{\vec{u}_{n}\}$ is bounded  (passing to a subsequence if necessary) in $L_x^2h^1$.  Therefore, by Proposition \ref{pr6.640}, we have
$$\vec{u}_n(0,x) =   \sum_{j=1}^J g_n^j e^{it_n^j \Delta_{\mathbb{R}^2}} \vec{\phi}^j + \vec{w}_n^J(x),$$
where $t_n^{j}\in \mathbb{R}$,  $g_n^{j}\in G$. By extracting subsequence and time translation, we can assume $t_n^{j}\rightarrow t^{j}\in \{-\infty,0,+\infty\}$ as $n\rightarrow \infty$. then following the argument in section 5 of \cite{TVZ1}, we can show $\sup\limits_{j}M(\vec{\phi}^{j})=m_0$. where $\vec{\phi}^{j}=\{\phi_p^{j}\}_{p\in \mathbb{Z}}$, this implies $J=1$, $M(\vec{\phi}^{1})=m_0$ and
$$\vec{u}_{n}(0,x)= g_n^{1} e^{it_n^{1} \triangle} \vec{\phi}^{1} +\vec{w}_n^{1}=:g_{n} e^{it_{n} \triangle} \vec{\phi} +\vec{w}_{n}.$$
Furthermore, we obtain $\lim_{n\rightarrow\infty}t_{n}=0$. This shows Proposition \ref{pr-y4.4} is true.
\end{proof}

Next we phrase the property of almost periodicity modulo $G$ in a more ``quantitative'' sense. At first we give the description of the precompactness of Hilbert space $L^2h^1$.
\begin{proposition}
$\mathfrak{X}\subset L^2h^1$ is bounded, then $\mathfrak{X}$ is precompact iff for arbitrary $\epsilon>0$, there exist $ K=K(\ep)>0,R=R(\ep)>0$,  such that $\sum\limits_{|j|>K}\langle j\rangle^{2}\|u_{j}\|_{L^2}^{2}<\epsilon$, besides,
 \begin{align*}
   \sum\limits_{|j|=0}^K \langle j\rangle^{2}\int_{|x|\geq R}|u_{j}(x)|^{2}\mathrm{d}x&<\epsilon, \\
   \sum\limits_{|j|=0}^K \langle j\rangle^{2}\int_{|\xi|\geq R}|\hat{u}_{j}(\xi)|^{2}\mathrm{d}\xi&<\epsilon.
 \end{align*}
 for all $\vec{u}=\{u_j\}_{j\in\mathbb{Z}}\in \mathfrak{X}.$
\end{proposition}
\begin{proof}We first prove the sufficiency. Suppose for any
 $\epsilon>0$, there exist $ K(\epsilon)>0,R(\epsilon)>0$,  such that $\sum\limits_{|j|>K(\epsilon)}\langle j\rangle^{2}\|u_{j}\|_{L^2}^{2}<\frac{\epsilon}{2}$ and
 \begin{align*}
   \sum\limits_{|j|=0}^{K(\epsilon)} \langle j\rangle^{2}\int_{|x|\geq R(\epsilon)}|u_{j}(x)|^{2}\mathrm{d}x&<\frac{\epsilon}{2}, \\
   \sum\limits_{|j|=0}^{K(\epsilon)} \langle j\rangle^{2}\int_{|\xi|\geq R(\epsilon)}|\hat{u}_{j}(\xi)|^{2}\mathrm{d}\xi&<\frac{\epsilon}{2}
 \end{align*}
 for all $\vec{u}=\{u_j\}_{j\in\mathbb{Z}}\in \mathfrak{X}$, then for arbitrary $n>0,$
 \begin{align*}
   &\sum\limits_{|j|=0}^{n}\left(\int_{|x|\geq R(\epsilon)}|u_{j}(x)|^{2}\mathrm{d}x+\int_{|\xi|\geq R(\epsilon)}|\hat{u}_{j}(\xi)|^{2}\mathrm{d}\xi\right) \\
   =&\sum\limits_{|j|=0}^{K(\epsilon)}\left(\int_{|x|\geq R(\epsilon)}|u_{j}(x)|^{2}\mathrm{d}x+\int_{|\xi|\geq R(\epsilon)}|\hat{u}_{j}(\xi)|^{2}\mathrm{d}\xi\right)\\
   &\quad + \sum\limits_{|j|=K(\epsilon)}^{n}\left(\int_{|x|\geq R(\epsilon)}|u_{j}(x)|^{2}\mathrm{d}x+\int_{|\xi|\geq R(\epsilon)}|\hat{u}_{j}(\xi)|^{2}\mathrm{d}\xi\right)\\
   \leq &\frac{\epsilon}{2}+2\sum\limits_{j\geq K(\epsilon)}\|u_j\|_{L^2}^{2}\\
   \leq&2\epsilon.
 \end{align*}
 Therefore, by the description of the precompactness of $\bigoplus\limits_{j=1}^{n}L^2$,  $\{(u_1,u_2,\cdots,u_n)\in\bigoplus\limits_{j=1}^{n}L^2 :\{u_j\}_{j\in\mathbb{Z}}\in \mathfrak{X}\}$ is precompact in $\bigoplus\limits_{j=1}^{n}L^2$. Thus  $\{(u_1,u_2,\cdots,u_{K(\epsilon)})\in\bigoplus\limits_{j=1}^{K(\epsilon)}L^2 :\{u_j\}_{j\in\mathbb{Z}}\in \mathfrak{X}\}$ is precompact in $\bigoplus\limits_{j=1}^{K(\epsilon)}L^2$. Therefore, there exist finite $n_1,n_2,\cdots,n_l$,  such that the $\frac{\epsilon}{(1+|K(\epsilon)|^2)^2}$ neighborhood of the finite set $\left\{(u_{1}^{n_1},u_2^{n_1}\cdots,u_{K(\epsilon)}^{n_1}),(u_{1}^{n_2},u_2^{n_2}\cdots,u_{K(\epsilon)}^{n_2}),\cdots,(u_{1}^{n_l},u_2^{n_l}\cdots,u_{K(\epsilon)}^{n_l})\right\}$ covers $\{(u_1,u_2,\cdots,u_{K(\epsilon)})\in\bigoplus\limits_{j=1}^{K(\epsilon)}L^2 :\{u_j\}_{j\in\mathbb{Z}}\in \mathfrak{X}\}.$

 We claim that the $3\epsilon$ neighborhood of $\left\{\{u^{n_1}_j\}_{j\in\mathbb{Z}},\{u^{n_2}_j\}_{j\in\mathbb{Z}},\cdots,\{u^{n_l}_j\}_{j\in\mathbb{Z}}\right\}$ covers $\mathfrak{X}$. In fact, for $\forall\vec{u}=\{u_j\}_{j\in\mathbb{Z}}\in \mathfrak{X}$, its first $K(\epsilon)$ components must be located in the $\frac{\epsilon}{(1+|K(\epsilon)|^2)^2}$ neighbourhood of one point $(u_{1}^{n_i},u_2^{n_i}\cdots,u_{K(\epsilon)}^{n_i})\in \left\{(u_{1}^{n_1},u_2^{n_1}\cdots,u_{K(\epsilon)}^{n_1}),(u_{1}^{n_2},u_2^{n_2}\cdots,u_{K(\epsilon)}^{n_2}),\cdots,(u_{1}^{n_l},u_2^{n_l}\cdots,u_{K(\epsilon)}^{n_l})\right\},$
thus
$$\left\|\vec{u}-\{u^{n_i}_j\}_{j\in\mathbb{Z}}\right\|_{L^2h^1}^{2}\leq \sum\limits_{j=0}^{K(\epsilon)}\langle j\rangle^{2}\|u_j-u_j^{n_i}\|_{L^2}^{2}+\sum\limits_{j\geq K(\epsilon)}\langle j\rangle^{2}\|u_j^{n_i}\|_{L^2}^{2}+\sum\limits_{j\geq K(\epsilon)}\langle j\rangle^{2}\|u_j\|_{L^2}^{2}\leq 3\epsilon,$$
which implies $\mathfrak{X}$ is precompact in $L^2h^1.$

 Now let's prove the necessity. Suppose $\mathfrak{X}\subset L^2h^1$ is precompact, then $P_{K}\mathfrak{X}=\{(u_1,u_2,\cdots,u_K):\{u_j\}_{j\in\mathbb{Z}}\in\mathfrak{X}\}$ is precompact in $\bigoplus\limits_{j=1}^{K}L^2$, where $P_{K}$ is a projection operator and $K$ is to be determined.

Because $\mathfrak{X}\subset L^2h^1$ is precompact, for arbitrary $\epsilon>0$, there exists a finite set $\{\vec{u}^{(1)},\vec{u}^{(2)},\cdots,\vec{u}^{(n_{\epsilon})}\}\subset \mathfrak{X}$ whose $\frac{\epsilon}{2}$ nets cover $\mathfrak{X}$. However, there exsits $ K=K(\epsilon)>0$ such that  $\forall \vec{u}^{(i)},i=1,2,\cdots,n_{\epsilon},$
$$\sum_{j\geq K(\epsilon)}\langle j\rangle^{2}\|u_j^{(i)}\|^2_{L^2}<\frac{\epsilon}{2}.$$
Therefore, for arbitrary $\vec{u}\in\mathfrak{X}$, there exists $\vec{u}^{(i)}\in \{\vec{u}^{(1)},\vec{u}^{(2)},\cdots,\vec{u}^{(n_{\epsilon})}\}$ such that  $\vec{u}$ lies in the $\frac{\epsilon}{2}$ neighbourhood of $\vec{u}^{(i)}$. Hence,
$$\sum_{j\geq K(\epsilon)}\langle j\rangle^{2}\|u_j\|^2_{L^2}\leq\sum_{j\geq K(\epsilon)}\langle j\rangle^{2}\|u_j-u_j^{(i)}\|^2_{L^2}+\sum_{j\geq K(\epsilon)}\langle j\rangle^{2}\|u_j^{(i)}\|^2_{L^2}<\epsilon.$$
Meanwhile, once $K(\epsilon)$ has been determined, $P_{K(\epsilon)}\mathfrak{X}\subset \bigoplus\limits_{j=1}^{K(\epsilon)}L^2$ is also precompact. By the precompactness of $\bigoplus\limits_{j=1}^{K(\epsilon)}L^2$, there exists $R(\epsilon)>0$ such that
 \begin{align*}
   \sum\limits_{|j|=0}^{K(\epsilon)} \langle j\rangle^{2}\int_{|x|\geq R(\epsilon)}|u_{j}(x)|^{2}\mathrm{d}x&<\epsilon, \\
   \sum\limits_{|j|=0}^{K(\epsilon)} \langle j\rangle^{2}\int_{|\xi|\geq R(\epsilon)}|\hat{u}_{j}(\xi)|^{2}\mathrm{d}\xi&<\epsilon.
 \end{align*}
This completes the proof.
\end{proof}
\begin{corollary}\label{co-y4.7}
Suppose $\mathfrak{X}\subset L^2h^1$ is precompact, then for arbitrary $\epsilon>0$, there exists $ R(\epsilon)>0$ such that
$$\sum\limits_{j\in\mathbb{Z}} \langle j\rangle^{2}\left[\int_{|x|\geq R(\epsilon)}|u_{j}(x)|^{2}\mathrm{d}x+\int_{|\xi|\geq R(\epsilon)}|\hat{u}_{j}(\xi)|^{2}\mathrm{d}\xi\right]<\epsilon,$$
for all $\vec{u}=\{u_j\}_{j\in\mathbb{Z}}\in \mathfrak{X}.$
\end{corollary}
\begin{proof}
Since $\mathfrak{X}\subset L^2h^1$ is precompact, by Proposition \ref{pr-y4.4}  for arbitrary $\epsilon>0$, there exist $K>0$ and $R>0$ such that $\sum\limits_{|j|>K}\langle j\rangle^{2}\|u_{j}\|_{L^2}^{2}<\frac{\epsilon}{4}$. In addition,
 \begin{align*}
   \sum\limits_{|j|=0}^K \langle j\rangle^{2}\int_{|x|\geq R}|u_{j}(x)|^{2}\mathrm{d}x&<\frac{\epsilon}{4}, \\
   \sum\limits_{|j|=0}^K \langle j\rangle^{2}\int_{|\xi|\geq R}|\hat{u}_{j}(\xi)|^{2}\mathrm{d}\xi&<\frac{\epsilon}{4}
 \end{align*}
 for all $\vec{u}=\{u_j\}_{j\in\mathbb{Z}}\in \mathfrak{X}.$
Thus, by Plancherel identity,
\begin{align*}
 &\sum\limits_{j\in\mathbb{Z}} \langle j\rangle^{2}\left[\int_{|x|\geq R(\epsilon)}|u_{j}(x)|^{2}\mathrm{d}x+\int_{|\xi|\geq R(\epsilon)}|\hat{u}_{j}(\xi)|^{2}\mathrm{d}\xi\right]\\
 \leq&\sum\limits_{j=1}^{K(\epsilon)} \langle j\rangle^{2}\left[\int_{|x|\geq R(\epsilon)}|u_{j}(x)|^{2}\mathrm{d}x+\int_{|\xi|\geq R(\epsilon)}|\hat{u}_{j}(\xi)|^{2}\mathrm{d}\xi\right]+2\sum\limits_{j\geq K(\epsilon)} \langle j\rangle^{2}\|u_{j}\|_{L^2}^{2}\\
 \leq& \frac{\epsilon}{4}+\frac{\epsilon}{4}+2\times \frac{\epsilon}{4}<\epsilon.
\end{align*}
\end{proof}
As a consequence of Proposition \ref{pr-y4.4}, we phrase the property of almost periodicity modulo $G$ in a ``quantitative'' version.
\begin{proposition}\label{pr-y4.8}
The following statements are equivalent.
\begin{enumerate}
  \item $\vec{u}\in C_{t,loc}^{0}L^2h^1(I\times\mathbb{R}^2\times\mathbb{Z})$ is almost periodic modulo $G.$
  \item $\{G\vec{u}(t):t\in I\}$ is precompact in $G\backslash L^{2}h^{1}(\mathbb{R}^2\times\mathbb{Z}).$
  \item $\exists x(t),\xi(t),N(t)$ such that  $\forall\eta>0$, $\exists K(\eta)>0,R(\eta)>0$ such that  for $t\in I,$
\begin{align}\label{eq-y4.2}
  \sum\limits_{|j|>K(\eta)}\langle j\rangle^{2}\|u_{j}\|_{L^2}^{2}&<\eta,\\ \label{eq-y4.3}
  \sum\limits_{|j|=0}^{K(\eta)} \langle j\rangle^{2}\int_{|x-x(t)|\geq \frac{R(\eta)}{N(t)}}|u_{j}(x)|^{2}\mathrm{d}x&<\eta, \\ \label{eq-y4.4}
  \sum\limits_{|j|=0}^{K(\eta)} \langle j\rangle^{2}\int_{|\xi-\xi(t)|\geq R(\eta)N(t)}|\hat{u}_{j}(\xi)|^{2}\mathrm{d}\xi&<\eta.
\end{align}
\end{enumerate}
\end{proposition}
\begin{corollary}\label{co-y4.9}
$\vec{u}\in C_{t,loc}^{0}L^2h^1(I\times\mathbb{R}^2\times\mathbb{Z})$ is almost periodic modulo $G$. Then there exist $x(t),\xi(t),N(t)$ such that for arbitrary $\eta>0$, there exists $ R(\eta)>0$ such that for $t\in I,$
\begin{equation}\label{eq-y3.6}
\sum\limits_{j\in\mathbb{Z}} \langle j\rangle^{2}\left[\int_{|x-x(t)|\geq \frac{R(\eta)}{N(t)}}|u_{j}(x)|^{2}\mathrm{d}x+\int_{|\xi-\xi(t)|\geq R(\eta)N(t)}|\hat{u}_{j}(\xi)|^{2}\mathrm{d}\xi\right]<\eta.
\end{equation}
\end{corollary}

\begin{theorem}[The estimate of $N(t)$]\label{th-y4.5}The following statements hold:
\begin{enumerate}
  \item For any nonzero almost periodic solution $\vec{u}$ to \eqref{eq-y8.1'} there exists $\delta(\vec{u}) > 0$ such that for any $t_0 \in I$,
      $$\|\vec{u}\|_{L^4_{t,x}l^{2}([t_0,t_0+\frac{\delta}{N(t_0)^2}]\times\mathbb{R}^2)}\sim \|\vec{u}\|_{L^4_{t,x}l^{2}([t_0-\frac{\delta}{N(t_0)^2},t_0]\times\mathbb{R}^2)}\sim1.$$
  \item If $J$ is an interval with $\|\vec{u}\|_{L^4_{t,x}l^{2}(J\times\mathbb{R}^2)}=1$, then for $t_1,t_2\in J$,  $N(t_1)\sim_{m_0} N(t_2)$, and $|\xi(t_1)-\xi(t_2)|\lesssim N(J_k)$, where $N(J_k):=\sup_{t\in J_k}N(t)$.
  \item  Suppose $\vec{u}$ is a minimal mass blowup solution with $N(t) \leq 1$.  Suppose also that $J$
   is some interval partitioned into subintervals $J_k$ with $\|\vec{u}\|_{L^4_{t,x}l^{2}(J_k\times\mathbb{R}^2)}= 1$ on each $J_k$, then $N(J_k)\sim \int_{J_k}N(t)^3 dt\sim \inf_{t\in J_k}N(t)$ and $\sum_{J_k}N(J_k)\sim \int_{J}N(t)^3 dt$.
   \item  If $\vec{u}(t,x)$ is a minimal mass blowup solution on an interval $J$, then
   $$\int_J N(t)^2dt\lesssim \|\vec{u}\|^{4}_{L_{t,x}^{4}l^{2}(J\times \mathbb{R}^2)}\lesssim 1+\int_J N(t)^2dt.$$
\end{enumerate}
\end{theorem}
\begin{proof}
  The proof is similar to that of Lemma 2.12, Lemma 2.13 and Lemma 2.15 in \cite{D} and we omit it.
\end{proof}
\begin{remark}\label{re-y4.10}
By Theorem \ref{th-y4.5}, $|N'(t)|,|\xi'(t)|\lesssim N(t)^3$. We can use this fact to control the movement of $\xi(t)$.
\end{remark}
\section{$U^p_{\triangle}(l^{2})$,  $V^p_{\triangle}(l^{2})$ spaces and bilinear Strichartz estimate }
In this section, we give the definition of $U^p_{\triangle}(l^{2})$,  $V^p_{\triangle}(l^{2})$ spaces and then prove corresponding bilinear Strichartz estimate in such spaces.

\begin{definition} Let $1\leq p <\infty $.  Then $U^p_{\triangle}(l^{2})$ is an atomic space, where atoms are piecewise solutions
to the linear equation
$$\vec{u}=\sum\limits_{k}\chi_{[t_k,t_{k+1}]}(t)e^{it\triangle}\vec{u}_k(x),\,\,
\sum\limits_{k}\|\vec{u}_k(x)\|_{L^2l^{2}}^p=1.$$
We define $\|\cdot\|_{U^p_{\triangle}(l^{2})}$ as
$$\|\vec{u}\|_{U^p_{\triangle}(l^{2})}:=\inf\left\{\sum\limits_{\lambda}|c_{\lambda}|:\vec{u} =\sum_{\lambda}c_{\lambda}\vec{u}^{\lambda},\,\vec{u}^{\lambda} \,are\, U^p_{\triangle}(l^{2}) \,atoms\right\}.$$\\
Let $DU^p_{\triangle}(l^{2})$ be the space
$$DU^p_{\triangle}(l^{2})=\{(i\partial_{t}+\triangle)\vec{u}:\vec{u}\in U^p_{\triangle}(l^{2})\}$$
and the norm is
$$\|(i\partial_{t}+\triangle)\vec{u}\|_{DU^p_{\triangle}(l^{2})} :=\left\|\int_{0}^{t}e^{i(t-s)\triangle}(i\partial_{s}\vec{u}
+\triangle\vec{u})(s)ds\right\|_{U^p_{\triangle}(l^{2})}.$$
We can also define the spaces $U^p_{\triangle}(l^{2})$ and $DU^p_{\triangle}(l^{2})$ in the same way, where we just replace the $U^p_{\triangle}(l^{2})$ atoms with $U^p_{\triangle}(l^{2})$ atoms
$$\vec{u}=\sum\limits_{k}\chi_{[t_k,t_{k+1}]}(t)e^{it\triangle}\vec{u}_k(x),\,\,
\sum\limits_{k}\|\vec{u}_k(x)\|_{L^2l^{2}}^p=1.$$
\end{definition}

\begin{definition}Let $1\leq p <\infty $.  Then $V^p_{\triangle}(l^{2})$ is the space of right continuous functions $\vec{v}\in L_{t}^{\infty}L_{x}^{2}l^{2}$ such that
$$\|\vec{v}\|_{V^p_{\triangle}(l^{2})}^p:=\|\vec{v}\|_{L_{t}^{\infty}L_{x}^{2}l^{2}}^p+\sup_{\{t_k\}\nearrow}\sum_{k}\|
e^{-it_{k+1}\triangle}\vec{v}(t_{k+1})-e^{-it_{k}\triangle}\vec{v}(t_{k})\|_{L^2l^{2}}^{p}.$$ Define
$V^p_{\triangle}(l^{2})$ as the space of right continuous functions $\vec{v}\in L_{t}^{\infty}L_{x}^{2}l^{2}$ such that
$$\|\vec{v}\|_{V^p_{\triangle}(l^{2})}^p:=\|\vec{v}\|_{L_{t}^{\infty}L_{x}^{2}l^{2}}^p +\sup_{\{t_k\}\nearrow}\sum_{k}\|e^{-it_{k+1}\triangle}\vec{v}(t_{k+1}) -e^{-it_{k}\triangle}\vec{v}(t_{k})\|_{L^2l^{2}}^{p},$$
where the supremum is taken over increasing sequences $t_{k}.$
\end{definition}

We collect some useful properties about $U^p_{\triangle}(l^{2})$,  $V^p_{\triangle}(l^{2})$ space below.
\begin{proposition}\label{pr-y5.3}
$U^p_{\triangle}(l^{2})$,  $V^p_{\triangle}(l^{2})$ space has the following properties:
\begin{enumerate}
  \item $U^p_{\triangle}(l^{2})$,  $V^p_{\triangle}(l^{2})$ is a Banach space.
  \item $U^p_{\triangle}(l^{2}) \subset V^p_{\triangle}(l^{2}) \subset U^q_{\triangle}(l^{2})$,  $1<p<q<\infty$.
  \item $(DU^p_{\triangle}(l^{2}))^{*}=V^{p'}_{\triangle}(l^{2}),\frac{1}{p}+\frac{1}{p'}=1.$ and $1<p<\infty$.
  \item These spaces are also closed under truncation in time.$$\chi_{I}:U^p_{\triangle}(l^{2})\rightarrow U^p_{\triangle}(l^{2});\,\,\,\,\chi_{I}:V^p_{\triangle}(l^{2})\rightarrow V^p_{\triangle}(l^{2}).$$
  \item  Suppose $J = I_1 \cup I_2 , I_1 = [a,b], I_2 = [b,c], a\leq b\leq c$, then
  $$\|\vec{u}\|_{U^p_{\triangle}(l^{2};J)}^{p}\leq\|\vec{u}\|_{U^p_{\triangle}(l^{2};I_1)}^{p}+
  \|\vec{u}\|_{U^p_{\triangle}(l^{2};I_2)}^{p},$$
  $$\|\vec{u}\|_{U^p_{\triangle}(l^{2};I_1)}\leq \|\vec{u}\|_{U^p_{\triangle}(l^{2};J)}.$$
  \item $\|\vec{u}\|_{L^p_t L^q_x l^{2}}+\|\vec{u}\|_{L^{\infty}_t L^2_x l^{2}}\lesssim \|\vec{u}\|_{U^p_{\triangle}(l^{2})}$, $(p,q)$ is an admissible pair, $p>2$.  i.e. $\frac{1}{p}+\frac{1}{q}=\frac{1}{2}.$
  \item There is the easy estimate
      $$\|\vec{u}\|_{U^p_{\triangle}(l^{2})}\lesssim\|\vec{u}(0)\|_{L^2l^{2}}+
      \|(i\partial_{t}+\triangle)\vec{u}\|_{DU^p_{\triangle}(l^{2})}.$$
\end{enumerate}
\end{proposition}
\begin{proof}The proof is standard. For completion we give the proof of (3), (6). Others can be found in  in \cite{D}, \cite{HHK} and \cite{KTV}.

First we prove (3), by the Chapter 4 in \cite{KTV},
$$\|\vec{v}\|_{V^{p'}_{\triangle}(l^{2})}= \sup\limits_{\|\vec{u}\|_{U^{p}_{\triangle}(l^{2})}\leq1}|
B(e^{-it\triangle}\vec{u},e^{-it\triangle}\vec{v})|,$$
where
\begin{align*}
  \left|B(e^{-it\triangle}\vec{u},e^{-it\triangle}\vec{v})\right|=&\left|\int_{\mathbb{R}}\langle\partial_t e^{-it\triangle}\vec{u}(t),e^{-it\triangle}\vec{v}(t)\rangle dt\right|\\
  =&\left|-i\int_{\mathbb{R}}\langle e^{-it\triangle}(i\partial_t\vec{u}+\triangle\vec{u})(t),e^{-it\triangle}\vec{v}(t)\rangle dt\right|\\
  =&\left|\int_{\mathbb{R}}\langle\partial_t(e^{-it\triangle}\int_{0}^{t}e^{i(t-s)\triangle}(i\partial_{s}\vec{u}
+\triangle\vec{u})(s)ds),e^{-it\triangle}\vec{v}(t)\rangle dt\right|\\
  =&\left|B\big(e^{-it\triangle}\int_{0}^{t}e^{i(t-s)\triangle}(i\partial_{s}\vec{u}
+\triangle\vec{u})(s)ds,e^{-it\triangle}\vec{v}\big)\right|.
\end{align*}Here $\<\cdot,\cdot\>$ means the paring between $L_x^2l^{2}$ and $L_x^2l^{2}$. Thus,
\begin{align*}
\|\vec{v}\|_{V^{p'}_{\triangle}(l^{2})}
=&\sup\limits_{\|(i\partial_{t}+\triangle)\vec{u}\|_{DU^p_{\triangle}(l^{2})}\leq1} \left|B\left(e^{-it\triangle}\int_{0}^{t}e^{i(t-s)\triangle}(i\partial_{s}\vec{u}
+\triangle\vec{u})(s)ds,e^{-it\triangle}\vec{v}\right)\right|\\
=&\sup\limits_{\|(i\partial_{t}+\triangle)\vec{u}\|_{DU^p_{\triangle}(l^{2})}\leq1}\left|\int_{\mathbb{R}}
\langle(i\partial_t\vec{u}+\triangle\vec{u})(t),\vec{v}(t)\rangle dt\right|\\
=&\sup\limits_{\|(i\partial_{t}+\triangle)\vec{u}\|_{DU^p_{\triangle}(l^{2})}\leq1}\left|\int_{\mathbb{R}} \int_{\mathbb{R}^2}\sum_{j\in\mathbb{Z}}(i\partial_t u_j+\triangle u_j) \bar{v}_j dxdt\right|.
\end{align*}
This means $(DU^p_{\triangle}(l^{2}))^{*}=V^{p'}_{\triangle}(l^{2}).$

Next we prove (6). It suffices to  check individual $U^p_{\triangle}(l^{2})$-atom
$$\vec{u}=\sum\limits_{k}\chi_{[t_k,t_{k+1}]}(t)e^{it\triangle}\vec{u}_k(x),\,\,
\sum\limits_{k}\|\vec{u}_k(x)\|_{L^2l^{2}}^p=1.$$
Without loss of generality, we may assume $\|\vec{u}\|_{U^p_{\triangle}(l^{2})}=1.$
For arbitrary $t\in\mathbb{R}$, there exists $k_{0}$ such that  $t\in[t_{k_0},t_{k_0+1}]$, then $\vec{u}(t,x)=e^{it\triangle}\vec{u}_{k_0}(x),$
\begin{align*}
  \|\vec{u}(t)\|_{L_x^2l^{2}}&=\|e^{it\triangle}\vec{u}_{k_0}\|_{L_x^2l^{2}}
   =\|\vec{u}_{k_0}\|_{L_x^2l^{2}}
   =\left(\|\vec{u}_{k_0}\|^p_{L_x^2l^{2}}\right)^{\frac{1}{p}}
   \leq \left(\sum_k\|\vec{u}_{k}\|^p_{L_x^2l^{2}}\right)^{\frac{1}{p}}
   =1.
\end{align*}
This implies $\|\vec{u}(t)\|_{L_t^{\infty}L_x ^2l^{2}}\lesssim\|\vec{u}\|_{U^p_{\triangle}(l^{2})}$. Furthermore,
\begin{align*}
  \|\vec{u}\|_{L_t^p L_x^q l^{2}}&=\left(\int_{\mathbb{R}}\|\vec{u}\|^p_{L_x^ql^2}dt\right)^{\frac{1}{p}}\\
   &=\left(\sum_{k}\int^{t_{k+1}}_{t_k}\|\vec{u}\|^p_{L_x^ql^2}dt\right)^{\frac{1}{p}}\\
   &=\left(\sum_{k}\int^{t_{k+1}}_{t_k}\|e^{it\triangle}\vec{u}_k\|^p_{L_x^ql^2}dt\right)^{\frac{1}{p}}\\
   &= \left(\sum_k\|e^{it\triangle}\vec{u}_{k}\|^p_{L_t^pL_x^ql^2([t_k,t_{k+1}])}\right)^{\frac{1}{p}}\\
   &\leq \left(\sum_k\|\vec{u}_{k}\|^p_{L_x^2l^{2}}\right)^{\frac{1}{p}}\\
   &=1.
\end{align*}
The last inequality is due to Strichartz estimate. Therefore, $\|\vec{u}\|_{L_t^p L_x^q l^{2}}\lesssim\|\vec{u}\|_{U^p_{\triangle}(l^{2})}.$
\end{proof}

\begin{lemma}\label{le-y4.5}
Suppose $J=\cup ^k_{m=1}J^m$, where $J^m$ are consecutive intervals, $J^m=[a_m,b_m]$, $a_{m+1}=b_m$. Also suppose that $\vec{F}=\{F_j\}_{j\in\mathbb{Z}}\in L ^1_t L^2_x l^{2} (J \times\mathbb{R}^2)$ (however our bound will not depend on $\|\vec{F}\|_{L ^1_t L^2_x l^{2}}$.) Then
for any $t_0 \in J$,
\begin{equation}\label{eq-y4.1'}
\begin{split}
  \big\|\int^t_{t_0} e^{i(t-\tau)\triangle}\vec{F}(\tau)d\tau \big\|_{U^2_{\triangle}(l^{2};J\times \mathbb{R}^2)} &\lesssim \sum^{k}_{m=1} \big\| \int_{J^m} e^{-i\tau\triangle}\vec{F}(\tau)d\tau \big\|_{L^2_x l^{2}} + \big[\sum^{k}_{m=1}(\|\vec{F}\|_{DU^{2}_{\triangle}(l^{2};J^m\times \mathbb{R}^2)})^2 \big]^{1/2}.
\end{split}
\end{equation}Where $$\|\vec{F}\|_{DU^{2}_{\triangle}(l^{2};J^m\times \mathbb{R}^2)}:=\sup_{\|\vec{v}\|_{V^{2}_{\triangle}(l^{2};J^m\times \mathbb{R}^2)}=1}\int_{J^m}\sum_{j\in \mathbb{Z}}F_j(\tau)v_j(\tau) d\tau.$$
\end{lemma}
\begin{proof}
The proof is the same as Lemma 3.4 in \cite{D}. So we won't repeat it here.
\end{proof}
\begin{proposition}[bilinear Strichartz estimate]\label{pr-y5.4}
$\frac{1}{p}+\frac{1}{q}=1,\vec{u}_{0}=\{u_{0,j}\}_{j\in\mathbb{Z}}, \vec{v}_0=\{v_{0,j}\}_{j\in\mathbb{Z}}$. Assume  $\hat{u}_{0,j}$ is supported on $|\xi|\sim N$ and $\hat{v}_{0,j}$ is supported on $|\xi|\sim M$ for $j\in\mathbb{Z}$. If $M \ll N$, then
\begin{equation}\label{eq-y4.2'}
  \left\|\|e^{it\triangle}\vec{u}_{0}\|_{l^{2}}\cdot\|e^{\pm it\triangle}\vec{v}_0\|_{l^{2}}\right\|_{L_t^p L_x^q(\mathbb{R}\times\mathbb{R}^2)}\lesssim\big(\frac{M}{N}\big)^{\frac{1}{p}}\|\vec{u}_{0}\|_{L_x^2l^{2}}
  \|\vec{v}_0\|_{L_x^2l^{2}}.
\end{equation}
\end{proposition}
\begin{proof}
On one hand, by \cite{B} and Minkowski inequality, we have
\begin{align*}
  &\left\|\|e^{it\triangle}\vec{u}_{0}\|_{l^{2}}\cdot\|e^{\pm it\triangle}\vec{v}_0\|_{l^{2}}\right\|_{L_t^2 L_x^2(\mathbb{R}\times\mathbb{R}^2)} \\
 =&\left\|\left(\sum_{j\in\mathbb{Z}}|e^{it\triangle}u_{0,j}|^{2}\right)^{\frac{1}{2}}
 \left(\sum_{j'\in\mathbb{Z}}|e^{\pm it\triangle}v_{0,j'}|^2\right)^{\frac{1}{2}}\right\|_{L_t^2 L_x^2(\mathbb{R}\times\mathbb{R}^2)}\\
 =&\left\|\left(\sum_{j,j'\in\mathbb{Z}}|e^{it\triangle}u_{0,j}|^2|e^{\pm it\triangle}v_{0,j'}|^2\right)^{\frac{1}{2}}\right\|_{L_t^2 L_x^2(\mathbb{R}\times\mathbb{R}^2)}\\
 =&\left(\sum_{j,j'\in\mathbb{Z}}\left\|e^{it\triangle}u_{0,j}e^{\pm it\triangle}v_{0,j'}\right\|_{L_t^2 L_x^2(\mathbb{R}\times\mathbb{R}^2)}^2\right)^{\frac{1}{2}}\\
 \leq &\big(\frac{M}{N}\big)^{\frac{1}{2}}\left(\sum_{j,j'\in\mathbb{Z}}
 \left\|u_{0,j}\right\|_{L_x^2(\mathbb{R}^2)}^2
 \left\|v_{0,j'}\right\|_{L_x^2(\mathbb{R}^2)}^2\right)^{\frac{1}{2}}\\
 =&\big(\frac{M}{N}\big)^{\frac{1}{2}}\|\vec{u}_{0}\|_{L_x^2 l^{2}}\|\vec{v}_0\|_{L_x^2 l^{2}}.
\end{align*} On the other hand, by H\"{o}lder inequality, we get
\begin{align*}
   &\left\|\|e^{it\triangle}\vec{u}_{0}\|_{l^{2}}\cdot\|e^{\pm it\triangle}\vec{v}_0\|_{l^{2}}\right\|_{L_t^{\infty} L_x^1(\mathbb{R}\times\mathbb{R}^2)} \\
   \leq&\left\|e^{it\triangle}\vec{u}_{0}\right\|_{L_t^{\infty} L_x^2l^{2}}\left\|e^{\pm it\triangle}\vec{v}_0\right\|_{L_t^{\infty} L_x^2l^{2}}\\
   \leq&\|\vec{u}_{0}\|_{L_x^2l^{2}}\|\vec{v}_0\|_{L_x^2l^{2}}.
\end{align*}
Interpolating between the above two inequalities, we obtain
$$\left\|\|e^{it\triangle}\vec{u}_{0}\|_{l^{2}}\cdot\|e^{\pm it\triangle}\vec{v}_0\|_{l^{2}}\right\|_{L_t^pL_x^q(\mathbb{R}\times\mathbb{R}^2)} \lesssim\big(\frac{M}{N}\big)^{\frac{1}{p}}\|\vec{u}_{0}\|_{L_x^2l^{2}}\|\vec{v}_0\|_{L_x^2l^{2}}.$$
\end{proof}
\begin{remark}\label{convolution}
Suppose that $g(t,x-y)$ and $h(t,x-z)$ are convolution kernels with the bounds
\begin{equation}\label{condition1}
 \|\sup_{t\in \mathbb{R}}|g(t,x)|\|_{L^1(\mathbb{R}^2)}\lesssim1\quad and\quad \|\sup_{t\in \mathbb{R}}|h(t,x)|\|_{L^1(\mathbb{R}^2)}\lesssim1,
\end{equation}then under the hypothesis of Proposition \ref{pr-y5.4}, we have
\begin{equation}\label{coneq-y4.2'}
  \left\|\|g\ast e^{it\triangle}\vec{u}_{0}\|_{l^{2}}\cdot\|h\ast e^{\pm it\triangle}\vec{v}_0\|_{l^{2}}\right\|_{L_t^p L_x^q(\mathbb{R}\times\mathbb{R}^2)}\lesssim\big(\frac{M}{N}\big)^{\frac{1}{p}}\|\vec{u}_{0}\|_{L_x^2l^{2}}
  \|\vec{v}_0\|_{L_x^2l^{2}},
\end{equation}where $\|g\ast e^{it\triangle}\vec{u}_{0}\|_{l^{2}}=\big(\sum_j|g\ast e^{it\triangle}u_{0,j}|^2\big)^{1/2}$. The kernels of $P_{\xi(t),j}$, $P_{\xi(t),\leq j}$ and $P_{\xi(t),\geq j}$ all satisfy \eqref{condition1}.
\end{remark}

\begin{proposition}\label{pr-y5.5}
 For $\vec{u}=\{u_j\}_{j\in\mathbb{Z}},\vec{v}=\{v_j\}_{j\in\mathbb{Z}}$, we assume $supp\ \ \hat{u}_j\subset\{\xi:|\xi|\sim N\},supp\ \ \hat{v}_j\subset\{\xi:|\xi|\sim M\}$. If $M\ll  N$, we have
\begin{equation}
\left\|\|\vec{u}\|_{l^{2}}\cdot\|\vec{v}\|_{l^{2}}\right\|_{L^p_t L^q_x}\lesssim\big(\frac{M}{N}\big)^{\frac{1}{p}}\|\vec{u}\|_{U^p_{\triangle}(l^{2})} \|\vec{v}\|_{U^p_{\triangle}(l^{2})},
\end{equation}where $\frac{1}{p}+\frac{1}{q}=1$.
\end{proposition}
\begin{proof}
We just take $U^p_{\triangle}(l^{2})$ and $U^p_{\triangle}(l^{2})$ atoms into consideration .  Let
$$\vec{u}=\sum_k \chi_{[t_k,t_{k+1}]}e^{it\triangle}\vec{u}_{k}, \sum\limits_{k}\|\vec{u}_k(x)\|_{L_x^2l^{2}}^p=1; \vec{v}=\sum_{k'} \chi_{[t_{k'},t_{k'+1}]}e^{it\triangle}\vec{v}_{k'}, \sum\limits_{k'}\|\vec{v}_{k'}(x)\|_{L_x^2l^{2}}^p=1.$$
Without loss of generality, we may assume $\|\vec{u}\|_{U^p_{\triangle}(l^{2})}=1, \|\vec{v}\|_{U^p_{\triangle}(l^{2})}=1$.  It suffices to show
$$\left\|\|\vec{u}\|_{l^{2}}\cdot\|\vec{v}\|_{l^{2}}\right\|_{L^p_t L^q_x}\lesssim \big(\frac{M}{N}\big)^{\frac{1}{p}}.$$
By Minkowski inequality and Proposition \ref{pr-y5.4}, we have
\begin{align*}
  \left\|\|\vec{u}\|_{l^{2}}\cdot\|\vec{v}\|_{l^{2}}\right\|^p_{L^p_t L^q_x}=&\left\|\|\sum_k \chi_{[t_k,t_{k+1}]}e^{it\triangle}\vec{u}_{k}\|_{l^{2}}\cdot\|\sum_{k'} \chi_{[t_{k'},t_{k'+1}]}e^{it\triangle}\vec{v}_{k'}\|_{l^{2}}\right\|^p_{L^p_t L^q_x} \\
  \leq&\left\|(\sum_{k}\chi_{[t_k,t_{k+1}]}\|e^{it\triangle}\vec{u}_{k}\|_{l^{2}})\cdot(\sum_{k'} \chi_{[t_{k'},t_{k'+1}]}\|e^{it\triangle}\vec{v}_{k'}\|_{l^{2}})\right\|^p_{L^p_t L^q_x}\\
  =&\left\|(\sum_{k,k'}\chi_{[t_k,t_{k+1}]}\chi_{[t_{k'},t_{k'+1}]}\|e^{it\triangle}\vec{u}_{k}\|_{l^{2}} \|e^{it\triangle}\vec{v}_{k'}\|_{l^{2}})\right\|^p_{L^p_t L^q_x}\\
  \leq&\int_{\mathbb{R}}\left(\sum_{k,k'}\chi_{[t_k,t_{k+1}]}\chi_{[t_{k'},t_{k'+1}]} \left\|\|e^{it\triangle}\vec{u}_{k}\|_{l^{2}} \|e^{it\triangle}\vec{v}_{k'}\|_{l^{2}}\right\|_{L_x^q}\right)^p(t)dt\\
  =&\sum_k\int_{t_k}^{t_{k+1}}\left(\sum_{k'}\chi_{[t_{k'},t_{k'+1}]} \left\|\|e^{it\triangle}\vec{u}_{k}\|_{l^{2}} \|e^{it\triangle}\vec{v}_{k'}\|_{l^{2}}\right\|_{L_x^q}\right)^p(t)dt\\
  \leq&\sum_k\sum_{k'}\int_{[t_{k'},t_{k'+1}]\cap[t_{k},t_{k+1}]} \left(\left\|\|e^{it\triangle}\vec{u}_{k}\|_{l^{2}} \|e^{it\triangle}\vec{v}_{k'}\|_{l^{2}}\right\|_{L_x^q}\right)^p(t)dt\\
  =&\sum_k\sum_{k'}\left(\left\|\|e^{it\triangle}\vec{u}_{k}\|_{l^{2}} \|e^{it\triangle}\vec{v}_{k'}\|_{l^{2}}\right\|_{L^p_t L_x^q([t_{k'},t_{k'+1}]\cap[t_{k},t_{k+1}])}\right)^p\\
  \leq&\frac{M}{N}\sum_{k}\sum_{k'}\|\vec{u}_{k}\|^p_{L^2_x l^{2}}\|\vec{v}_{k'}\|^p_{L^2_x l^{2}}\\
  =&\frac{M}{N}(\sum_{k}\|\vec{u}_{k}\|^p_{L^2_x l^{2}})(\sum_{k'}\|\vec{v}_{k'}\|^p_{L^2_xh^{0}})\\
  =&\frac{M}{N}.
\end{align*}Therefore,
$$\left\|\|\vec{u}\|_{l^{2}}\cdot\|\vec{v}\|_{l^{2}}\right\|_{L^p_t L^q_x}\lesssim \big(\frac{M}{N}\big)^{\frac{1}{p}}.$$
\end{proof}

In next section, we will use these properties to give the crucial estimates.

\section{Long time Strichartz estimate}
The argument in this section is similar to the section 5's of \cite{D}. However, because the nonlinear terms in this paper are different from the nonlinear term in \cite{D}, we will take a slightly different norm to exploit the nonlinear terms' as much symmetry as possible.

Let's first do some preparations. Fix three constants $0<\ep_3 \ll \ep_2\ll \ep_1<1$ in the following. By Corollary \ref{co-y4.9}  and Remark \ref{re-y4.10}, $\ep_1,\ep_2,\ep_3$ can also satisfy
\begin{equation}\label{eq-y5.1}
|N'(t)|+|\xi'(t)|\leq2^{-20}\frac{N(t)^3}{\ep_1^{1/2}},
\end{equation}
\begin{equation}\label{eq-y5.2}
\sum\limits_{j\in\mathbb{Z}} \left[\int_{|x-x(t)|\geq \frac{2^{-20}\ep_3^{-1/4}}{N(t)}}|u_{j}(x)|^{2}\mathrm{d}x+\int_{|\xi-\xi(t)|\geq 2^{-20}\ep_3^{-1/4}N(t)}|\hat{u}_{j}(\xi)|^{2}\mathrm{d}\xi\right]<\epsilon_2^2.
\end{equation}

Suppose $M=2^{k_0}$ is a dyadic integer with $k_0\geq0$. Let $[0,T]$ be an interval such that $\|\vec{u}\|^4_{L^4_{t,x}l^{2}([0,T])}=M$ and $\int_0^T N(t)^3 dt=\ep_3M$. Partition $[0,T]=\cup_{l=0}^{M-1} J_l$ with $\|\vec{u}\|_{L^4_{t,x}l^{2}(J_l)}=1$, we call the intervals $J_l$'s small intervals.

\begin{definition}\label{de-y5.1}
For an integer $0\leq j<k_0$, $0\leq k<2^{k_0-j}$, let
$$G_{k}^{j}=\cup_{\alpha=k2^j}^{(k+1)2^j-1}J^{\alpha}.$$
Where $J^{\alpha}$'s satisfy $[0,T]=\cup_{\alpha=0}^{M-1} J^{\alpha}$  with
\begin{equation}\label{eq-y5.33}
\int_{J^{\alpha}} \big( N(t)^3 + \ep_3\|\vec{u}(t)\|^4_{L_x^4l^{2}(\mathbb{R}^2\times \mathbb{Z})} \big)dt=2\ep_3.
\end{equation}
 For$j\geq k_0$ let $G_{k}^j=[0,T].$
 Now suppose that $G_k^j=[t_0,t_1]$, let $\xi(G_k^j)=\xi(t_0)$ and define $\xi(J_l)$, $\xi(J^{\alpha})$ in a similar manner.
\end{definition}
\begin{remark}\label{re-y5.1}There are some differences between $J_l$'s and $J^{\alpha}$'s:
\begin{enumerate}
\item As in the remark of Dodson's book \cite{D6}, If $N(t)$ was constant, then there would exist a constant $C$ such that at most $C$ $J_l$ intervals intersect any one $J^{\alpha}$, and at most $C$ $J^{\alpha}$ intervals intersect any one $J_l$. However, since $N(t)$ is free to move around, the partitions $J_l$'s and $J^{\alpha}$'s can be quite different.
\item It follows from Theorem \ref{th-y4.5} that $N(J_l)\sim \int_{J_l}N(t)^3 dt\sim \inf_{t\in J_l}N(t)$. Additionally, by \eqref{eq-y5.33}, we obtain
   \begin{equation}\label{eq-y5.34}
   \sum_{J_l\subset G_k^j}N(J_l)\lesssim \sum_{J_l\subset G_k^j}\int_{J_l}N(t)^3dt\lesssim \int_{G_k^j}N(t)^3dt\lesssim \sum_{\alpha=k2^j}^{(k+1)2^j-1}\int_{J^{\alpha}}N(t)^3dt\lesssim 2^j\ep_3.
   \end{equation}
\item By \eqref{eq-y5.1} and Definition \ref{de-y5.1},  for all $t \in G_k^{j}$,
  \begin{equation}
  |\xi(t)-\xi(G_k^{j})|\leq \int_{G_k^{j}}2^{-20}\ep_{1}^{-1/2}N(t)^3dt \leq 2^{j-19}\ep_3 \ep_{1}^{-1/2}.
  \end{equation}
  Therefore, for all $t \in G_k^{j}$ and $i\geq j$,
  $$\{\xi:2^{i-1}\leq |\xi-\xi(t)|\leq 2^{i+1} \} \subset \{\xi:2^{i-2}\leq |\xi-\xi(G_k^{j})|\leq 2^{i+2} \} \subset \{\xi:2^{i-3}\leq |\xi-\xi(t)|\leq 2^{i+3} \},$$
  and
  $$\{\xi:|\xi-\xi(t)|\leq 2^{i+1} \} \subset \{\xi: |\xi-\xi(G_k^{j})|\leq 2^{i+2} \} \subset \{\xi: |\xi-\xi(t)|\leq 2^{i+3} \}.$$
\end{enumerate}
\end{remark}

The following lemma gives some ``smallness" characterization of the intervals $J_l$'s and $J^{\alpha}$'s.
\begin{lemma}\label{le-y6.2}
 Suppose $\vec{u}(t)$ is a minimal mass blowup solution to \eqref{eq-y1}. If $J$ is a time interval with
$\|\vec{u}\|_{L^4_{t,x}l^{2}(J)}\lesssim 1$, then
$$\|\vec{u}\|_{U^2_{\triangle}(l^{2};J)}\lesssim1 \ \ and \ \ \left\|P_{>\frac{N(J)}{2^4\ep_3^{1/4}}} \vec{u}\right\|_{U^2_{\triangle}(l^{2};J)}\lesssim \ep_2,$$ where $N(J)=\sup_{t\in J}N(t)$. Furthermore, by Proposition \ref{pr-y5.3}(6),
 $$\|\vec{u}\|_{L^p_tL^q_x l^{2}(J)}\lesssim 1,\,\,\, \|P_{>\frac{N(J)}{2^4\ep_3^{1/4}}} \vec{u}\|_{L^p_tL^q_x l^{2}(J)}\lesssim \ep_2.$$
for $(p,q)$ admissible pair and $p>2$.
\end{lemma}
\begin{proof}Let $J=[0,b]$. By Duhamel's formula, for $t \in J,$
\begin{align*}
\vec{u} = e^{it\Delta}\vec{u_{0}}  + i\int_0^t e^{i(t-s) \Delta} \mathbf{F}(\vec{u})\,\mathrm{d}s,
\end{align*}
where $\mathbf{F}(\vec{u}):=\{F_j(\vec{u})\}_{j\in\mathbb{Z}}:=\{\sum\limits_{\mathcal{R}(j)}u_{j_1}\bar{u}_{j_2}u_{j_3}\}_{j\in\mathbb{Z}}$. Then it follows from Proposition \ref{pr-y5.3} that
\begin{equation*}
\begin{split}
   \|\vec{u}\|_{U^2_{\triangle}(l^{2};J)}
   \leq& \|\vec{u}_{0}\|_{L^2_x l^{2}}+\|\mathbf{F}(\vec{u})\|_{DU^2_{\triangle}(l^{2};J)} \\
   \leq& \|\vec{u}_{0}\|_{L^2_x l^{2}}+\|\mathbf{F}(\vec{u})\|_{L^{\frac{4}{3}}_{t,x}l^{2}(J)}\\
   \leq& \|\vec{u}_{0}\|_{L^2_x l^{2}}+\|\vec{u}\|^3_{L^{4}_{t,x}l^{2}(J)}\\
   \lesssim& 1.
\end{split}
\end{equation*}
Since $\vec{u}$ is an almost periodic solution, by \eqref{eq-y5.2}, we have
$$\left\|P_{>\frac{N(J)}{2^{20}\ep_3^{1/4}}} \vec{u}\right\|_{L_t^{\infty}L^2_xl^2(J)}\leq \left\|P_{>\frac{N(t)}{2^{20}\ep_3^{1/4}}} \vec{u}\right\|_{L_t^{\infty}L^2_xl^2(J)}\leq \ep_2.$$
Thus,
\begin{align*}
   \left\|P_{>\frac{N(J)}{2^{4}\ep_3^{1/4}}} \vec{u}\right\|_{U^2_{\triangle}(l^{2};J)} \leq & \left\|P_{>\frac{N(J)}{2^{4}\ep_3^{1/4}}} e^{it\Delta}\vec{u_{0}}\right\|_{U^2_{\triangle}(l^{2};J)} + \left\|P_{>\frac{N(J)}{2^{4}\ep_3^{1/4}}} \mathbf{F}(\vec{u})\right\|_{DU^2_{\triangle}(l^{2};J)} \\
   \lesssim& \left\|P_{>\frac{N(J)}{2^{20}\ep_3^{1/4}}} \vec{u}_{0}\right\|_{L_t^{\infty}L^2_xl^2(J)}+ \left\|P_{>\frac{N(J)}{2^{4}\ep_3^{1/4}}} \mathbf{F}(\vec{u})\right\|_{L_t^{3/2}L_x^{6/5}l^{2}(J)}\\
   \lesssim& \ep_2 + \left\|P_{>\frac{N(J)}{2^{20}\ep_3^{1/4}}} \vec{u}\right\|_{L_t^{\infty}L^2_x l^{2}(J)} \|\vec{u}\|^2_{L_t^3L_x^6l^{2}(J)}\\
   \lesssim& \ep_2 + \ep_2 \left(\|\vec{u}_{0}\|_{L_t^{\infty}L_x^2l^{2}(J)} + \|\vec{u}\|^3_{L_{t,x}^4l^{2}(J)}\right)^2\\
   \lesssim& \ep_2.
\end{align*}
\end{proof}
\begin{remark}\label{re-y5.4}
If $N(J)< 2^{i-5}\ep_3^{1/2},$
$$\|P_{\xi(G_{\alpha}^{i}), i-2 \leq \cdot \leq i+2} \mathbf{F}(\vec{u})\|_{L_t^{3/2}L_x^{6/5}l^{2}(J)}\lesssim \big\|P_{>\frac{N(J)}{2^{20}\ep_3^{1/4}}} \vec{u}\big\|_{L_t^{\infty}L^2_xl^2(J)} \|\vec{u}\|^2_{L_t^3L_x^6l^{2}(J)}\lesssim \ep_2.$$
So for $0\leq i \leq 11$, $N(G_{\alpha}^{i})< 2^{i-5}\ep_3^{1/2},$ since $G_{\alpha}^{i}$ is a union of $\leq 2^{11}$ such small intervals,
$$\|P_{\xi(G_{\alpha}^{i}), i-2 \leq \cdot \leq i+2} \mathbf{F}(\vec{u})\|_{L_t^{3/2}L_x^{6/5}l^{2}(G_{\alpha}^{i})}\lesssim \ep_2.$$
\end{remark}
Now we define our spaces as in the Section 3 of \cite{D}, in which we derive the long time Strichartz estimates.
\begin{definition}[$\tilde{X}_{k_0}$ spaces]\label{de-y5.2}
For any $G_k^{j}\subset[0,T]$ let
\begin{equation}\label{eq-y5.3'}
  \|\vec{u}\|^2_{X(G_k^{j})}:=\sum_{0\leq i< j}2^{i-j}\sum_{G_{\alpha}^{i}\subset G_k^{j}}\|P_{\xi(G_{\alpha}^{i}), i-2\leq\cdot\leq i+2}\vec{u}\|^2_{U^2_{\triangle}(l^{2};G_{\alpha}^{i}\times\mathbb{R}^2)}
  +\sum_{i\geq j}\|P_{\xi(G_{k}^{j}),i-2\leq\cdot\leq i+2} \vec{u}\|^2_{U^2_{\triangle}(l^{2};G_{k}^{j}\times\mathbb{R}^2)}.
\end{equation}
Here $P_{\xi(t),i-2\leq\cdot\leq i+2} \vec{u}=e^{ix \cdot \xi(t)}P_{i-2\leq\cdot\leq i+2}(e^{-ix\cdot \xi(t)} \vec{u})$ with $P_{i-2\leq\cdot\leq i+2}$ being the Littlewood-Paley projector.\\
Then  define $\tilde{X}_{k_0}$ to be the supremum of \eqref{eq-y5.3'} over all intervals $G_k^j\subset[0,T]$ with $k\leq k_0$.
\begin{equation}\label{X}
\|\vec{u}\|^2_{\tilde{X}_{k_0}([0,T])}:=\sup_{0\leq j\leq k_0}\sup_{G_k^{j}\subset[0,T]}\|\vec{u}\|^2_{X(G_k^{j})}.
\end{equation}
Also for $0\leq k_{\ast}\leq k_0$, let
\begin{equation}\label{XXX}
\|\vec{u}\|^2_{\tilde{X}_{k_{\ast}}([0,T])}:=\sup_{0\leq j\leq k_{\ast}}\sup_{G_k^{j}\subset[0,T]}\|\vec{u}\|^2_{X(G_k^{j})}.
\end{equation}
\end{definition}

\begin{definition}[$\tilde{Y}_{k_0}$ spaces]
The $\tilde{Y}_{k_0}$ norm measures the $\tilde{X}_{k_0}$ norm of $\vec{u}$ at scales much higher than $N(t)$. This norm provides some crucial ``smallness", closing a bootstrap argument in the next section. Let
\begin{equation}
\begin{split}
  \|\vec{u}\|^2_{Y(G_k^{j})}&:=\sum_{0< i< j}2^{i-j}\sum_{G_{\alpha}^{i}\subset G_k^{j}:N(G_{\alpha}^{i})\leq 2^{i-5}\ep_3^{1/2}}\|P_{\xi(G_{\alpha}^{i}), i-2\leq\cdot\leq i+2} \vec{u}\|^2_{U^2_{\triangle}(l^{2};G_{\alpha}^{i}\times\mathbb{R}^2)} \\
    & +\sum_{i\geq j,i>0:N(G_{k}^{j})\leq 2^{i-5}\ep_3^{1/2}}\|P_{\xi(G_{k}^{j}),i-2\leq\cdot\leq i+2} \vec{u}\|^2_{U^2_{\triangle}(l^{2};G_{k}^{j}\times\mathbb{R}^2)}.
\end{split}
\end{equation}
Define $\|\vec{u}\|_{\tilde{Y}_{k_{\ast}}([0,T])}$
using $\|\vec{u}\|_{Y(G_k^{j})}$ in the same way as $\|\vec{u}\|_{\tilde{X}_{k_{\ast}}([0,T])}$ was done.
\end{definition}

After giving the long time Strichartz norms, we should point out the relationship between $L^p_tL^q_xh^{0}$ norm and the long time Strichartz norms.
\begin{lemma}\label{le-y5.7}
For $i < j$, $(p,q)$ an admissible pair, we have
\begin{equation}\label{eq-y5.9}
  \|P_{\xi(t),i}\vec{u}\|_{L^p_tL^q_xl^2(G^j_k\times\mathbb{R}^2)}\lesssim 2^{\frac{j-i}{p}} \|\vec{u}\|_{\tilde{X}_j(G^j_k)}.
\end{equation}
\begin{equation}\label{eq-y5.9'}
  \|P_{\xi(t),\geq j}\vec{u}\|_{L^p_tL^q_xl^2(G^j_k\times\mathbb{R}^2)}\lesssim  \|\vec{u}\|_{X(G^j_k)}.
\end{equation}
\end{lemma}

\begin{proof}By the definition of $\tilde{X}_j$ and Proposition \ref{pr-y5.3},
\begin{equation*}
\begin{split}
   &\|P_{\xi(t),i}\vec{u}\|_{L^p_tL^q_xl^2(G^j_k\times\mathbb{R}^2)}\\
   =& \big(\sum_{G^i_{\alpha} \subset G^j_k}\|P_{\xi(t),i}\vec{u}\|^p_{L^p_tL^q_xl^2(G^i_{\alpha}\times\mathbb{R}^2)}\big)^{1/p} \\
   \lesssim & \big(\sum_{G^i_{\alpha} \subset G^j_k}\|P_{\xi(t),i}\vec{u}\|^2_{L^p_tL^q_xl^2(G^i_{\alpha}\times\mathbb{R}^2)}\big)^{1/p} \big(\sup_{G^i_{\alpha} \subset G^j_k}\|P_{\xi(t),i}\vec{u}\|_{L^p_tL^q_xl^2(G^i_{\alpha}\times\mathbb{R}^2)}\big)^{1-\frac{2}{p}}\\
   \lesssim & \big(\sum_{G^i_{\alpha} \subset G^j_k}\|P_{\xi(G^i_{\alpha}), i-2\leq \cdot \leq i+2} \vec{u}\|^2_{U^2_{\triangle}(l^{2};G^i_{\alpha}\times\mathbb{R}^2)}\big)^{1/p} \big(\sup_{G^i_{\alpha} \subset G^j_k}\|P_{\xi(G^i_{\alpha}), i-2\leq \cdot \leq i+2} \vec{u}\|^2_{U^2_{\triangle}(l^{2};G^i_{\alpha}\times\mathbb{R}^2)}\big)^{1-\frac{2}{p}}\\
   \lesssim & 2^{\frac{j-i}{p}} \|\vec{u}\|^{\frac{2}{p}}_{X(G^j_k)} \|\vec{u}\|^{1-\frac{2}{p}}_{\tilde{X}_j(G^j_k)}\\
   \lesssim & 2^{\frac{j-i}{p}} \|\vec{u}\|_{\tilde{X}_j(G^j_k)}.
\end{split}
\end{equation*}
Also, by Lemma \ref{le-y2.1} and Minkowski inequality,
\begin{equation*}
\begin{split}
   &\|P_{\xi(t),\geq j}\vec{u}\|_{L^p_tL^q_xl^2(G^j_k\times\mathbb{R}^2)}\\
   \sim&\big\|\big(\sum_{l\geq j} |P_{\xi(t),l}\vec{u}|^2\big)^{1/2}\big\|_{L^p_tL^q_xl^2(G^j_k\times\mathbb{R}^2)}\\
   \lesssim & \big(\sum_{l\geq j}\|P_{\xi(t),l}\vec{u}\|^2_{L^p_tL^q_xl^2(G^j_k\times\mathbb{R}^2)}\big)^{1/2}\\
   \lesssim & \|\vec{u}\|_{X(G^j_k)}.
\end{split}
\end{equation*}
\end{proof}

Now we are ready to state the main theorem of this section.
\begin{theorem}[Long time Strichartz estimate]\label{th-y6.2}
Suppose $\vec{u}(t)$ is an almost periodic solution to
\begin{equation*}
\begin{cases}
i\partial_t u_j + \Delta_{\mathbb{R}^2} u_j = \sum\limits_{\mathcal{R}(j)} u_{j_1} \bar{u}_{j_2} u_{j_3},\\
u_j(0) = u_{0,j}.
\end{cases}
\end{equation*}for $\vec{u}_{0}=\{u_{0,j}\}_{j\in\mathbb{Z}}\in L^2_xh^1(\mathbb{R}^2\times\mathbb{Z}).$
Then there exists a constant $C>0$ (only depending on $\vec{u}$), such that for any $M=2^{k_0}$, $\ep_1,\ep_2,\ep_3$ satisfying above conditions, $\|\vec{u}\|^4_{L^4_{t,x}l^{2}([0,T])}=M$ and $\int_0^T N(t)^3 dt=\ep_3M$,
$$\|\vec{u}\|_{\tilde{X}_{k_0}([0,T])}\leq C.$$
\end{theorem}
\begin{remark}
Throughout this section the implicit constant depends only on $\vec{u}$, and not on $M$, or $\ep_1,\ep_2,\ep_3$.
\end{remark}

\begin{proof}[Proof of Theorem \ref{th-y6.2}]
 By Definition \ref{de-y5.2}, it suffices to show there exists a constant $C>0$ (only depending on $\vec{u}$), such that for any $0 \leq j \leq k_0$ and $G_k^j \subset [0,T]$,
\begin{equation}
 \sum_{0\leq i< j}2^{i-j}\sum_{G_{\alpha}^{i}\subset G_k^{j}}\|P_{\xi(G_{\alpha}^{i}), i-2\leq\cdot\leq i+2}\vec{u}\|^2_{U^2_{\triangle}(l^{2};G_{\alpha}^{i}\times\mathbb{R}^2)}
  +\sum_{i\geq j}\|P_{\xi(G_{k}^{j}),i-2\leq\cdot\leq i+2} \vec{u}\|^2_{U^2_{\triangle}(l^{2};G_{k}^{j}\times\mathbb{R}^2)}\leq C.
\end{equation}
In order to get this estimate, we will combine an induction on $0\leq k_{\ast} \leq k_0$ and a bootstrap argument. First we give the basic inductive estimates.
\begin{lemma}\label{le-y5.8}
We have
$$\|\vec{u}\|_{\tilde{X}_0([0,T])}\leq C.$$
$$\|\vec{u}\|_{\tilde{Y}_0([0,T])}\leq C\ep_2^{3/4}.$$
And for $0\leq k_{\ast} \leq k_0$,
$$\|\vec{u}\|^2_{\tilde{X}_{k_{\ast}+1}([0,T])}\leq 2\|\vec{u}\|^2_{\tilde{X}_{k_{\ast}}([0,T])}.$$
$$\|\vec{u}\|^2_{\tilde{Y}_{k_{\ast}+1}([0,T])}\leq 2\|\vec{u}\|^2_{\tilde{Y}_{k_{\ast}}([0,T])}.$$
\end{lemma}
\begin{proof}
By the decomposition of $J^{\alpha}$'s intervals in Definition \ref{de-y5.1} and Lemma \ref{le-y6.2}, $\|\vec{u}\|_{U_{\triangle}^2(l^{2};J^{\alpha})} \lesssim 1$ for each $J^{\alpha} \subset G_k^j$. Therefore, by Duhamel's formula and Proposition \ref{pr-y5.3} (7), (6) and (2), we have
\begin{equation*}
\begin{split}
    & \big(\sum_{i\geq 0} \| P_{\xi(J^{\alpha}),i-2 \leq\cdot\leq i+2} \vec{u}\|^2_{U_{\triangle}^2(l^{2};J^{\alpha})}\big)^{1/2} \\
    \lesssim & \big( \sum_{i\geq 0} \| P_{\xi(J^{\alpha}),i-2 \leq\cdot\leq i+2} \vec{u}(t_0)\|^2_{L_x^2l^{2}}\big)^{1/2} +  \big( \sum_{i\geq0} \| P_{\xi(J^{\alpha}),i-2 \leq\cdot\leq i+2} \mathbf{F}(\vec{u})\|^2_{DU^2_{\triangle}(l^{2};J^{\alpha})}  \big)^{1/2}\\
    \lesssim & \| \vec{u}(t_0) \|_{L_x^2l^{2}} + \|\mathbf{F}(\vec{u})\|_{L^1_tL^2_xh^{0}(J^{\alpha})}\\
    \lesssim & \| \vec{u}(t_0) \|_{L_x^2l^{2}} + \|\vec{u}\|^3_{L^3_tL^6_xh^{0}(J^{\alpha})}\\
    \lesssim & \| \vec{u}(t_0) \|_{L_x^2l^{2}} + \|\vec{u}\|^3_{U^2_{\triangle}(l^{2};J^{\alpha})}\\
    \lesssim &1.
\end{split}
\end{equation*}
The second term in the third line comes from $L^1_tL^2_xh^{0} \subset DU^2_{\triangle}(l^{2})$, Minkowski inequality and Plancherel's identity. By Definition \ref{de-y5.2},
$$\|\vec{u}\|_{\tilde{X}_0([0,T])}\leq C.$$
We can similarly calculate
\begin{equation*}
\begin{split}
    & \big(\sum_{i\geq 0:N(J^{\alpha})\leq \ep_3^{1/2}2^{i-5}} \| P_{\xi(J^{\alpha}),i-2 \leq\cdot\leq i+2} \vec{u}\|^2_{U_{\triangle}^2(l^{2};J^{\alpha})}\big)^{1/2} \\
    \lesssim & \| P_{\xi(J^{\alpha}),\geq 8\ep_3^{-1/2}N(J^{\alpha})}\vec{u}(t_0) \|_{L_x^2l^{2}} + \|P_{\xi(J^{\alpha}),\geq 8\ep_3^{-1/2}N(J^{\alpha})} \mathbf{F}(\vec{u})\|_{L^1_tL^2_xh^{0}(J^{\alpha})}\\
    \lesssim & \| P_{\xi(t),\geq 4\ep_3^{-1/2}N(t)}\vec{u} \|_{L_t^{\infty}L_x^2l^{2}(J^{\alpha})} + \|P_{\xi(t),\geq 4\ep_3^{-1/2}N(t)} \mathbf{F}(\vec{u})\|_{L^1_tL^2_xh^{0}(J^{\alpha})}\\
    \lesssim & \| P_{\xi(t),\geq \ep_3^{-1/2}N(t)}\vec{u} \|^{3/4}_{L_t^{\infty}L_x^2l^{2}(J^{\alpha})} \big( \|\vec{u} \|^{1/4}_{L_t^{\infty}L_x^2l^{2}(J^{\alpha})} + \|\vec{u}\|^{9/4}_{L^{9/4}_tL^{18}_xh^{0}(J^{\alpha})}\big)\\
    \lesssim & \| P_{\xi(t),\geq \ep_3^{-1/2}N(t)}\vec{u} \|^{3/4}_{L_t^{\infty}L_x^2l^{2}(J^{\alpha})} \big( \|\vec{u} \|^{1/4}_{L_t^{\infty}L_x^2l^{2}(J^{\alpha})} + \|\vec{u}\|^{9/4}_{U_{\triangle}^2(l^{2};J^{\alpha})}\big)\\
    \lesssim &\ep_2^{3/4}.
\end{split}
\end{equation*}
In  the last inequality we use \eqref{eq-y5.2}, mass conversation and $\|\vec{u}\|_{U_{\triangle}^2(l^{2};J^{\alpha})} \lesssim 1$. Therefore,
$$\|\vec{u}\|_{\tilde{Y}_0([0,T])}\leq C\ep_2^{3/4}.$$
From Definition \ref{de-y5.1}, $G^{j+1}_k= G^{j}_{2k} \bigcup G^{j}_{2k+1}$ with $G^{j}_{2k} \bigcap G^{j}_{2k+1}=\varnothing$, and for $0\leq i\leq j$, if $G^{i}_{\alpha} \subset G^{j+1}_k= G^{j}_{2k} \bigcup G^{j}_{2k+1}$, then either $G^{i}_{\alpha} \subset G^{j}_{2k}$ or $G^{i}_{\alpha} \subset G^{j}_{2k+1}$ would happen. Thus
\begin{equation}\label{eq-y5.8}
\begin{split}
  &\sum_{0\leq i< j+1}2^{i-(j+1)}\sum_{G_{\alpha}^{i}\subset G_k^{j+1}}\|P_{\xi(G_{\alpha}^{i}), i-2\leq\cdot\leq i+2}\vec{u}\|^2_{U^2_{\triangle}(l^{2};G_{\alpha}^{i}\times\mathbb{R}^2)}\\
  \leq& 2^{-1}\sum_{0\leq i< j}2^{i-j} \big[\sum_{G_{\alpha}^{i}\subset G_{2k}^{j}}\|P_{\xi(G_{\alpha}^{i}), i-2\leq\cdot\leq i+2}\vec{u}\|^2_{U^2_{\triangle}(l^{2};G_{\alpha}^{i}\times\mathbb{R}^2)}+\sum_{G_{\alpha}^{i}\subset G_{2k+1}^{j}}\|P_{\xi(G_{\alpha}^{i}), i-2\leq\cdot\leq i+2}\vec{u}\|^2_{U^2_{\triangle}(l^{2};G_{\alpha}^{i}\times\mathbb{R}^2)}\big]\\
  &+ 2^{-1} \big[\|P_{\xi(G_{2k}^{j}), j-2\leq\cdot\leq j+2}\vec{u}\|^2_{U^2_{\triangle}(l^{2};G_{2k}^{j}\times\mathbb{R}^2)} + \|P_{\xi(G_{2k+1}^{j}), j-2\leq\cdot\leq j+2}\vec{u}\|^2_{U^2_{\triangle}(l^{2};G_{2k+1}^{j}\times\mathbb{R}^2)}\big]\\
  \leq& \frac{1}{2}\big[ \|\vec{u}\|^2_{X(G_{2k}^{j})}+\|\vec{u}\|^2_{X(G_{2k+1}^{j})}\big].
\end{split}
\end{equation}
Meanwhile notice that for all $t \in G_k^{j+1}$, from Remark \ref{re-y5.1}(3), we have
\begin{equation}
  |\xi(t)-\xi(G_k^{j+1})|\leq 2^{-18}\ep_3 \ep_{1}^{-1/2}.
\end{equation}
Therefore, for all $t \in G_k^{j+1}$ and $i\geq j$,
$$\{\xi:2^{i-1}\leq |\xi-\xi(t)|\leq 2^{i+1} \} \subset \{\xi:2^{i-2}\leq |\xi-\xi(G_k^{j+1})|\leq 2^{i+2} \} \subset \{\xi:2^{i-3}\leq |\xi-\xi(t)|\leq 2^{i+3} \},$$
which,combined with Proposition \ref{pr-y5.3} (5), yields
\begin{equation}\label{eq-y5.10}
\begin{split}
  &\sum_{i\geq j+1}\|P_{\xi(G_{k}^{j+1}),i-2\leq\cdot\leq i+2} \vec{u}\|^2_{U^2_{\triangle}(l^{2};G_{k}^{j+1}\times\mathbb{R}^2)}\\
  \leq & \sum_{i\geq j+1} \big[\|P_{\xi(G_{k}^{j+1}),i-2\leq\cdot\leq i+2} \vec{u}\|^2_{U^2_{\triangle}(l^{2};G_{2k}^{j}\times\mathbb{R}^2)} + \|P_{\xi(G_{k}^{j+1}),i-2\leq\cdot\leq i+2} \vec{u}\|^2_{U^2_{\triangle}(l^{2};G_{2k+1}^{j}\times\mathbb{R}^2)} \big]\\
  \leq & \sum_{i\geq j+1} \big[\|P_{\xi(G_{2k}^{j}),i-3\leq\cdot\leq i+3} \vec{u}\|^2_{U^2_{\triangle}(l^{2};G_{2k}^{j}\times\mathbb{R}^2)} + \|P_{\xi(G_{2k+1}^{j}),i-3\leq\cdot\leq i+3} \vec{u}\|^2_{U^2_{\triangle}(l^{2};G_{2k+1}^{j}\times\mathbb{R}^2)} \big]\\
\end{split}
\end{equation}
Thus \eqref{eq-y5.8} and \eqref{eq-y5.10} give
\begin{equation}
\begin{split}
  \|\vec{u}\|^2_{X(G_k^{j+1})}&=\sum_{0\leq i< j+1}2^{i-(j+1)}\sum_{G_{\alpha}^{i}\subset G_k^{j+1}}\|P_{\xi(G_{\alpha}^{i}), i-2\leq\cdot\leq i+2}\vec{u}\|^2_{U^2_{\triangle}(l^{2};G_{\alpha}^{i}\times\mathbb{R}^2)}
  +\sum_{i\geq j+1}\|P_{\xi(G_{k}^{j+1}),i-2\leq\cdot\leq i+2} \vec{u}\|^2_{U^2_{\triangle}(l^{2};G_{k}^{j+1}\times\mathbb{R}^2)}\\
  &\leq \big[ \|\vec{u}\|^2_{X(G_{2k}^{j})}+\|\vec{u}\|^2_{X(G_{2k+1}^{j})}\big].
\end{split}
\end{equation}
Finally, from the Definition \ref{de-y5.2}, we have
$$\|\vec{u}\|^2_{\tilde{X}_{k_{\ast}+1}([0,T])}\leq 2\|\vec{u}\|^2_{\tilde{X}_{k_{\ast}}([0,T])}.$$
Similarly we can show
$$\|\vec{u}\|^2_{\tilde{Y}_{k_{\ast}+1}([0,T])}\leq 2\|\vec{u}\|^2_{\tilde{Y}_{k_{\ast}}([0,T])}.$$
\end{proof}
Next we will derive the bootstrap estimates that will be used to complete the proof of Theorem \ref{th-y6.2}.  Fix $0 \leq j\leq k_0$ and $G_k^j \subset [0,T]$.  For $0 \leq i < j$, Duhamel's principle implies
\begin{equation}\label{eq-y5.14}
\begin{split}
 \|P_{\xi(G_{\alpha}^i),i-2\leq\cdot\leq i+2}\vec{u}\|_{U^2_{\triangle}(l^{2};G_{\alpha}^i)}&\lesssim \|P_{\xi(G_{\alpha}^i),i-2\leq\cdot\leq i+2}\vec{u}(t_{\alpha}^i)\|_{L_x^2l^{2}}\\
 + &\|\int_{t_{\alpha}^i}^{t} e^{i(t-\tau)\triangle} P_{\xi(G_{\alpha}^{i}),i-2\leq\cdot\leq i+2}
\mathbf{F}(\vec{u}(\tau))d\tau\|_{U_{\triangle}^2(l^{2};G_{\alpha}^{i})}.
\end{split}
\end{equation}
The free evolution term can be easily handled.
Choose $t_{\alpha}^i$ satisfying
$$\|P_{\xi(G_{\alpha}^i),i-2\leq\cdot\leq i+2}\vec{u}(t_{\alpha}^i)\|_{L_x^2l^{2}}=\inf_{t\in G_{\alpha}^i}\|P_{\xi(G_{\alpha}^i),i-2\leq\cdot\leq i+2}\vec{u}(t)\|_{L_x^2l^{2}}. $$
Then by \eqref{eq-y5.33} and Remark \ref{re-y5.1},
\begin{equation}\label{eq-y5.15}
\begin{split}
  &\sum_{0\leq i<j}2^{i-j} \sum_{G^i_{\alpha}\subset G^j_{k}} \|P_{\xi(G_{\alpha}^i),i-2\leq\cdot\leq i+2}\vec{u}(t_{\alpha}^i)\|^2_{L_x^2l^{2}} \\
  \lesssim& 2^{-j} \ep_3^{-1}\int_{G^j_{k}}\big( N(t)^3 + \ep_3\|\vec{u}(t)\|^4_{L_x^4l^{2}(\mathbb{R}^2\times \mathbb{Z})} \big)\sum_{0\leq i<j}\|P_{\xi(t),i-3\leq\cdot\leq i+3}\vec{u}(t)\|^2_{L_x^2l^{2}} dt\\
  \lesssim& 2^{-j} \ep_3^{-1} \|\vec{u}\|^2_{L_t^{\infty}L_x^2l^{2}} \int_{G^j_{k}}\big( N(t)^3 + \ep_3\|\vec{u}(t)\|^4_{L_x^4l^{2}(\mathbb{R}^2\times \mathbb{Z})} \big)dt\\
  \lesssim& 1.
\end{split}
\end{equation}
For $i \geq j$ simply take $t_{k}^j = t_0$, where $t_0$ is the left endpoint of $G^j_k$. Then
\begin{equation}\label{eq-y5.16}
 \sum_{i\geq j} \|P_{\xi(G_{k}^j),i-2\leq\cdot\leq i+2}\vec{u}(t_0)\|^2_{L_x^2l^{2}}\lesssim \|\vec{u}(t_0)\|^2_{L_x^2l^{2}}\lesssim 1.
\end{equation}
Therefore,
\begin{equation}\label{eq-y5.17}
 \sum_{0\leq i<j}2^{i-j} \sum_{G^i_{\alpha}\subset G^j_{k}} \|P_{\xi(G_{\alpha}^i),i-2\leq\cdot\leq i+2}\vec{u}(t_{\alpha}^i)\|^2_{L_x^2l^{2}}+\sum_{i\geq j} \|P_{\xi(G_{k}^j),i-2\leq\cdot\leq i+2}\vec{u}(t_0)\|^2_{L_x^2l^{2}}\lesssim 1.
\end{equation}
So \eqref{eq-y5.14}, \eqref{eq-y5.17} and Definition \ref{de-y5.2} give
\begin{equation}\label{eq-y5.18}
\begin{split}
   \|\vec{u}\|^2_{X(G_k^{j})}&\lesssim 1+ \sum_{i\geq j}\big\|\int_{t_{k}^j}^{t} e^{i(t-\tau)\triangle} P_{\xi(G_{k}^{j}),i-2\leq\cdot\leq i+2}
\mathbf{F}(\vec{u}(\tau))d\tau \big\|^2_{U_{\triangle}^2(l^{2};G_{k}^{j})} \\
   + & \sum_{0\leq i\leq j} 2^{i-j}\sum_{G^{i}_{\alpha}\subset G^j_k} \big\|\int_{t_{\alpha}^i}^{t} e^{i(t-\tau)\triangle} P_{\xi(G_{\alpha}^{i}),i-2\leq\cdot\leq i+2}
\mathbf{F}(\vec{u}(\tau))d\tau \big\|^2_{U_{\triangle}^2(l^{2};G_{\alpha}^{i})}.
\end{split}
\end{equation}
Similarly, we can get
\begin{equation}\label{eq-y5.19}
\begin{split}
   \|\vec{u}\|^2_{Y(G_k^{j})}&\lesssim \ep_2^{3/2}+ \sum_{i\geq j;N(G^j_k)\leq 2^{i-5}\ep_3^{1/2}}\big\|\int_{t_{k}^j}^{t} e^{i(t-\tau)\triangle} P_{\xi(G_{k}^{j}),i-2\leq\cdot\leq i+2}
\mathbf{F}(\vec{u}(\tau))d\tau \big\|^2_{U_{\triangle}^2(l^{2};G_{k}^{j})} \\
   + & \sum_{0\leq i\leq j} 2^{i-j}\sum_{G^{i}_{\alpha}\subset G^j_k;N(G^{i}_{\alpha})\leq 2^{i-5}\ep_3^{1/2}} \big\|\int_{t_{\alpha}^i}^{t} e^{i(t-\tau)\triangle} P_{\xi(G_{\alpha}^{i}),i-2\leq\cdot\leq i+2}
\mathbf{F}(\vec{u}(\tau))d\tau \big\|^2_{U_{\triangle}^2(l^{2};G_{\alpha}^{i})}.
\end{split}
\end{equation}
Now let's perform some subtle analysis on the summation of the $U^2_{\triangle}(l^{2})$ norm of Duhamel term in \eqref{eq-y5.18}. Take the intervals $G^{i}_{\alpha}\subset G^j_k$ with $N(G^{i}_{\alpha})\geq 2^{i-5}\ep_3^{1/2}$. These intervals appear in \eqref{eq-y5.18} but not \eqref{eq-y5.19}, and their contribution to the summation of the $U^2_{\triangle}(l^{2})$ norm of Duhamel term is small. Notice first that there are at most two small intervals, call them $J_1$ and $J_2$ , that intersect $G^j_k$ but are not contained in $G^j_k$. By Lemma \ref{le-y6.2} and $U^2_{\triangle}(l^{2})\subset U^3_{\triangle}(l^{2})\subset L^3_tL^6_xh^{0}$, $\|\vec{u}\|_{L^3_tL^6_xh^{0}(J_l)}\lesssim 1$. Therefore by $l^1\subset l^2$,
\begin{equation}\label{eq-y5.20}
\begin{split}
  \sum_{0\leq i<j}2^{i-j}\sum_{G^{i}_{\alpha}\subset G^j_k}\|\mathbf{F}(\vec{u})\|^2_{L^1_tL^2_xl^2(G^{i}_{\alpha}\cap (J_1 \cup J_2))}&\lesssim \sum_{0\leq i<j}2^{i-j} \|\mathbf{F}(\vec{u})\|^2_{L^1_tL^2_xl^2(J_1 \cup J_2)}\\
  &\lesssim\|\mathbf{F}(\vec{u})\|^2_{L^1_tL^2_xl^2(J_1)}+\|\mathbf{F}(\vec{u})\|^2_{L^1_tL^2_xl^2(J_2)}\\
  &\lesssim \|\vec{u}\|^6_{L^3_tL^6_xl^2(J_1)}+\|\vec{u}\|^6_{L^3_tL^6_xl^2(J_2)}\lesssim 1.
\end{split}
\end{equation}
Next observe that \eqref{eq-y5.1} and $N(G^{i}_{\alpha})\geq 2^{i-5}\ep_3^{1/2}$ implies $N(t)\geq 2^{i-6}\ep_3^{1/2}$ for all $t\in G^{i}_{\alpha}$, so by Proposition \ref{pr-y5.3}(5), $L^1_tL^2_xh^{0}\subset DU^2_{\triangle}(l^{2})$, \eqref{eq-y5.20}, \eqref{eq-y5.34} and Definition \ref{de-y5.1}, we have
\begin{equation}\label{eq-y5.21}
\begin{split}
  & \sum_{0\leq i\leq j} 2^{i-j}\sum_{G^{i}_{\alpha}\subset G^j_k;N(G^{i}_{\alpha})\geq 2^{i-5}\ep_3^{1/2}} \big\|\int_{t_{\alpha}^i}^{t} e^{i(t-\tau)\triangle} P_{\xi(G_{\alpha}^{i}),i-2\leq\cdot\leq i+2}
  \mathbf{F}(\vec{u}(\tau))d\tau \big\|^2_{U_{\triangle}^2(l^{2};G_{\alpha}^{i})}\\
  \lesssim& \sum_{0\leq i\leq j} 2^{i-j}\sum_{G^{i}_{\alpha}\subset G^j_k;N(G^{i}_{\alpha})\geq 2^{i-5}\ep_3^{1/2}} \sum_{J_l\cap G^j_k \neq \varnothing} \big\|\int_{t_{\alpha}^i}^{t} e^{i(t-\tau)\triangle} P_{\xi(G_{\alpha}^{i}),i-2\leq\cdot\leq i+2}
  \mathbf{F}(\vec{u}(\tau))d\tau \big\|^2_{U_{\triangle}^2(l^{2};G_{\alpha}^{i}\cap J_l)}\\
  =& \sum_{0\leq i\leq j} 2^{i-j}\sum_{G^{i}_{\alpha}\subset G^j_k;N(G^{i}_{\alpha})\geq 2^{i-5}\ep_3^{1/2}} \sum_{J_l\cap G^j_k \neq \varnothing} \big\|P_{\xi(G_{\alpha}^{i}),i-2\leq\cdot\leq i+2}
  \mathbf{F}(\vec{u})\big\|^2_{DU_{\triangle}^2(l^{2};G_{\alpha}^{i}\cap J_l)}\\
  \lesssim&\sum_{0\leq i<j}2^{i-j}\sum_{G^{i}_{\alpha}\subset G^j_k;N(G^{i}_{\alpha})\geq 2^{i-5}\ep_3^{1/2}} \sum_{J_l\cap G^j_k \neq \varnothing } \|\mathbf{F}(\vec{u})\|^2_{L^1_tL^2_xl^2(G^{i}_{\alpha}\cap J_l)}\\
  \lesssim&1+ \sum_{0\leq i<j}2^{i-j} \big(\sum_{J_l\subset G_k^j;N(J_l)\geq 2^{i-6}\ep_3^{1/2}}\|\mathbf{F}(\vec{u})\|^2_{L^1_tL^2_xl^2(J_l)}\big)\\
  \lesssim&1+ \sum_{0\leq i<j}2^{i-j} \big(\sum_{J_l\subset G_k^j;N(J_l)\geq 2^{i-6}\ep_3^{1/2}}\|\vec{u}\|^6_{L^3_tL^6_xl^2(J_l)}\big)\\
  \lesssim &1+ \sum_{J_l\subset G_k^j} \sum_{0\leq i<j;2^{i}\leq2^6\ep_3^{-1/2}N(J_l)}2^{i-j}\\
  \lesssim &1+ 2^{-j}\ep_3^{-1/2}\sum_{J_l\subset G_k^j}N(J_l)\\
  \lesssim &1+ \ep_3^{1/2}\lesssim 1.
\end{split}
\end{equation}
Similarly, if $N(G^{j}_{k})\geq 2^{j-5}\ep_3^{1/2}$, then \eqref{eq-y5.1} implies that $N(t)\geq 2^{j-6}\ep_3^{1/2}$ for all $t\in G^{j}_{k}$. Combing \eqref{eq-y5.33}, we have $\int_{G^{j}_{k}} N(t)^3dt \leq 2^{j+1}\ep_3 <2^7N(t)$ for all $t\in G^{j}_{k}$, so $\int_{G^{j}_{k}} N(t)^3dt \lesssim \inf_{t\in G^{j}_{k}}N(t)$. This implies $\int_{G^{j}_{k}} N(t)^2dt \lesssim1$. Therefore by Minkowski inequality, \eqref{eq-y3.1} and Theorem \ref{th-y4.5}, for $G^j_k$ with $N(G^{j}_{k})\geq 2^{j-5}\ep_3^{1/2}$, we have
\begin{equation}\label{eq-y5.22}
  \sum_{i\geq j;N(G^{j}_{k})\geq 2^{j-5}\ep_3^{1/2}}\|P_{\xi(G_{k}^{j}),i-2\leq\cdot\leq i+2}
  \mathbf{F}(\vec{u})\|^2_{L^1_tL^2_xl^2(G^{j}_{k})}\lesssim \|
  \mathbf{F}(\vec{u})\|^2_{L^1_tL^2_xl^2(G^{j}_{k})}\lesssim \|\vec{u}\|^6_{L^3_tL^6_xl^2(G^{j}_{k})}\lesssim 1.
\end{equation}
Therefore,
\begin{equation}\label{eq-y5.23}
\begin{split}
   \|\vec{u}\|^2_{X(G_k^{j})}&\lesssim 1+ \sum_{i\geq j;N(G^{j}_{k})\leq 2^{j-5}\ep_3^{1/2}}\big\|\int_{t_{k}^j}^{t} e^{i(t-\tau)\triangle} P_{\xi(G_{k}^{j}),i-2\leq\cdot\leq i+2}
\mathbf{F}(\vec{u}(\tau))d\tau \big\|^2_{U_{\triangle}^2(l^{2};G_{k}^{j})} \\
   + & \sum_{0\leq i\leq j} 2^{i-j}\sum_{G^{i}_{\alpha}\subset G^j_k;N(G^{i}_{\alpha})\leq 2^{i-5}\ep_3^{1/2}} \big\|\int_{t_{\alpha}^i}^{t} e^{i(t-\tau)\triangle} P_{\xi(G_{\alpha}^{i}),i-2\leq\cdot\leq i+2}
\mathbf{F}(\vec{u}(\tau))d\tau \big\|^2_{U_{\triangle}^2(l^{2};G_{\alpha}^{i})}.
\end{split}
\end{equation}
Following the same argument, we can also show that
\begin{equation}\label{eq-y5.23'}
\begin{split}
    &\sum_{i\geq j;2^{j-10}\ep_3^{1/2} \leq N(G^{j}_{k})\leq 2^{j-5}\ep_3^{1/2}}\big\|\int_{t_{k}^j}^{t} e^{i(t-\tau)\triangle} P_{\xi(G_{k}^{j}),i-2\leq\cdot\leq i+2}
\mathbf{F}(\vec{u}(\tau))d\tau \big\|^2_{U_{\triangle}^2(l^{2};G_{k}^{j})} \\
   + & \sum_{0\leq i\leq j} 2^{i-j}\sum_{G^{i}_{\alpha}\subset G^j_k;2^{i-10}\ep_3^{1/2} \leq N(G^{i}_{\alpha})\leq 2^{i-5}\ep_3^{1/2}} \big\|\int_{t_{\alpha}^i}^{t} e^{i(t-\tau)\triangle} P_{\xi(G_{\alpha}^{i}),i-2\leq\cdot\leq i+2}
\mathbf{F}(\vec{u}(\tau))d\tau \big\|^2_{U_{\triangle}^2(l^{2};G_{\alpha}^{i})}\\
  \lesssim&\ep_2^{3/2}.
\end{split}
\end{equation}
Therefore, \eqref{eq-y5.23} and \eqref{eq-y5.19} become
\begin{equation}\label{eq-y5.23''}
\begin{split}
   \|\vec{u}\|^2_{X(G_k^{j})}&\lesssim 1+ \sum_{i\geq j;N(G^{j}_{k})\leq 2^{j-10}\ep_3^{1/2}}\big\|\int_{t_{k}^j}^{t} e^{i(t-\tau)\triangle} P_{\xi(G_{k}^{j}),i-2\leq\cdot\leq i+2}
\mathbf{F}(\vec{u}(\tau))d\tau \big\|^2_{U_{\triangle}^2(l^{2};G_{k}^{j})} \\
   + & \sum_{0\leq i\leq j} 2^{i-j}\sum_{G^{i}_{\alpha}\subset G^j_k;N(G^{i}_{\alpha})\leq 2^{i-10}\ep_3^{1/2}} \big\|\int_{t_{\alpha}^i}^{t} e^{i(t-\tau)\triangle} P_{\xi(G_{\alpha}^{i}),i-2\leq\cdot\leq i+2}
\mathbf{F}(\vec{u}(\tau))d\tau \big\|^2_{U_{\triangle}^2(l^{2};G_{\alpha}^{i})},
\end{split}
\end{equation}and
\begin{equation}\label{eq-y5.19'}
\begin{split}
   \|\vec{u}\|^2_{Y(G_k^{j})}&\lesssim \ep_2^{3/2}+ \sum_{i\geq j;N(G^j_k)\leq 2^{i-10}\ep_3^{1/2}}\big\|\int_{t_{k}^j}^{t} e^{i(t-\tau)\triangle} P_{\xi(G_{k}^{j}),i-2\leq\cdot\leq i+2}
\mathbf{F}(\vec{u}(\tau))d\tau \big\|^2_{U_{\triangle}^2(l^{2};G_{k}^{j})} \\
   + & \sum_{0\leq i\leq j} 2^{i-j}\sum_{G^{i}_{\alpha}\subset G^j_k;N(G^{i}_{\alpha})\leq 2^{i-10}\ep_3^{1/2}} \big\|\int_{t_{\alpha}^i}^{t} e^{i(t-\tau)\triangle} P_{\xi(G_{\alpha}^{i}),i-2\leq\cdot\leq i+2}
\mathbf{F}(\vec{u}(\tau))d\tau \big\|^2_{U_{\triangle}^2(l^{2};G_{\alpha}^{i})}.
\end{split}
\end{equation}

\begin{remark}
According to Lemma \ref{le-y5.8}, Remark \ref{re-y5.4} and the fact $L_t^{3/2}L_x^{6/5}l^{2} \subset DU^{2}_{\triangle}(l^{2})$,
\begin{equation*}
\begin{split}
   &\sum_{i\geq j,0\leq i\leq 11;N(G^{j}_{k})\leq 2^{j-5}\ep_3^{1/2}}\big\|\int_{t_{k}^j}^{t} e^{i(t-\tau)\triangle} P_{\xi(G_{k}^{j}),i-2\leq\cdot\leq i+2}
\mathbf{F}(\vec{u}(\tau))d\tau \big\|^2_{U_{\triangle}^2(l^{2};G_{k}^{j})} \\
   + & \sum_{0\leq i\leq 11} 2^{i-j}\sum_{G^{i}_{\alpha}\subset G^j_k;N(G^{i}_{\alpha})\leq 2^{i-5}\ep_3^{1/2}} \big\|\int_{t_{\alpha}^i}^{t} e^{i(t-\tau)\triangle} P_{\xi(G_{\alpha}^{i}),i-2\leq\cdot\leq i+2}
\mathbf{F}(\vec{u}(\tau))d\tau \big\|^2_{U_{\triangle}^2(l^{2};G_{\alpha}^{i})}\\
   \lesssim 1.
\end{split}
\end{equation*}
Therefore the summation over $i$ in the right side of \eqref{eq-y5.23''} and \eqref{eq-y5.19'}can start from $i>11$. In the following we will not mention this unless necessary.
\end{remark}
Now, we can give the key bootstrap estimate below, which suffices to establish the long time strichartz estimates.
\begin{theorem}\label{th-y5.10}
 \begin{equation}\label{eq-y5.24}
\begin{split}
   & \sum_{i\geq j;N(G^{j}_{k})\leq 2^{i-10}\ep_3^{1/2}}\big\|\int_{t_{k}^j}^{t} e^{i(t-\tau)\triangle} P_{\xi(G_{k}^{j}),i-2\leq\cdot\leq i+2}
\mathbf{F}(\vec{u}(\tau))d\tau \big\|^2_{U_{\triangle}^2(l^{2};G_{k}^{j})} \\
   + & \sum_{0\leq i\leq j} 2^{i-j}\sum_{G^{i}_{\alpha}\subset G^j_k;N(G^{i}_{\alpha})\leq 2^{i-10}\ep_3^{1/2}} \big\|\int_{t_{\alpha}^i}^{t} e^{i(t-\tau)\triangle} P_{\xi(G_{\alpha}^{i}),i-2\leq\cdot\leq i+2}
\mathbf{F}(\vec{u}(\tau))d\tau \big\|^2_{U_{\triangle}^2(l^{2};G_{\alpha}^{i})}\\
 \lesssim &\ep_2^{1/3} \|\vec{u}\|^{5/3}_{\tilde{X}_j([0,T])} \|\vec{u}\|^{2}_{\tilde{Y}_j([0,T])} + \ep_2^{2} \|\vec{u}\|^{2}_{\tilde{Y}_j([0,T])} + \|\vec{u}\|^{4}_{\tilde{Y}_j([0,T])} \big( 1 + \|\vec{u}\|^{8}_{\tilde{X}_j([0,T])} \big).
\end{split}
\end{equation}
\end{theorem}

Indeed, to prove Theorem \ref{th-y6.2}, we perform a bootstrap argument. If
$$\|\vec{u}\|^{2}_{\tilde{X}_{k_{\ast}}([0,T])}\leq C_0,$$ and
$$\|\vec{u}\|^{2}_{\tilde{Y}_{k_{\ast}}([0,T])}\leq C\ep_2^{3/2}\leq \ep_2.$$
Then by Lemma \ref{le-y5.8},
$$\|\vec{u}\|^{2}_{\tilde{X}_{k_{\ast}+1}([0,T])}\leq 2C_0,$$ and
$$\|\vec{u}\|^{2}_{\tilde{Y}_{k_{\ast}+1}([0,T])}\leq 2\ep_2.$$
Meanwhile by \eqref{eq-y5.23''}, \eqref{eq-y5.19'},  \eqref{eq-y5.24} and Definition \ref{de-y5.2}, we have
\begin{equation}\label{eq-y5.25}
\|\vec{u}\|_{\tilde{X}_{k_{\ast}+1}([0,T])}\leq C (1+ \ep_2^{2/3}(2C_0)^{5/6} + \ep_2^{3/2}+\ep_2(1+2C_0)^8),
\end{equation}and
\begin{equation}\label{eq-y5.25}
\|\vec{u}\|_{\tilde{Y}_{k_{\ast}+1}([0,T])}\leq C (\ep_2^{3/4}+ \ep_2^{2/3}(2C_0)^{5/6} + \ep_2^{3/2}+\ep_2(1+2C_0)^8).
\end{equation}
Take $C_{0}=2^6 C,\ep_2>0$ sufficiently small, then we make the bootstrap closed. This implies
$$\|\vec{u}\|_{\tilde{X}_{k_{\ast}+1}([0,T])}\leq C_0,$$ and
$$\|\vec{u}\|_{\tilde{Y}_{k_{\ast}+1}([0,T])}\leq \ep_2^{1/2},$$
which yield Theorem \ref{th-y6.2} by induction on $k_{\ast}$.
\end{proof}

\begin{proof}[Proof of Theorem \ref{th-y5.10}]
Let's recall $\mathbf{F} (\vec{u}(t)): = \{ F_j(\vec{u}(t)) \}_{j\in\mathbb{Z}}: =\{\sum\limits_{\mathcal{R}(j)}u_{j_1}\bar{u}_{j_2}u_{j_3}\}_{j\in\mathbb{Z}} =:\sum^{\rightarrow}\limits_{\mathcal{R}(j)}u_{j_1}\bar{u}_{j_2}u_{j_3}$ at first. Because we are projecting $\mathbf{F} (\vec{u}(t))$ to frequencies at $\approx 2^{i}$, there must be at least one factor of each terms in $\sum\limits_{\mathcal{R}(j)}u_{j_1}\bar{u}_{j_2}u_{j_3}$ with frequency higher than $2^{i-5}$. This implies
\begin{equation}\label{eq-y5.27}
  \begin{split}
     P_{\xi( G_{\alpha}^{i}), i-2 \leq \cdot \leq i+2} \left(\sum\limits_{\mathcal{R}(j)}u_{j_1}\bar{u}_{j_2}u_{j_3}\right)
     &= P_{\xi( G_{\alpha}^{i}), i-2 \leq \cdot \leq i+2}O \left( \sum\limits_{\mathcal{R}(j)} P_{\xi( G_{\alpha}^{i}), \leq i-8} u_{j_1} P_{\xi( G_{\alpha}^{i}), \leq i-8} \bar{u}_{j_2} P_{\xi( G_{\alpha}^{i}), i-5\leq\cdot\leq i+5} u_{j_3}\right)  \\
     &\quad +  P_{\xi( G_{\alpha}^{i}), i-2 \leq \cdot \leq i+2} \left[ O\big( \sum\limits_{\mathcal{R}(j)} (P_{\xi( G_{\alpha}^{i}), \geq i-5} u_{j_1}) (P_{\xi( G_{\alpha}^{i}), > i-8} \bar{u}_{j_2}) u_{j_3}\big)\right]\\
     &=:(a)+(b),
   \end{split}
\end{equation}where the ``$O$" represents the different frequencies will be located in different $u_{j_l}$s', $l=1,2,3$. However, their estimates are the same. So we denote them as ``$O$".

Following Dodson's argument in \cite{D}, the terms \eqref{eq-y5.27}(b) with two high frequency factors can be handled easily. Due to \eqref{eq-y5.2}, each high frequency factor in these terms can be intuitively thought ``small" in some sense. So let's deal with the terms \eqref{eq-y5.27}(b) at first.

\begin{theorem}\label{th-y5.13}
For a fixed $G_k^j \subset [0,T]$, $j>0$,
\begin{equation}
\begin{split}
   & \sum_{i\geq j;N(G^{j}_{k})\leq 2^{i-10}\ep_3^{1/2}}\big\|\int_{t_{k}^j}^{t} e^{i(t-\tau)\triangle} P_{\xi(G_{k}^{j}),i-2\leq\cdot\leq i+2}
\sum^{\rightarrow}\limits_{\mathcal{R}(j)} (P_{\xi( G_{\alpha}^{i}), \geq i-5} u_{j_1}) (P_{\xi( \tau), \geq i-10} \bar{u}_{j_2}) u_{j_3}d\tau \big\|^2_{U_{\triangle}^2(l^{2};G_{k}^{j})} +\\
   & \sum_{0\leq i\leq j} 2^{i-j}\sum_{G^{i}_{\alpha}\subset G^j_k;N(G^{i}_{\alpha})\leq 2^{i-10}\ep_3^{1/2}} \big\|\int_{t_{\alpha}^i}^{t} e^{i(t-\tau)\triangle} P_{\xi(G_{\alpha}^{i}),i-2\leq\cdot\leq i+2}
\sum^{\rightarrow}\limits_{\mathcal{R}(j)} (P_{\xi( G_{\alpha}^{i}), \geq i-5} u_{j_1}) (P_{\xi(\tau), \geq i-10} \bar{u}_{j_2}) u_{j_3}d\tau \big\|^2_{U_{\triangle}^2(l^{2};G_{\alpha}^{i})}\\
 \lesssim &\ep_2^{1/3} \|\vec{u}\|^{5/3}_{\tilde{X}_j([0,T])} \|\vec{u}\|^{2}_{\tilde{Y}_j([0,T])} .
\end{split}
\end{equation}
\end{theorem}

\begin{proof}
 By Proposition \ref{pr-y5.3} (3), for any $\vec{v}=\{v_j\}_{j \in \mathbb{Z}}$, $\hat{v}_j$ is supported on $\{\xi: 2^{i-2} \leq|\xi - \xi(G_{\alpha}^{i})|\leq 2^{i+2}\}$ for every $j\in \mathbb{Z}$ and $\|\vec{v}\|_{V^2_{\triangle}(l^{2};G_{\alpha}^{i})}=1$. It follows from \eqref{eq-y2.6}, Proposition \ref{pr-y5.5}, mass conservation, Proposition \ref{pr-y5.3} (2)(6), Lemma \ref{le-y5.7}, $N(G^{i}_{\alpha})\leq 2^{i-10}\ep_3^{1/2}$ and \eqref{eq-y5.2} that
 \begin{equation}\label{eq-y5.31}
\begin{split}
   & \int_{G_{\alpha}^{i}} \<\vec{v},\sum^{\rightarrow}\limits_{\mathcal{R}(j)} (P_{\xi( G_{\alpha}^{i}), \geq i-5} u_{j_1}) (P_{\xi(\tau), \geq i-10} \bar{u}_{j_2}) (u_{j_3})\> (\tau) d\tau \\
  =& \int_{G_{\alpha}^{i}} \int_{\mathbb{R}^2} \big[\sum_{j \in \mathbb{Z}} v_j \sum_{\mathcal{R}(j)} (P_{\xi( G_{\alpha}^{i}), \geq i-5} u_{j_1}) (P_{\xi(\tau), \geq i-10} \bar{u}_{j_2}) (u_{j_3})\big](\tau,x)dx d\tau \\
  \leq& \sum_{l\geq i-5}\int_{G_{\alpha}^{i}} \int_{\mathbb{R}^2}\|\vec{v}\|_{l^{2}} \|P_{\xi(G_{\alpha}^{i}),l} \vec{u}\|_{l^{2}} \|P_{\xi(\tau),\geq i-10} \vec{u} \|_{l^{2}} \|\vec{u}\|_{l^{2}}(\tau,x)dx d\tau\\
  \leq& \sum_{l\geq i-5} \left\|\|\vec{v}\|_{l^{2}} \|P_{\xi(G_{\alpha}^{i}),l} \vec{u}\|_{l^{2}} \right\|^{1/2}_{L_t^{\frac{5}{2}}L_x^{\frac{5}{3}}(G_{\alpha}^{i} \times \mathbb{R}^2)} \|P_{\xi(G_{\alpha}^{i}),l} \vec{u}\|^{1/2}_{L_t^{5/2} L_x^{10} l^{2}(G_{\alpha}^{i} \times \mathbb{R}^2)} \|\vec{v}\|^{1/2}_{L_t^{5/2}L_x^{10}l^{2}(G_{\alpha}^{i} \times \mathbb{R}^2)}\\
  & \quad \times \|P_{\xi(\tau),\geq i-10} \vec{u}\|_{L_t^{\frac{5}{2}}L_x^{10} l^{2}(G_{\alpha}^{i} \times \mathbb{R}^2)} \|\vec{u} \|_{L_t^{\infty}L_x^2 l^{2}(G_{\alpha}^{i} \times \mathbb{R}^2)}\\
  \lesssim& \sum_{l\geq i-5}2^{\frac{i-l}{5}} \|P_{\xi(G_{\alpha}^{i}),l} \vec{u}\|_{U^2_{\triangle}(l^{2};G_{\alpha}^{i})} \|\vec{v}\|_{V^2_{\triangle}(l^{2};G_{\alpha}^{i})} \|P_{\xi(\tau),\geq i-10} \vec{u}\|_{L_t^{\frac{5}{2}}L_x^{10} l^{2}(G_{\alpha}^{i} \times \mathbb{R}^2)}\\
  \lesssim& \sum_{l\geq i-5}2^{\frac{i-l}{5}} \|P_{\xi(G_{\alpha}^{i}),l} \vec{u}\|_{U^2_{\triangle}(l^{2};G_{\alpha}^{i})} \|P_{\xi(\tau),\geq i-10} \vec{u}\|^{1/6}_{L_t^{\infty}L_x^{2} l^{2}(G_{\alpha}^{i} \times \mathbb{R}^2)} \|P_{\xi(\tau),\geq i-10} \vec{u}\|^{5/6}_{L_t^{25/12}L_x^{50} l^{2}(G_{\alpha}^{i} \times \mathbb{R}^2)}\\
  \lesssim& \ep_2^{1/6} \|\vec{u}\|^{5/6}_{\tilde{X}_i([0,T])} \sum_{l\geq i-5}2^{\frac{i-l}{5}} \|P_{\xi(G_{\alpha}^{i}),l} \vec{u}\|_{U^2_{\triangle}(l^{2};G_{\alpha}^{i})}\\
  \lesssim& \ep_2^{1/6} \|\vec{u}\|^{5/6}_{\tilde{X}_j([0,T])} \big[\sum_{l\geq i-5}2^{\frac{i-l}{5}} \|P_{\xi(G_{\alpha}^{i}),l} \vec{u}\|^2_{U^2_{\triangle}(l^{2};G_{\alpha}^{i})}\big]^{1/2}.
\end{split}
\end{equation}The last inequality follows from Cauchy-Schwarz inequality with $(\sum_{l\geq i-5}2^{\frac{i-l}{5}})^{1/2}\lesssim1$.

Now for any $0\leq l \leq j$, $G^j_k$ overlaps $2^{j-l}$ intervals $G^l_{\beta}$ and for $0 \leq i \leq l$, each $G^l_{\beta}$ overlaps $2^{l-i}$ intervals $G^i_{\alpha}$. Additionally, each $G^i_{\alpha}$ is the subset of one $G^l_{\beta}$. We can put the summation below into different groups according to $l\geq j$ and $0\leq l\leq j$. Using Fubini's Theorem, we can change the order of summation of ``$i$" and ``$l$". Therefore,
\begin{equation}\label{eq-y5.32}
\begin{split}
    & \sum_{0\leq i\leq j} 2^{i-j}\sum_{G^{i}_{\alpha}\subset G^j_k;N(G^{i}_{\alpha})\leq 2^{i-10}\ep_3^{1/2}} \big[\sum_{l\geq i-5}2^{\frac{i-l}{5}} \|P_{\xi(G_{\alpha}^{i}),l} \vec{u}\|^2_{U^2_{\triangle}(l^{2};G_{\alpha}^{i})}\big] \\
    \lesssim&  \sum_{0\leq i\leq j} 2^{i-j}\sum_{G^{l}_{\beta}\subset G^j_k;N(G^{l}_{\beta})\leq 2^{l-5}\ep_3^{1/2}} \sum_{G_{\alpha}^{i} \subset G^{l}_{\beta}} \big[\sum_{i-5\leq l\leq j}2^{\frac{i-l}{5}} \|P_{\xi(G_{\beta}^{l}),l-2\leq \cdot\leq l+2} \vec{u}\|^2_{U^2_{\triangle}(l^{2};G_{\beta}^{l})}\big]\\
    &+ \sum_{0\leq i\leq j} 2^{i-j}\sum_{G^j_k;N(G^{j}_{k})\leq 2^{l-5}\ep_3^{1/2}} \sum_{G_{\alpha}^{i} \subset G^{j}_{k}} \big[\sum_{l\geq j}2^{\frac{i-l}{5}} \|P_{\xi(G_{k}^{j}),l-2\leq \cdot\leq l+2} \vec{u}\|^2_{U^2_{\triangle}(l^{2};G_{k}^{j})}\big]\\
    \lesssim&  \sum_{0\leq i\leq j} 2^{i-j}\sum_{G^{l}_{\beta}\subset G^j_k;N(G^{l}_{\beta})\leq 2^{l-5}\ep_3^{1/2}} 2^{l-i} \big[\sum_{i-5\leq l\leq j}2^{\frac{i-l}{5}} \|P_{\xi(G_{\beta}^{l}),l-2\leq \cdot\leq l+2} \vec{u}\|^2_{U^2_{\triangle}(l^{2};G_{\beta}^{l})}\big]\\
    &+ \sum_{0\leq i\leq j} 2^{i-j}\sum_{G^j_k;N(G^{j}_{k})\leq 2^{l-5}\ep_3^{1/2}} 2^{j-i} \big[\sum_{l\geq j}2^{\frac{i-l}{5}} \|P_{\xi(G_{k}^{j}),l-2\leq \cdot\leq l+2} \vec{u}\|^2_{U^2_{\triangle}(l^{2};G_{k}^{j})}\big]\\
    \lesssim& \sum_{0\leq l\leq j} 2^{l-j}\sum_{G^{l}_{\beta}\subset G^j_k;N(G^{l}_{\beta})\leq 2^{l-5}\ep_3^{1/2}} \big( \|P_{\xi(G_{\beta}^{l}),l-2\leq \cdot\leq l+2} \vec{u}\|^2_{U^2_{\triangle}(l^{2};G_{\beta}^{l})}\big)(\sum_{0\leq i\leq l+5}2^{\frac{i-l}{5}})\\
    &+  \sum_{l\geq j;N(G^{j}_{k})\leq 2^{l-5}\ep_3^{1/2}} \big( \|P_{\xi(G_{k}^{j}),l-2\leq \cdot\leq l+2} \vec{u}\|^2_{U^2_{\triangle}(l^{2};G_{k}^{j})} \big) (\sum_{0\leq i\leq l}2^{\frac{i-l}{5}})\\
   \lesssim& \|\vec{u}\|^2_{\tilde{Y}_j([0,T])}.
\end{split}
\end{equation}
Similarly,
\begin{equation}\label{eq-y5.33'}
 \sum_{i\geq j;N(G^{j}_{k})\leq 2^{i-10}\ep_3^{1/2}} \big[\sum_{l\geq i-5}2^{\frac{i-l}{5}} \|P_{\xi(G_{k}^{j}),l} \vec{u}\|^2_{U^2_{\triangle}(l^{2};G_{\alpha}^{i})}\big]\lesssim \|\vec{u}\|^2_{\tilde{Y}_j([0,T])},
\end{equation}
which together with \eqref{eq-y5.31} and \eqref{eq-y5.32}, implies Theorem \ref{th-y5.13}.

\end{proof}

Now we take \eqref{eq-y5.27}(a) into consideration.
\begin{theorem}\label{th-y5.14}
  For any $0\leq i \leq j$, $G^i_{\alpha} \subset G^j_{k}$, $N(G^{i}_{\alpha})\leq 2^{i-10}\ep_3^{1/2}$,
\begin{equation}\label{eq-y5.38}
\begin{split}
   & \left\| \int^t_{t^i_{\alpha}} e^{i(t-\tau)\triangle} P_{\xi( G_{\alpha}^{i}), i-2 \leq \cdot \leq i+2} \big( \sum^{\rightarrow}\limits_{\mathcal{R}(j)} P_{\xi(\tau), \leq i-10} u_{j_1} P_{\xi(\tau), \leq i-10} \bar{u}_{j_2} P_{\xi( G_{\alpha}^{i}), i-5\leq\cdot\leq i+5} u_{j_3}\big)d\tau \right\|_{U^2_{\triangle}(l^{2};G_{\alpha}^{i})} \\
   \lesssim & \|P_{\xi( G_{\alpha}^{i}), i-5 \leq \cdot \leq i+5}\vec{u}\|_{U^2_{\triangle}(l^{2};G_{\alpha}^{i})} \big[\ep_2 +\|\vec{u}\|_{\tilde{Y}_i([0,T])} (1+\|\vec{u}\|_{\tilde{X}_i([0,T])})^{4} \big].
\end{split}
\end{equation}
And, for $i\geq j$, $N(G^{j}_{k})\leq 2^{i-10}\ep_3^{1/2}$,
\begin{equation}\label{eq-y5.39}
\begin{split}
   & \left\| \int^t_{t^j_{k}} e^{i(t-\tau)\triangle} P_{\xi( G_{k}^{j}), i-2 \leq \cdot \leq i+2} \big( \sum^{\rightarrow}\limits_{\mathcal{R}(j)} P_{\xi(\tau), \leq i-10} u_{j_1} P_{\xi(\tau), \leq i-10} \bar{u}_{j_2} P_{\xi( G_{k}^{j}), i-5\leq\cdot\leq i+5} u_{j_3}\big)d\tau \right\|_{U^2_{\triangle}(l^{2};G_{k}^{j})} \\
   \lesssim &2^{\frac{3(j-i)}{4}} \|P_{\xi( G_{k}^{j}), i-5 \leq \cdot \leq i+5}\vec{u}\|_{U^2_{\triangle}(l^{2};G_{k}^{j})} \big[\ep_2 +\|\vec{u}\|_{\tilde{Y}_j([0,T])} (1+\|\vec{u}\|_{\tilde{X}_j([0,T])})^{4} \big].
\end{split}
\end{equation}
\end{theorem}
Theorem \ref{th-y5.14} and Theorem \ref{th-y5.13} yield Theorem \ref{th-y5.10} by summation over ``$i$" in the same way as \eqref{eq-y5.38} and \eqref{eq-y5.39}.

\end{proof}

\begin{proof}[Proof of Theorem \ref{th-y5.14}]
Since each term in \eqref{eq-y5.27}(a) has only one high frequency factor and two low frequency factors, the proof of Theorem \ref{th-y5.14} is considerably involved and will occupy the remainder of this section. The proof will focus on \eqref{eq-y5.38}, as the proof of \eqref{eq-y5.39} is nearly identical.

For a given $G^i_{\alpha}$ there are at most two small intervals $J_1$, $J_2$ that overlap $G^i_{\alpha}$ but are not contained in $G^i_{\alpha}$. Let $\tilde{G}^i_{\alpha}=G^i_{\alpha}\backslash (J_1\cup J_2)$. Then by Proposition \ref{pr-y5.3} (2)(3)(5)(6),
\begin{equation}\label{eq-y5.40}
\begin{split}
   & \left\| \int^t_{t^i_{\alpha}} e^{i(t-\tau)\triangle} P_{\xi( G_{\alpha}^{i}), i-2 \leq \cdot \leq i+2} \big( \sum^{\rightarrow}\limits_{\mathcal{R}(j)} P_{\xi(\tau), \leq i-10} u_{j_1} P_{\xi(\tau), \leq i-10} \bar{u}_{j_2} P_{\xi( G_{\alpha}^{i}), i-5\leq\cdot\leq i+5} u_{j_3}\big)d\tau \right\|_{U^2_{\triangle}(l^{2};G_{\alpha}^{i})} \\
   \lesssim &\left\| \int^t_{t^i_{\alpha}} e^{i(t-\tau)\triangle} P_{\xi( G_{\alpha}^{i}), i-2 \leq \cdot \leq i+2} \big( \sum^{\rightarrow}\limits_{\mathcal{R}(j)} P_{\xi(\tau), \leq i-10} u_{j_1} P_{\xi(\tau), \leq i-10} \bar{u}_{j_2} P_{\xi( G_{\alpha}^{i}), i-5\leq\cdot\leq i+5} u_{j_3}\big)d\tau \right\|_{U^2_{\triangle}(l^{2};\tilde{G}_{\alpha}^{i})} \\
   &\quad +\big\|\sum^{\rightarrow}\limits_{\mathcal{R}(j)} P_{\xi(\tau), \leq i-10} u_{j_1} P_{\xi(\tau), \leq i-10} \bar{u}_{j_2} P_{\xi( G_{\alpha}^{i}), i-5\leq\cdot\leq i+5} u_{j_3}\big\|_{L_{t,x}^{4/3}l^{2}(J_1\cap G_{\alpha}^{i})}\\
   &\quad +\big\|\sum^{\rightarrow}\limits_{\mathcal{R}(j)} P_{\xi(\tau), \leq i-10} u_{j_1} P_{\xi(\tau), \leq i-10} \bar{u}_{j_2} P_{\xi( G_{\alpha}^{i}), i-5\leq\cdot\leq i+5} u_{j_3}\big\|_{L_{t,x}^{4/3}l^{2}(J_2\cap G_{\alpha}^{i})}.
\end{split}
\end{equation}
We may assume $t^i_{\alpha}\in \tilde{G}_{\alpha}^{i}$. If $t^i_{\alpha}\notin \tilde{G}_{\alpha}^{i}$, then we can move $t^i_{\alpha}$ into $\tilde{G}_{\alpha}^{i}$ at a cost of
\begin{equation}\label{eq-y5.41}
\begin{split}
   &\big\|\sum^{\rightarrow}\limits_{\mathcal{R}(j)} P_{\xi(\tau), \leq i-10} u_{j_1} P_{\xi(\tau), \leq i-10} \bar{u}_{j_2} P_{\xi( G_{\alpha}^{i}), i-5\leq\cdot\leq i+5} u_{j_3}\big\|_{L_{t,x}^{4/3}l^{2}(J_1\cap G_{\alpha}^{i})}\\
   +& \big\|\sum^{\rightarrow}\limits_{\mathcal{R}(j)} P_{\xi(\tau), \leq i-10} u_{j_1} P_{\xi(\tau), \leq i-10} \bar{u}_{j_2} P_{\xi( G_{\alpha}^{i}), i-5\leq\cdot\leq i+5} u_{j_3}\big\|_{L_{t,x}^{4/3}l^{2}(J_2\cap G_{\alpha}^{i})}.
\end{split}
\end{equation}
Indeed, if $t^i_{\alpha}\notin \tilde{G}_{\alpha}^{i}$,  we suppose without loss of generality that $t^i_{\alpha}\in J_1$ and let $\tilde{t}^i_{\alpha}$ be the left endpoint of $\tilde{G}_{\alpha}^{i}$. By Strichartz estimate Proposition \ref{pr-y3.1},
\begin{equation}\label{eq-y5.41'}
\begin{split}
   & \left\| \int^{\tilde{t}^i_{\alpha}}_{t^i_{\alpha}} e^{i(\tilde{t}^i_{\alpha}-\tau)\triangle} P_{\xi( G_{\alpha}^{i}), i-2 \leq \cdot \leq i+2} \big( \sum^{\rightarrow}\limits_{\mathcal{R}(j)} P_{\xi(\tau), \leq i-10} u_{j_1} P_{\xi(\tau), \leq i-10} \bar{u}_{j_2} P_{\xi( G_{\alpha}^{i}), i-5\leq\cdot\leq i+5} u_{j_3}\big)d\tau \right\|_{L_x^2 l^{2}} \\
   \lesssim&\big\|\sum^{\rightarrow}\limits_{\mathcal{R}(j)} P_{\xi(\tau), \leq i-10} u_{j_1} P_{\xi(\tau), \leq i-10} \bar{u}_{j_2} P_{\xi( G_{\alpha}^{i}), i-5\leq\cdot\leq i+5} u_{j_3}\big\|_{L_{t,x}^{4/3}l^{2}(J_1\cap G_{\alpha}^{i})}.
\end{split}
\end{equation}
Then, for $J_l\in \{J_1,J_2\}$, by the bilinear estimate Proposition \ref{pr-y5.4} and \ref{pr-y5.5}, Proposition \ref{pr-y5.3} (6), $N(t)\leq 2^{i-5}\ep_3^{1/2}$ on $G^i_{\alpha}$, \eqref{eq-y5.1} and \eqref{eq-y5.2}, Lemma \ref{le-y6.2} we have
\begin{equation}\label{eq-y5.42}
\begin{split}
   &\big\|\sum^{\rightarrow}\limits_{\mathcal{R}(j)} P_{\xi(\tau), \leq i-10} u_{j_1} P_{\xi(\tau), \leq i-10} \bar{u}_{j_2} P_{\xi( G_{\alpha}^{i}), i-5\leq\cdot\leq i+5} u_{j_3}\big\|_{L_{t,x}^{4/3}l^{2}(J_l\cap G_{\alpha}^{i})}\\
   \lesssim&\big\| \| P_{\xi(J_l), \leq \ep_3^{-1/4}N(J_l)} \vec{u}\|_{l^{2}} \|P_{\xi( G_{\alpha}^{i}), i-5\leq\cdot\leq i+5} \vec{u}\|_{l^{2}} \big\|_{L_{t,x}^{2}(J_l\cap G_{\alpha}^{i})} \|P_{\xi(\tau), \leq i-10} \vec{u}\|_{L_{t,x}^{4}l^{2}(J_l\cap G_{\alpha}^{i})}\\
   +& \| P_{\xi(J_l), \geq \ep_3^{-1/4}N(J_l)} \vec{u}\|_{L_t^{\infty}L^2_x l^{2}(J_l\cap G^i_{\alpha})} \|P_{\xi( G_{\alpha}^{i}), i-5\leq\cdot\leq i+5} \vec{u}\|_{L_t^{8/3}L^8_x l^{2}(J_l\cap G_{\alpha}^{i})} \|P_{\xi(\tau), \leq i-10} \vec{u}\|_{L_t^{8/3}L^8_x l^{2}(J_l\cap G_{\alpha}^{i})}\\
   \lesssim& \sum_{2^k \leq \ep_3^{-1/4}N(J_l)} 2^{\frac{k-i}{2}} \|P_{\xi( G_{\alpha}^{i}), i-5\leq\cdot\leq i+5} \vec{u}\|_{U^2_{\triangle} (l^{2};G_{\alpha}^{i})} \|P_{\xi(J_l), k} \vec{u}\|_{U^2_{\triangle} (l^{2};J_l)}+\ep_2 \|P_{\xi( G_{\alpha}^{i}), i-5\leq\cdot\leq i+5} \vec{u}\|_{U^2_{\triangle} (l^{2};G_{\alpha}^{i})}\\
   \lesssim& \ep_2 \|P_{\xi( G_{\alpha}^{i}), i-5\leq\cdot\leq i+5} \vec{u}\|_{U^2_{\triangle} (l^{2};G_{\alpha}^{i})}.
\end{split}
\end{equation}
Therefore, at the price of \eqref{eq-y5.42} we have now simplified to a situation in which $G_{\alpha}^{i}$ is the union of a bunch of small intervals $J_l$'s.

Now by Lemma \ref{le-y4.5},
\begin{align}\label{eq-y5.43}
   & \left\| \int^t_{t^i_{\alpha}} e^{i(t-\tau)\triangle} P_{\xi( G_{\alpha}^{i}), i-2 \leq \cdot \leq i+2} \big( \sum^{\rightarrow}\limits_{\mathcal{R}(j)} P_{\xi(\tau), \leq i-10} u_{j_1} P_{\xi(\tau), \leq i-10} \bar{u}_{j_2} P_{\xi( G_{\alpha}^{i}), i-5\leq\cdot\leq i+5} u_{j_3}\big)d\tau \right\|_{U^2_{\triangle}(l^{2};G_{\alpha}^{i})} \\ \label{eq-y5.44}
   \lesssim & \sum_{0\leq l_2\leq i-10}\big( \sum_{J_l\subset G^{i}_{\alpha};N(J_l)\geq \ep_3^{1/2}2^{l_{2}-5}} \big\| P_{\xi( G_{\alpha}^{i}), i-2 \leq \cdot \leq i+2} \big( \sum^{\rightarrow}\limits_{\mathcal{R}(j)} P_{\xi(t), l_2} u_{j_1} P_{\xi(t), \leq l_2} \bar{u}_{j_2} P_{\xi( G_{\alpha}^{i}), i-5\leq\cdot\leq i+5} u_{j_3}\big) \big\|^2_{DU^{2}_{\triangle}(l^{2};J_l)} \big)^{1/2}\\ \label{eq-y5.45}
   +&\sum_{0\leq l_2\leq i-10}  \sum_{J_l\subset G^{i}_{\alpha};N(J_l)\geq \ep_3^{1/2}2^{l_{2}-5}} \big\|\int_{J_l} e^{-it\triangle} P_{\xi( G_{\alpha}^{i}), i-2 \leq \cdot \leq i+2} \big( \sum^{\rightarrow}\limits_{\mathcal{R}(j)} P_{\xi(t), l_2} u_{j_1} P_{\xi(t), \leq l_2} \bar{u}_{j_2} P_{\xi( G_{\alpha}^{i}), i-5\leq\cdot\leq i+5} u_{j_3}\big)dt \big\|_{L_x^2l^{2}}\\ \label{eq-y5.46}
   + & \sum_{0\leq l_2\leq i-10}\big( \sum_{G^{l_2}_{\beta}\subset G^{i}_{\alpha};N(G^{l_2}_{\beta})\leq \ep_3^{1/2}2^{l_{2}-5}} \big\| P_{\xi( G_{\alpha}^{i}), i-2 \leq \cdot \leq i+2} \big( \sum^{\rightarrow}\limits_{\mathcal{R}(j)} P_{\xi(t), l_2} u_{j_1} P_{\xi(t), \leq l_2} \bar{u}_{j_2} P_{\xi( G_{\alpha}^{i}), i-5\leq\cdot\leq i+5} u_{j_3}\big) \big\|^2_{DU^{2}_{\triangle}(l^{2};G^{l_2}_{\beta})} \big)^{1/2}\\ \label{eq-y5.47}
   +&\sum_{0\leq l_2\leq i-10}  \sum_{G^{l_2}_{\beta}\subset G^{i}_{\alpha};N(G^{l_2}_{\beta})\leq \ep_3^{1/2}2^{l_{2}-5}} \big\|\int_{G^{l_2}_{\beta}} e^{-it\triangle} P_{\xi( G_{\alpha}^{i}), i-2 \leq \cdot \leq i+2} \big( \sum^{\rightarrow}\limits_{\mathcal{R}(j)} P_{\xi(t), l_2} u_{j_1} P_{\xi(t), \leq l_2} \bar{u}_{j_2} P_{\xi( G_{\alpha}^{i}), i-5\leq\cdot\leq i+5} u_{j_3}\big)dt \big\|_{L_x^2l^{2}}.
\end{align}
First let's look at \eqref{eq-y5.44}. On one hand, \eqref{eq-y5.1} implies $|\xi(t)-\xi(G^i_{\alpha})| \ll 2^i$.
Hence taking $\vec{v}=\{v_j\}_{j \in \mathbb{Z}}$, $\hat{v}_j$ supported on $\{\xi:2^{i-2}\leq |\xi - \xi(G_{\alpha}^{i})|\leq 2^{i+2} \}$ for every $j\in \mathbb{Z}$ and $\|\vec{v}\|_{V^2_{\triangle}(l^{2};J_l)}=1$, then by Cauchy-Schwarz and \eqref{eq-y2.6}, Proposition \ref{pr-y5.5}, $V^2_{\triangle}(l^{2})\subset U^4_{\triangle}(l^{2})\subset L^{4}_t L^{4}_x l^{2}$, Lemma \ref{le-y6.2} we have
\begin{equation}\label{eq-y5.48}
\begin{split}
&\left| \int_{J_l} \int_{\mathbb{R}^2} \sum_{j\in \mathbb{Z}} v_j \sum\limits_{\mathcal{R}(j)} P_{\xi(t), l_2} u_{j_1} P_{\xi(t), \leq l_2} \bar{u}_{j_2} P_{\xi( G_{\alpha}^{i}), i-5\leq\cdot\leq i+5} u_{j_3}dxdt\right|\\
\lesssim& \int_{J_l} \int_{\mathbb{R}^2} \|\vec{v}\|_{l^{2}} \|P_{\xi(t), \leq l_2} \vec{u}\|_{l^{2}} \|P_{\xi( G_{\alpha}^{i}), i-5\leq\cdot\leq i+5} \vec{u}\|_{l^{2}} \|P_{\xi(t), l_2} \vec{u}\|_{l^{2}}dxdt\\
\lesssim& \left\|\|P_{\xi( G_{\alpha}^{i}), i-5\leq\cdot\leq i+5} \vec{u}\|_{l^{2}} \|P_{\xi(t), l_2} \vec{u}\|_{l^{2}}\right\|_{L^{2}_t L^{2}_x (J_l\times\mathbb{R}^2)} \|\vec{v}\|_{L^{4}_t L^{4}_x l^{2} (J_l)} \|P_{\xi(t), \leq l_2} \vec{u}\|_{L^{4}_t L^{4}_x l^{2} (J_l)}\\
\lesssim& 2^{\frac{l_2-i}{2}} \|P_{\xi( G_{\alpha}^{i}), i-5\leq\cdot\leq i+5} \vec{u}\|_{U^2_{\triangle}(l^{2};G_{\alpha}^{i})}.
\end{split}
\end{equation}
On the other hand, by Proposition \ref{pr-y5.5}, Proposition \ref{pr-y5.3} (2)(6), Lemma \ref{le-y6.2} and $N(J_l)\geq \ep_3^{1/2} 2^{l_2-5}$,
\begin{equation}\label{eq-y5.49}
  \begin{split}
    &\left| \int_{J_l} \int_{\mathbb{R}^2} \sum_{j\in \mathbb{Z}} v_j \sum\limits_{\mathcal{R}(j)} P_{\xi(t), l_2} u_{j_1} P_{\xi(t), \leq l_2} \bar{u}_{j_2} P_{\xi( G_{\alpha}^{i}), i-5\leq\cdot\leq i+5} u_{j_3}dxdt\right|\\
     \lesssim& \left\|\|P_{\xi( G_{\alpha}^{i}), i-5\leq\cdot\leq i+5} \vec{u}\|_{l^{2}} \|P_{\xi(t), l_2} \vec{u}\|_{l^{2}}\right\|_{L^{2}_t L^{2}_x (J_l\times\mathbb{R}^2)} \left\|\|\vec{v}\|_{l^{2}} \|P_{\xi(t), \leq l_2} \vec{u}\|_{l^{2}}\right\|^{5/6}_{L^{20/9}_t L^{20/11}_x (J_l\times\mathbb{R}^2)}\\
     &\quad \times \|P_{\xi(t), \leq l_2} \vec{u}\|^{1/6}_{L^{\frac{8}{3}}_t L^{8}_x l^{2} (J_l)} \|\vec{v}\|^{1/6}_{L^{\frac{8}{3}}_t L^{8}_x l^{2} (J_l)}\\
     \lesssim& 2^{\frac{l_2-i}{2}} \|P_{\xi( G_{\alpha}^{i}), i-5\leq\cdot\leq i+5} \vec{u}\|_{U^2_{\triangle}(l^{2};G_{\alpha}^{i})}\times 2^{\frac{3(l_2-i)}{8}}\\
     \lesssim& \ep_3^{-7/16} 2^{-7i/8} N(J_l)^{7/8}\|P_{\xi( G_{\alpha}^{i}), i-5\leq\cdot\leq i+5} \vec{u}\|_{U^2_{\triangle}(l^{2};G_{\alpha}^{i})}.
  \end{split}
\end{equation}
Interpolating \eqref{eq-y5.48} and \eqref{eq-y5.49}, we obtain

\begin{equation*}
\begin{split}
  & \big\| P_{\xi( G_{\alpha}^{i}), i-2 \leq \cdot \leq i+2} \big( \sum^{\rightarrow}\limits_{\mathcal{R}(j)} P_{\xi(\tau), l_2} u_{j_1} P_{\xi(\tau), \leq l_2} \bar{u}_{j_2} P_{\xi( G_{\alpha}^{i}), i-5\leq\cdot\leq i+5} u_{j_3}\big) \big\|^2_{DU^{2}_{\triangle}(l^{2};J_l)} \\
  \lesssim& \ep_3^{-1/2} 2^{-i} N(J_l) 2^{\frac{3(l_2-i)}{7}}\|P_{\xi( G_{\alpha}^{i}), i-5\leq\cdot\leq i+5} \vec{u}\|^2_{U^2_{\triangle}(l^{2};G_{\alpha}^{i})}.
\end{split}
\end{equation*}
Finally, by \eqref{eq-y5.34},
\begin{equation}\label{eq-y5.50}
\begin{split}
  \eqref{eq-y5.44}\lesssim&\sum_{0\leq l_2\leq i-10}\big( \sum_{J_l\subset G^{i}_{\alpha};N(J_l)\geq \ep_3^{1/2}2^{l_{2}-5}} \ep_3^{-1/2} 2^{-i} N(J_l) 2^{\frac{3(l_2-i)}{7}}\|P_{\xi( G_{\alpha}^{i}), i-5\leq\cdot\leq i+5} \vec{u}\|^2_{U^2_{\triangle}(l^{2};G_{\alpha}^{i})} \big)^{1/2} \\
  \lesssim&\ep_3^{1/4}\|P_{\xi( G_{\alpha}^{i}), i-5\leq\cdot\leq i+5} \vec{u}\|_{U^2_{\triangle}(l^{2};G_{\alpha}^{i})}\\
  \lesssim&\ep_2^{2}\|P_{\xi( G_{\alpha}^{i}), i-5\leq\cdot\leq i+5} \vec{u}\|_{U^2_{\triangle}(l^{2};G_{\alpha}^{i})}.
\end{split}
\end{equation}
For \eqref{eq-y5.45} - \eqref{eq-y5.47} we will use three bilinear estimates below that rely on the interaction Morawetz estimate of \cite{FL}, whose proof will be postponed to the next section. Such estimates will give a logarithmic improvement over what would be obtained from \eqref{eq-y4.2'} directly. This improvement is quite helpful to the proof.
\begin{theorem}[First bilinear Strichartz estimate]\label{th-y5.16}
Suppose that $\vec{v}_0=\{v_{0,j}\}_{j\in \mathbb{Z}}\in L_x^2l^{2} (\mathbb{R}^2\times \mathbb{Z} )$ and $\hat{v}_{0,j}$ is supported on
$\{\xi:2^{i-5}\leq |\xi - \xi(G_{\alpha}^{i})|\leq 2^{i+5} \}$ for every $j\in \mathbb{Z}$.  Also suppose $J_l\subset G^i_{\alpha}$ is a small interval and $|\xi(t)-\xi(G^i_{\alpha})|\leq 2^{i-10}$ for all $t\in G^i_{\alpha}$. Then for any $0\leq l_2 \leq i-10$,
\begin{equation}\label{eq-y5.51}
\begin{split}
  &\left\| \|e^{it\triangle}\vec{v}_0\|_{l^{2}} \|P_{\xi(t),\leq l_2} \vec{u}\|_{l^{2}} \right\|^2_{L_t^2L_x^2(J_l\times\mathbb{R}^2)} \lesssim 2^{l_2-i}\|\vec{v}_0\|^2_{L_x^2l^{2}(\mathbb{R}^2\times \mathbb{Z})}  \\
  +& 2^{-i}\|\vec{v}_0\|^2_{L_x^2l^{2}(\mathbb{R}^2\times \mathbb{Z})} \big(\int_{J_l} |\xi'(t)|\sum_{l_1\leq l_2}2^{\frac{l_1-l_2}{2}} \|P_{\xi(t),l_2-3 \leq \cdot\leq l_2+3} \vec{u}\|_{L_x^2l^{2}(\mathbb{R}^2\times \mathbb{Z})} \|P_{\xi(t), l_1} \vec{u}\|_{L_x^2l^{2}(\mathbb{R}^2\times \mathbb{Z})} dt\big)
\end{split}
\end{equation}
The same estimate also holds when $P_{\xi(t),\leq l_2}$ is replaced by $P_{\xi(t), l_2}$.
\end{theorem}

Now let's use Theorem \ref{th-y5.16} to estimate \eqref{eq-y5.45}.
\begin{equation}\label{eq-y5.64}
\begin{split}
 &\big\|\int_{J_l} e^{-it\triangle} P_{\xi( G_{\alpha}^{i}), i-2 \leq \cdot \leq i+2} \big( \sum^{\rightarrow}\limits_{\mathcal{R}(j)} P_{\xi(t), l_2} u_{j_1} P_{\xi(t), \leq l_2} \bar{u}_{j_2} P_{\xi( G_{\alpha}^{i}), i-5\leq\cdot\leq i+5} u_{j_3}\big)dt \big\|_{L_x^2l^{2}}\\
 \lesssim&\sup_{\|\vec{v}_0\|_{L^2_xh^{0}}=1} \big\| \sum_{j}e^{it\triangle}v_{0,j} \sum\limits_{\mathcal{R}(j)} P_{\xi(t), l_2} u_{j_1} P_{\xi(t), \leq l_2} \bar{u}_{j_2} P_{\xi( G_{\alpha}^{i}), i-5\leq\cdot\leq i+5} u_{j_3}\big\|_{L^1_{t,x}(J_l\times\mathbb{R}^2)}\\
 \lesssim&\sup_{\|\vec{v}_0\|_{L^2_xh^{0}}=1} \big\| \|e^{it\triangle}\vec{v}_{0}\|_{l^{2}}  \|P_{\xi(t), l_2} \vec{u}\|_{l^{2}} \|P_{\xi(t), \leq l_2} \vec{u}\|_{l^{2}} \|P_{\xi( G_{\alpha}^{i}), i-5\leq\cdot\leq i+5} \vec{u}\|_{l^{2}}\big\|_{L^1_{t,x}(J_l\times\mathbb{R}^2)},
\end{split}
\end{equation}
for $\vec{v}_0=\{v_{0,j}\}_{j\in \mathbb{Z}}$ with $\hat{v}_{0,j}$ supported on
$\{\xi:2^{i-5}\leq |\xi - \xi(G_{\alpha}^{i})|\leq 2^{i+5} \}$ for every $j\in \mathbb{Z}$. On one hand, by Theorem \ref{th-y5.16}, we have
\begin{equation*}
\begin{split}
&\big\|\|e^{it\triangle}\vec{v}_{0}\|_{l^{2}}  \|P_{\xi(t), l_2} \vec{u}\|_{l^{2}}\big\|_{L^2_tL^2_x(J_l\times\mathbb{R}^2)}\\
\lesssim&2^{(l_2-i)/2}+2^{-i/2} \big(\int_{J_l} |\xi'(t)|\sum_{l_1\leq l_2}2^{\frac{l_1-l_2}{2}} \|P_{\xi(t),l_2-3 \leq \cdot\leq l_2+3} \vec{u}\|_{L_x^2l^{2}(\mathbb{R}^2\times \mathbb{Z})} \|P_{\xi(t), l_1} \vec{u}\|_{L_x^2l^{2}(\mathbb{R}^2\times \mathbb{Z})} dt\big)^{1/2}.
\end{split}
\end{equation*}
On the other hand, applying Theorem \ref{th-y5.16} to atoms of $P_{\xi( G_{\alpha}^{i}), i-5\leq\cdot\leq i+5} \vec{u}$, we have
\begin{equation*}
\begin{split}
  &\big\|\|P_{\xi(t), \leq l_2} \vec{u}\|_{l^{2}} \|P_{\xi( G_{\alpha}^{i}), i-5\leq\cdot\leq i+5} \vec{u}\|_{l^{2}}\big\|_{L^2_tL^2_x(J_l\times\mathbb{R}^2)}\\
  \lesssim&2^{(l_2-i)/2}\|P_{\xi( G_{\alpha}^{i}), i-5\leq\cdot\leq i+5} \vec{u}\|_{U^2_{\triangle}(l^{2};J_l\times\mathbb{R}^2)}+ 2^{-i/2}\|P_{\xi( G_{\alpha}^{i}), i-5\leq\cdot\leq i+5} \vec{u}\|_{U^2_{\triangle}(l^{2};J_l\times\mathbb{R}^2)}\\
  &  \times \big(\int_{J_l} |\xi'(t)|\sum_{l_1\leq l_2}2^{\frac{l_1-l_2}{2}} \|P_{\xi(t),l_2-3 \leq \cdot\leq l_2+3}  \vec{u}\|_{L_x^2l^{2}(\mathbb{R}^2\times \mathbb{Z})} \|P_{\xi(t), l_1} \vec{u}\|_{L_x^2l^{2}(\mathbb{R}^2\times \mathbb{Z})} dt\big)^{1/2}.
\end{split}
\end{equation*}
So
\begin{equation}\label{eq-y5.65}
\begin{split}
  \eqref{eq-y5.64} \lesssim& 2^{l_2-i}\|P_{\xi( G_{\alpha}^{i}), i-5\leq\cdot\leq i+5} \vec{u}\|_{U^2_{\triangle}(l^{2};J_l\times\mathbb{R}^2)}+ 2^{-i}\|P_{\xi( G_{\alpha}^{i}), i-5\leq\cdot\leq i+5} \vec{u}\|_{U^2_{\triangle}(l^{2};J_l\times\mathbb{R}^2)}\\
  &\times\big(\int_{J_l} |\xi'(t)|\sum_{l_1\leq l_2}2^{\frac{l_1-l_2}{2}} \|P_{\xi(t),l_2-3 \leq \cdot\leq l_2+3}  \vec{u}\|_{L_x^2l^{2}(\mathbb{R}^2\times \mathbb{Z})} \|P_{\xi(t), l_1} \vec{u}\|_{L_x^2l^{2}(\mathbb{R}^2\times \mathbb{Z})} dt\big).
\end{split}
\end{equation}
Rearranging the order of summation and by \eqref{eq-y5.34}, we have
\begin{equation}\label{eq-y5.65''}
\begin{split}
&\sum_{0\leq l_2\leq i-10}  \sum_{J_l\subset G^{i}_{\alpha};N(J_l)\geq \ep_3^{1/2}2^{l_{2}-5}}2^{l_2-i}\|P_{\xi( G_{\alpha}^{i}), i-5\leq\cdot\leq i+5} \vec{u}\|_{U^2_{\triangle}(l^{2};J_l\times\mathbb{R}^2)}\\
\leq & 2^{-i} \ep_3^{-1/2}\|P_{\xi(G_{\alpha}^{i}),i-5\leq\cdot\leq i+5} \vec{u}\|_{U^2_{\triangle}(l^{2};G_{\alpha}^{i}\times\mathbb{R}^2)} \sum_{J_l\subset G^i_{\alpha}}N(J_l)\\
\lesssim &  \ep_3^{1/2} \|P_{\xi(G_{\alpha}^{i}),i-5\leq\cdot\leq i+5} \vec{u}\|_{U^2_{\triangle}(l^{2};G_{\alpha}^{i}\times\mathbb{R}^2)}.
\end{split}
\end{equation}
H\"{o}lder inequality and Young's inequality imply that
\begin{equation}\label{eq-y5.67'}
\sum_{0\leq l_2\leq i-10}\|P_{\xi(t),l_2-3 \leq \cdot\leq l_2+3}  \vec{u}\|_{L_x^2l^{2}(\mathbb{R}^2\times \mathbb{Z})}\sum_{l_1\leq l_2}2^{\frac{l_1-l_2}{2}} \|P_{\xi(t), l_1} \vec{u}\|_{L_x^2l^{2}(\mathbb{R}^2\times \mathbb{Z})} \lesssim \|\vec{u}\|^2_{L_x^2l^{2}(\mathbb{R}^2\times \mathbb{Z})}.
\end{equation}
So by \eqref{eq-y5.1}, Theorem \ref{th-y4.5} (3), \eqref{eq-y5.34} and conservation of mass,
\begin{equation}\label{eq-y5.66}
\begin{split}
&\sum_{0\leq l_2\leq i-10}  \sum_{J_l\subset G^{i}_{\alpha}} 2^{-i}\|P_{\xi( G_{\alpha}^{i}), i-5\leq\cdot\leq i+5} \vec{u}\|_{U^2_{\triangle}(l^{2};J_l\times\mathbb{R}^2)}\\
&\quad\times\big(\int_{J_l} |\xi'(t)|\sum_{l_1\leq l_2}2^{\frac{l_1-l_2}{2}} \|P_{\xi(t),l_2-3 \leq \cdot\leq l_2+3}  \vec{u}\|_{L_x^2l^{2}(\mathbb{R}^2\times \mathbb{Z})} \|P_{\xi(t), l_1} \vec{u}\|_{L_x^2l^{2}(\mathbb{R}^2\times \mathbb{Z})} dt\big)\\
\lesssim& \sum_{J_l\subset G^{i}_{\alpha}} 2^{-i}\|P_{\xi( G_{\alpha}^{i}), i-5\leq\cdot\leq i+5} \vec{u}\|_{U^2_{\triangle}(l^{2};G_{\alpha}^{i}\times\mathbb{R}^2)}(N(J_l)\ep_1^{-1/2} \|\vec{u}\|^2_{L^{\infty}_tL_x^2l^{2}(J_l\times\mathbb{R}^2\times \mathbb{Z})})\\
\lesssim& \ep_3^{1/2} \|P_{\xi(G_{\alpha}^{i}),i-5\leq\cdot\leq i+5} \vec{u}\|_{U^2_{\triangle}(l^{2};G_{\alpha}^{i}\times\mathbb{R}^2)}.
\end{split}
\end{equation}
Combining \eqref{eq-y5.65} and \eqref{eq-y5.66},
$$\eqref{eq-y5.45}\lesssim \ep_3^{1/2} \|P_{\xi(G_{\alpha}^{i}),i-5\leq\cdot\leq i+5} \vec{u}\|_{U^2_{\triangle}(l^{2};G_{\alpha}^{i}\times\mathbb{R}^2)}.$$
Next we deal with \eqref{eq-y5.46}. In order to control \eqref{eq-y5.46}, we need the following
\begin{theorem}[Second bilinear Strichartz estimate]\label{th-y5.17}
Suppose that $\vec{v}_0=\{v_{0,j}\}_{j\in \mathbb{Z}}\in L_x^2l^{2} (\mathbb{R}^2\times \mathbb{Z} )$, $\hat{v}_{0,j}$ supported on
$\{\xi:2^{i-5}\leq |\xi - \xi(G_{\alpha}^{i})|\leq 2^{i+5} \}$ for every $j\in \mathbb{Z}$.  Then for any $0\leq l_2 \leq i-10$, $G_{\beta}^{l_2}\subset G^i_{\alpha}$,
\begin{equation}\label{eq-y5.67}
\left\| \|e^{it\triangle}\vec{v}_0\|_{l^{2}} \cdot\|P_{\xi(t),\leq l_2} \vec{u}\|_{l^{2}} \right\|^2_{L_t^2L_x^2(G_{\beta}^{l_2} \times\mathbb{R}^2)} \lesssim \|\vec{v}_0\|^2_{L_x^2l^{2}(\mathbb{R}^2\times \mathbb{Z})}(1+\|\vec{u}\|^4_{\tilde{X}_{i}(G_{\alpha}^{i}\times\mathbb{R}^2)}).
\end{equation}
\end{theorem}

Now let's use Theorem \ref{th-y5.17} to estimate \eqref{eq-y5.46}. For $G^{i-10}_{\beta}\subset G^i_{\alpha}$, $0\leq l_2\leq i-10$, Remark \ref{convolution} and Theorem \ref{th-y5.17} ensure that
\begin{equation}\label{eq-y5.93}
\begin{split}
&\big\| \|P_{\xi(G^i_{\alpha}),i-5\leq \cdot\leq i+5}\vec{u}\|_{l^{2}}\cdot \|P_{\xi(t),\leq l_2}\vec{u}\|_{l^{2}} \big\|_{L^2_{t,x}(G^{i-10}_{\beta}\times \mathbb{R}^2 )}\\
\lesssim & \big\| \|P_{\xi(G^i_{\alpha}),i-5\leq \cdot\leq i+5}\vec{u}\|_{l^{2}}\cdot \|P_{\xi(t),\leq i-10}\vec{u}\|_{l^{2}} \big\|_{L^2_{t,x}(G^{i-10}_{\beta}\times \mathbb{R}^2 )} \\
\lesssim & \|P_{\xi(G^i_{\alpha}),i-5\leq \cdot\leq i+5} \vec{u}\|_{U^2_{\triangle}(l^{2};G^i_{\alpha})} (1+\|\vec{u}\|^2_{\tilde{X}_{i}(G_{\alpha}^{i}\times\mathbb{R}^2)}).
\end{split}
\end{equation}
Since $G_{\alpha}^{i}$ consists of $2^{10}$ subintervals $G^{i-10}_{\beta}$, we have
\begin{equation}\label{eq-y5.94}
\big\| \|P_{\xi(G^i_{\alpha}),i-5\leq \cdot\leq i+5}\vec{u}\|_{l^{2}}\cdot \|P_{\xi(t),\leq l_2}\|_{l^{2}} \big\|_{L^2_{t,x}(G^{i}_{\alpha}\times \mathbb{R}^2 )}\lesssim \|P_{\xi(G^i_{\alpha}),i-5\leq \cdot\leq i+5}\vec{u}\|_{U^2_{\triangle}(l^{2};G^i_{\alpha})}(1+\|\vec{u}\|^2_{\tilde{X}_{i}(G_{\alpha}^{i}\times\mathbb{R}^2)}).
\end{equation}
Now suppose that, for each $G^{l_2}_{\beta} \subset G^i_{\alpha}$, take $\vec{v}^{l_2}_{\beta}=\{v^{l_2}_{\beta,j}\}_{j \in \mathbb{Z}}$ with $\hat{v}^{l_2}_{\beta,j}$ supported on $\{\xi:2^{i-2}\leq |\xi - \xi(G_{\alpha}^{i})|\leq 2^{i+2} \}$ for every $j\in \mathbb{Z}$ and $\|\vec{v}^{l_2}_{\beta}\|_{V^2_{\triangle}(l^{2};G^{l_2}_{\beta})}=1$, then by H\"{o}lder inequality, \eqref{eq-y2.6} and \eqref{eq-y5.94}, we have
\begin{equation*}
\begin{split}
&\big( \sum_{G^{l_2}_{\beta}\subset G^{i}_{\alpha};N(G^{l_2}_{\beta})\leq \ep_3^{1/2}2^{l_{2}-5}} \big\|\sum_{j} v^{l_2}_{\beta,j} \sum\limits_{\mathcal{R}(j)} (P_{\xi(t), l_2} u_{j_1}) (P_{\xi(t), \leq l_2} \bar{u}_{j_2}) (P_{\xi( G_{\alpha}^{i}), i-5\leq\cdot\leq i+5} u_{j_3}) \big\|^2_{L^1_{t,x}(G^{l_2}_{\beta}\times\mathbb{R}^2)} \big)^{1/2}\\
\lesssim & \big(\sup_{G^{l_2}_{\beta}\subset G^{i}_{\alpha};N(G^{l_2}_{\beta})\leq \ep_3^{1/2}2^{l_{2}-5}} \big\|\|\vec{v}^{l_2}_{\beta}\|_{l^{2}}\cdot \|P_{\xi(t),l_2}\vec{u}\|_{l^{2}} \big\|_{L^2_{t,x}(G^{l_2}_{\beta}\times\mathbb{R}^2)}\big) \big(\big\| \|P_{\xi(G^i_{\alpha}),i-5\leq \cdot\leq i+5}\vec{u}\|_{l^{2}}\cdot \|P_{\xi(t),\leq l_2}\|_{l^{2}} \big\|_{L^2_{t,x}(G^{i}_{\alpha}\times \mathbb{R}^2 )}\big)\\
\lesssim & \big(\sup_{G^{l_2}_{\beta}\subset G^{i}_{\alpha};N(G^{l_2}_{\beta})\leq \ep_3^{1/2}2^{l_{2}-5}} \big\|\|\vec{v}^{l_2}_{\beta}\|_{l^{2}}\cdot \|P_{\xi(t),l_2}\vec{u}\|_{l^{2}} \big\|_{L^2_{t,x}(G^{l_2}_{\beta}\times\mathbb{R}^2)}\big) \|P_{\xi(G^i_{\alpha}),i-5\leq \cdot\leq i+5}\vec{u}\|_{U^2_{\triangle}(l^{2};G^i_{\alpha})} (1+\|\vec{u}\|^2_{\tilde{X}_{i}(G_{\alpha}^{i}\times\mathbb{R}^2)}).
\end{split}
\end{equation*}
By Proposition \ref{pr-y5.5}, $U^2_{\triangle}(l^{2})\subset V^2_{\triangle}(l^{2})\subset U^3_{\triangle}(l^{2})\subset L^{3}_t L^{6}_x l^{2}$ and $N(G^{l_2}_{\beta})\leq \ep_3^{1/2}2^{l_{2}-5}$,
\begin{equation*}
\begin{split}
&\big\|\|\vec{v}^{l_2}_{\beta}\|_{l^{2}}\cdot \|P_{\xi(t),l_2}\vec{u}\|_{l^{2}} \big\|_{L^2_{t,x}(G^{l_2}_{\beta}\times\mathbb{R}^2)}\\
\lesssim& \|\|\vec{v}^{l_2}_{\beta}\|_{l^{2}}\cdot \|P_{\xi(t),l_2}\vec{u}\|_{l^{2}} \big\|^{1/2}_{L^3_{t}L^{3/2}_{x}(G^{l_2}_{\beta}\times\mathbb{R}^2)} \|\vec{v}^{l_2}_{\beta}\|^{1/2}_{L^3_{t}L^{6}_{x}l^{2}(G^{l_2}_{\beta}\times\mathbb{R}^2)} \|P_{\xi(t),l_2}\vec{u}\|^{1/2}_{L^3_{t}L^{6}_{x}l^{2}(G^{l_2}_{\beta}\times\mathbb{R}^2)}\\
\lesssim& 2^{(l_2-i)/6}\|\vec{u}\|_{\tilde{Y}_i(G^{i}_{\alpha})}.
\end{split}
\end{equation*}
Therefore,
\begin{equation*}
  \eqref{eq-y5.46}\lesssim \|P_{\xi(G^i_{\alpha}),i-5\leq \cdot\leq i+5}\vec{u}\|_{U^2_{\triangle}(l^{2};G^i_{\alpha})} \|\vec{u}\|_{\tilde{Y}_i(G^{i}_{\alpha})} (1+\|\vec{u}\|^2_{\tilde{X}_{i}(G_{\alpha}^{i}\times\mathbb{R}^2)}).
\end{equation*}

Finally we estimate \eqref{eq-y5.47}. Take $\vec{v}_0=\{v_{0,j}\}_{j\in \mathbb{Z}}$ with $\hat{v}_{0,j}$ supported on
$\{\xi:2^{i-2}\leq |\xi - \xi(G_{\alpha}^{i})|\leq 2^{i+2} \}$ for every $j\in \mathbb{Z}$ and $\|\vec{v}_0\|_{L^2_xh^{0}}=1$, by H\"{o}lder and Proposition \ref{pr-y5.4},
\begin{equation}\label{eq-y5.95}
\begin{split}
 &\big\|\int_{G^{l_2}_{\beta}} e^{-it\triangle} P_{\xi( G_{\alpha}^{i}), i-2 \leq \cdot \leq i+2} \big( \sum^{\rightarrow}\limits_{\mathcal{R}(j)} P_{\xi(t), l_2} u_{j_1} P_{\xi(t), \leq l_2} \bar{u}_{j_2} P_{\xi( G_{\alpha}^{i}), i-5\leq\cdot\leq i+5} u_{j_3}\big)dt \big\|_{L_x^2l^{2}}\\
 \lesssim&\sup_{\|\vec{v}_0\|_{L^2_xh^{0}}=1} \big\| \sum_{j}e^{it\triangle}v_{0,j} \sum\limits_{\mathcal{R}(j)} P_{\xi(t), l_2} u_{j_1} P_{\xi(t), \leq l_2} \bar{u}_{j_2} P_{\xi( G_{\alpha}^{i}), i-5\leq\cdot\leq i+5} u_{j_3}\big\|_{L^1_{t,x}(G^{l_2}_{\beta}\times\mathbb{R}^2)}\\
 \lesssim&\sup_{\|\vec{v}_0\|_{L^2_xh^{0}}=1} \big\| \|e^{it\triangle}\vec{v}_{0}\|_{l^{2}}  \|P_{\xi(t), l_2} \vec{u}\|_{l^{2}} \|P_{\xi(t), \leq l_2} \vec{u}\|_{l^{2}} \|P_{\xi( G_{\alpha}^{i}), i-5\leq\cdot\leq i+5} \vec{u}\|_{l^{2}}\big\|_{L^1_{t,x}(G^{l_2}_{\beta}\times\mathbb{R}^2)}\\
 \lesssim&2^{(l_2-i)/2}\| P_{\xi(G^{l_2}_{\beta}),l_2-2\leq\cdot\leq l_2+2} \vec{u}\|_{U^2_{\triangle}(l^{2};G^{l_2}_{\beta})} \big\|\|P_{\xi(t), \leq l_2} \vec{u}\|_{l^{2}} \cdot \|P_{\xi( G_{\alpha}^{i}), i-5\leq\cdot\leq i+5} \vec{u}\|_{l^{2}}\big\|_{L^2_{t,x}(G^{l_2}_{\beta}\times\mathbb{R}^2)}.
\end{split}
\end{equation}
Therefore, by Cauchy-Schwartz inequality,
\begin{equation}\label{eq-y5.96}
\eqref{eq-y5.47}\lesssim \|\vec{u}\|_{\tilde{Y}_i(G^{i}_{\alpha})} \big( \sum_{0\leq l_2\leq i-10}\big\|\|P_{\xi(t), \leq l_2} \vec{u}\|_{l^{2}} \cdot \|P_{\xi( G_{\alpha}^{i}), i-5\leq\cdot\leq i+5} \vec{u}\|_{l^{2}}\big\|^2_{L^2_{t,x}(G^{i}_{\alpha}\times\mathbb{R}^2)}\big)^{1/2}.
\end{equation}
Theorem \ref{th-y5.14} will follow from the final bilinear estimate.
\begin{theorem}[Third bilinear Strichartz estimate]\label{th-y5.19}
\begin{equation}\label{eq-y5.97}
\sum_{0\leq l_2\leq i-10}\big\|\|P_{\xi(t), \leq l_2} \vec{u}\|_{l^{2}} \cdot \|e^{it\triangle}\vec{v}_0\|_{l^{2}}\big\|^2_{L^2_{t,x}(G^{i}_{\alpha}\times\mathbb{R}^2)}\lesssim \|\vec{v}_0\|^2_{L^2_xl^2}(1+\|\vec{u}\|^6_{\tilde{X}_i(G^{i}_{\alpha})}).
\end{equation}
\end{theorem}
Indeed, if \eqref{eq-y5.97} holds, Applying \eqref{eq-y5.97} to atoms of $P_{\xi( G_{\alpha}^{i}), i-5\leq\cdot\leq i+5} \vec{u}$, we have
\begin{equation*}
\eqref{eq-y5.47}\lesssim \|\vec{u}\|_{\tilde{Y}_i(G^{i}_{\alpha})} \|P_{\xi(G^i_{\alpha}),i-5\leq\cdot\leq i+5}\vec{u}\|_{U^2_{\triangle}(l^{2};G^i_{\alpha})}(1+\|\vec{u}\|^3_{\tilde{X}_i(G^{i}_{\alpha})}).
\end{equation*}

Theorem \ref{th-y5.19} then implies Theorem \ref{th-y5.14}, which together with Theorem \ref{th-y5.13} implies Theorem \ref{th-y5.10}, which  in turn implies Theorem \ref{th-y6.2}.

\end{proof}

\section{Proof of three bilinear Strichartz estimates}
In this section, we pay our debt from last section and complete the proofs of three bilinear Strichartz estimates. Actually, the proofs of three bilinear Strichartz estimates are all dependent on the interaction Morawetz estimates, thus are similar. The big differences among their proofs are caused by the time interval. Compared with Theorem \ref{th-y5.17} and Theorem \ref{th-y5.19},  Theorem \ref{th-y5.16}'s time interval is small, this makes the proof much easier. While the proofs of Theorem \ref{th-y5.17} and Theorem \ref{th-y5.19} are similar, so we just give the proof of Theorem \ref{th-y5.17} below.

\begin{proof}[Proof of Theorem \ref{th-y5.17}]Let $\vec{v}=\{v_j\}_{j\in \mathbb{Z}}$, $v_j=e^{it\triangle}v_{0,j}$ and $\vec{w}=P_{\xi(t),\leq l_2} \vec{u}$, $w_{j'}=P_{\xi(t),\leq l_2} u_{j'}$. Then  $\vec{v}$ and  $\vec{w}$ solve equations
\begin{equation*}
  i\partial_t v_j+\triangle v_j=0,
\end{equation*}
and
\begin{equation*}
  i\partial_t w_{j'}+\triangle w_{j'}=\sum_{\mathcal{R}(j')}w_{j'_1}\bar{w}_{j'_2}w_{j'_3} +N_{1,j'}+N_{2,j'}=\sum_{\mathcal{R}(j')}w_{j'_1}\bar{w}_{j'_2}w_{j'_3} +N_{j'},
\end{equation*}
where $N_{1,j'}=P_{\xi(t),\leq l_2}\big(\sum_{\mathcal{R}(j')}u_{j'_1}\bar{u}_{j'_2}u_{j'_3}\big)- \sum_{\mathcal{R}(j')}w_{j'_1}\bar{w}_{j'_2}w_{j'_3}$ and $N_{2,j'}=(\frac{d}{dt}P_{\xi(t),\leq l_2})u_{j'}$ with $\frac{d}{dt}P_{\xi(t),\leq l_2}$ given by Fourier multiplier
\begin{equation}\label{eq-yfm1}
-\nabla \phi(\frac{\xi-\xi(t)}{2^{l_2}})\cdot \frac{\xi'(t)}{2^{l_2}}.
\end{equation}

Let
\begin{equation}\label{eq-y5.52}
\begin{split}
M(t)=\sum_{j,j'\in \mathbb{Z}} \bigg(&\int_{\mathbb{R}^2} \int_{\mathbb{R}^2} |w_{j'}(t,y)|^2 \frac{(x-y)}{|x-y|}\cdot Im[\bar{v}_j\nabla v_j](t,x)dxdy\\
 +& \int_{\mathbb{R}^2} \int_{\mathbb{R}^2} |v_{j}(t,y)|^2 \frac{(x-y)}{|x-y|}\cdot Im[\bar{w}_{j'}\nabla w_{j'}](t,x)dxdy \bigg).
\end{split}
\end{equation}
Following \cite{FL} and \cite{D}, we can calculate
\begin{align}\nonumber
      &\sum_{j,j'\in \mathbb{Z}}   \int_{G_{\beta}^{l_2}} \int_{\mathbb{R}^2}|\bar{w}_{j'}v_j|^2(t,x)dxdt \lesssim 2^{l_2-2i}\sup_{t\in G_{\beta}^{l_2}}|M(t)|   \\ \label{eq-y5.68}
    + &2^{l_2-2i} \big| \sum_{j,j'\in \mathbb{Z}}   \int_{G_{\beta}^{l_2}} \int_{\mathbb{R}^2} \int_{\mathbb{R}^2} |v_j(t,y)|^2 \frac{(x-y)}{|x-y|}\cdot Im[\bar{N}_{j'} (\nabla- i\xi(t))w_{j'}](t,x)dxdydt \big|\\ \label{eq-y5.69}
    + & 2^{l_2-2i}\big| \sum_{j,j'\in \mathbb{Z}}   \int_{G_{\beta}^{l_2}} \int_{\mathbb{R}^2} \int_{\mathbb{R}^2} |v_j(t,y)|^2 \frac{(x-y)}{|x-y|}\cdot Im[\bar{w}_{j'} (\nabla- i\xi(t))N_{j'}](t,x)dxdydt \big|\\ \label{eq-y5.70}
    + & 2^{l_2-2i}\big| \sum_{j,j'\in \mathbb{Z}}   \int_{G_{\beta}^{l_2}} \int_{\mathbb{R}^2} \int_{\mathbb{R}^2} Im[\bar{v}_{j} (\nabla- i\xi(t))v_{j}](t,x) \frac{(x-y)}{|x-y|}\cdot Im[\bar{w}_{j'}N_{j'}](t,y)dxdydt \big|.
\end{align}
First,
\begin{equation}\label{eq-y5.71}
\begin{split}
|M(t)|=&\big|\sum_{j,j'\in \mathbb{Z}}   \big(\int_{\mathbb{R}^2} \int_{\mathbb{R}^2} |w_{j'}(t,y)|^2 \frac{(x-y)}{|x-y|}\cdot Im[\bar{v}_j(\nabla-i\xi(t)) v_j](t,x)dxdy\\
 & \quad \quad \quad \quad+\int_{\mathbb{R}^2} \int_{\mathbb{R}^2} |v_{j}(t,y)|^2 \frac{(x-y)}{|x-y|}\cdot Im[\bar{w}_{j'}(\nabla-i\xi(t)) w_{j'}](t,x)dxdy \big)\big|\\
 \lesssim & 2^i \|\vec{w}\|^2_{L_t^{\infty}L^2_xl^2} \|\vec{v}_0\|^2_{L^2_xl^2}\lesssim 2^i \|\vec{v}_0\|^2_{L^2_xl^2}.
\end{split}
\end{equation}
So $2^{l_2-2i}\sup_{t\in G_{\beta}^{l_2}}|M(t)|$ is bounded by the right-hand side of \eqref{eq-y5.67}.\\
Next by \eqref{eq-y5.1}, \eqref{eq-y5.33}, Bernstein inequality, conservation of mass and Strichartz estimate,
\begin{equation*}
\begin{split}
|\eqref{eq-y5.68}|\, \lesssim &2^{l_2-2i} \|\vec{N}_1\|_{L^{4/3}_{t,x}l^{2}(G_{\beta}^{l_2}\times \mathbb{R}^2)} \|(\nabla-i\xi(t)) \vec{w}\|_{L^4_{t,x}l^{2}(G_{\beta}^{l_2}\times \mathbb{R}^2)} \|\vec{v}\|^2_{L^{\infty}_t L^2_x l^{2}(G_{\beta}^{l_2}\times \mathbb{R}^2)}+ 2^{-2i}\|\vec{v}\|^2_{L^{\infty}_t L_x^2 l^{2}(G_{\beta}^{l_2}\times\mathbb{R}^2)}\\
&\quad\quad\times \big(\sum_{j'}\int_{G_{\beta}^{l_2}} |\xi'(t)| \|P_{\xi(t),l_2-3 \leq \cdot\leq l_2+3} u_{j'}\|_{L_x^2(\mathbb{R}^2)} \|(\nabla-i\xi(t))P_{\xi(t),\leq l_2} u_{j'}\|_{L_x^2(\mathbb{R}^2)} dt\big)\\
\lesssim& 2^{l_2-2i} \|\vec{N}_1\|_{L^{4/3}_{t,x}l^{2}(G_{\beta}^{l_2}\times \mathbb{R}^2)} \|(\nabla-i\xi(t)) \vec{w}\|_{L^4_{t,x}l^{2}(G_{\beta}^{l_2}\times \mathbb{R}^2)} \|\vec{v}_0\|^2_{L^2_x l^{2}(\mathbb{R}^2\times \mathbb{Z})} + 2^{2l_2-2i}\|\vec{v}_0\|^2_{L_x^2 l^{2}(\mathbb{R}^2\times \mathbb{Z})}.
\end{split}
\end{equation*}
Let $m(t;\xi)=\frac{\xi-\xi(t)}{2^{l_1}}\phi(\frac{\xi-\xi(t)}{2^{l_1}})$, $m^{\vee}(t;x)=\int m(t;\xi) e^{ix\xi}d\xi$. Then by Minkowski inequality, Young's inequality, $\sup_{t}\int |m^{\vee}(t;x)|dx\lesssim1$ and \eqref{eq-y5.9}, we have
\begin{equation}\label{eq-y5.72'}
\begin{split}
&\|(\nabla-i\xi(t)) \vec{w}\|_{L^4_{t,x}l^{2}(G_{\beta}^{l_2}\times \mathbb{R}^2)}\\
\lesssim&\sum_{0\leq l_1\leq l_2}\|(\nabla-i\xi(t))P_{\xi(t),l_1} \vec{u}\|_{L^4_{t,x}l^{2}(G_{\beta}^{l_2}\times \mathbb{R}^2)}\\
=&\sum_{0\leq l_1\leq l_2} 2^{l_1} \|(\sum_{j'}|\int m^{\vee}(x-y) P_{\xi(t),l_1}u_{j'}(y)dy|^2 )^{1/2}\|_{L^4_{t,x}(G_{\beta}^{l_2}\times \mathbb{R}^2)}\\
\lesssim & \sum_{0\leq l_1\leq l_2} 2^{l_1} \|\int |m^{\vee}(x-y)| \|P_{\xi(t),l_1}\vec{u}\|_{l^{2}}(y)dy\|_{L^4_{t,x}(G_{\beta}^{l_2}\times \mathbb{R}^2)}\\
\lesssim & \sum_{0\leq l_1\leq l_2} 2^{l_1} \|P_{\xi(t),l_1}\vec{u}\|_{L^4_{t,x}l^{2}(G_{\beta}^{l_2}\times \mathbb{R}^2)}\\
\lesssim & \sum_{0\leq l_1\leq l_2} 2^{l_1} 2^{(l_2-l_1)/4}\|\vec{u}\|_{\tilde{X}_{l_2}(G_{\beta}^{l_2})}\lesssim 2^{l_2}\|\vec{u}\|_{\tilde{X}_{i}(G_{\alpha}^{i})}.
\end{split}
\end{equation}
Therefore,
\begin{equation}\label{eq-y5.72}
\begin{split}
|\eqref{eq-y5.68}| \lesssim 2^{2l_2-2i}\|\vec{v}_0\|^2_{L^2_x l^{2}(\mathbb{R}^2\times \mathbb{Z})} \|\vec{u}\|_{\tilde{X}_{i}(G_{\alpha}^{i})} \|\vec{N}_1\|_{L^{4/3}_{t,x}l^{2}(G_{\beta}^{l_2}\times \mathbb{R}^2)}  + 2^{2l_2-2i}\|\vec{v}_0\|^2_{L_x^2 l^{2}(\mathbb{R}^2\times \mathbb{Z})}.
\end{split}
\end{equation}
Split $u_{j'} =u^{h}_{j'} + u^l_{j'}$, where $u^l_{j'}=P_{\xi(t),\leq l_2-5}u_{j'}$. Denote $\vec{u}^l=\{u^l_{j'}\}_{j'\in \mathbb{Z}}$ and $\vec{u}^h=\{u^h_{j'}\}_{j'\in \mathbb{Z}}$. By this decomposition we rewrite $N_{1,j'}$ as follows
\begin{align}\nonumber
N_{1,j'}=&P_{\xi(t),\leq l_2}\big(\sum_{\mathcal{R}(j')}u_{j'_1}\bar{u}_{j'_2}u_{j'_3}\big)- \sum_{\mathcal{R}(j')}P_{\xi(t),\leq l_2} u_{j'_1}P_{\xi(t),\leq l_2} \bar{u}_{j'_2} P_{\xi(t),\leq l_2} u_{j'_3}\\ \label{eq-y5.73}
=& P_{\xi(t),\leq l_2}\big(\sum_{\mathcal{R}(j')}u^l_{j'_1}\bar{u}^l_{j'_2}u^l_{j'_3}\big)- \sum_{\mathcal{R}(j')}P_{\xi(t),\leq l_2} u^l_{j'_1}P_{\xi(t),\leq l_2} \bar{u}^l_{j'_2} P_{\xi(t),\leq l_2} u^l_{j'_3}\\ \label{eq-y5.74}
&+2\left(P_{\xi(t),\leq l_2}\big(\sum_{\mathcal{R}(j')}u^l_{j'_1}\bar{u}^l_{j'_2}u^h_{j'_3}\big)- \sum_{\mathcal{R}(j')}P_{\xi(t),\leq l_2} u^l_{j'_1}P_{\xi(t),\leq l_2} \bar{u}^l_{j'_2} P_{\xi(t),\leq l_2} u^h_{j'_3}\right)\\ \label{eq-y5.75}
&+P_{\xi(t),\leq l_2}\big(\sum_{\mathcal{R}(j')}u^l_{j'_1}\bar{u}^h_{j'_2}u^l_{j'_3}\big)- \sum_{\mathcal{R}(j')}P_{\xi(t),\leq l_2} u^l_{j'_1}P_{\xi(t),\leq l_2} \bar{u}^h_{j'_2} P_{\xi(t),\leq l_2} u^l_{j'_3}\\ \label{eq-y5.76}
&+O \left(P_{\xi(t),\leq l_2}\big(\sum_{\mathcal{R}(j')}u^h_{j'_1}\bar{u}^h_{j'_2}u_{j'_3}\big)- \sum_{\mathcal{R}(j')}P_{\xi(t),\leq l_2} u^h_{j'_1}P_{\xi(t),\leq l_2} \bar{u}^h_{j'_2} P_{\xi(t),\leq l_2} u_{j'_3}\right).
\end{align}where the ``$O$" in \eqref{eq-y5.76} represents there are two high frequency factors in this term,  may be  $u^h_{j'_1}$ and $\bar{u}^h_{j'_2}$, or $\bar{u}^h_{j'_2}$ and $u^h_{j'_3}$, or $u^h_{j'_1}$ and $u^h_{j'_3}$. However, their estimates are the same.

First we notice that \eqref{eq-y5.73}=0. For \eqref{eq-y5.74}, we can easily see $\hat{u}^h_{j'_3}$ is supported on $\{\sigma: |\sigma-\xi(t)|\leq 2^{l_2+10}\}$. Thus
\begin{equation}\label{eq-y5.77}
\begin{split}
&P_{\xi(t),\leq l_2}\big(\sum_{\mathcal{R}(j')}u^l_{j'_1}\bar{u}^l_{j'_2}u^h_{j'_3}\big)- \sum_{\mathcal{R}(j')}P_{\xi(t),\leq l_2} u^l_{j'_1}P_{\xi(t),\leq l_2} \bar{u}^l_{j'_2} P_{\xi(t),\leq l_2} u^h_{j'_3}\\
=&\sum_{\mathcal{R}(j')}\big( P_{\xi(t),\leq l_2}(e^{-ix\xi(t)}u^l_{j'_1}e^{ix\xi(t)}\bar{u}^l_{j'_2}u^h_{j'_3})- e^{-ix\xi(t)}P_{\xi(t),\leq l_2}u^l_{j'_1}e^{ix\xi(t)}P_{\xi(t),\leq l_2} \bar{u}^l_{j'_2} P_{\xi(t),\leq l_2} u^h_{j'_3}\big)\\
=&\sum_{\mathcal{R}(j')}\iiint e^{i\xi x} [(\phi(\frac{\xi-\xi(t)}{2^{l_2}})-\phi(\frac{\sigma-\xi(t)}{2^{l_2}})) \phi(\frac{\xi-\eta}{2^{l_2}}) \phi(\frac{\eta-\sigma}{2^{l_2}}) \phi(\frac{\sigma-\xi(t)}{2^{l_2+10}})] \\
&\qquad\times \hat{u}^l_{j'_1}(\xi-\eta +\xi(t)) \bar{\hat{u}}^l_{j'_2}(\sigma-\eta+\xi(t)) \hat{u}^h_{j'_3}(\sigma) d\sigma d\eta d\xi\\
=&\sum_{\mathcal{R}(j')}\iiint e^{i\xi x} [(\phi(\frac{\xi-\xi(t)}{2^{l_2}})-\phi(\frac{\sigma-\xi(t)}{2^{l_2}})) \phi(\frac{\xi-\eta}{2^{l_2}}) \phi(\frac{\eta-\sigma}{2^{l_2}}) \phi(\frac{\sigma-\xi(t)}{2^{l_2+10}})] \\
&\qquad\times \int e^{-iy(\xi-\eta +\xi(t))}u^l_{j'_1}(y)dy \int e^{-iz(\eta-\sigma-\xi(t))}\bar{u}^l_{j'_2}(z)dz \int e^{-iw\sigma} u^h_{j'_3}(w)dw d\sigma d\eta d\xi\\
=&\sum_{l_1\leq l_2}\sum_{\mathcal{R}(j')}\iiint e^{-iy\xi(t)}P_{\xi(t), l_1}u^l_{j'_1}(y)e^{iz\xi(t)}P_{\xi(t),\leq l_1} \bar{u}^l_{j'_2}(z) u^h_{j'_3}(w)\iiint e^{i\xi (x-y)+i\eta(y-z)+i\sigma(z-w)}  \\
&\times [(\phi(\frac{\xi-\xi(t)}{2^{l_2}})-\phi(\frac{\sigma-\xi(t)}{2^{l_2}}))\phi(\frac{\sigma-\xi(t)}{2^{l_2+10}}) \psi_{l_1}(\xi-\eta) \phi(\frac{\eta-\sigma}{2^{l_1}})]d\sigma d\eta d\xi dydzdw\\
+&\sum_{l_1\leq l_2}\sum_{\mathcal{R}(j')}\iiint e^{-iy\xi(t)}P_{\xi(t),\leq l_1}u^l_{j'_1}(y)e^{iz\xi(t)}P_{\xi(t), l_1} \bar{u}^l_{j'_2}(z) u^h_{j'_3}(w)\iiint e^{i\xi (x-y)+i\eta(y-z)+i\sigma(z-w)}  \\
&\times [(\phi(\frac{\xi-\xi(t)}{2^{l_2}})-\phi(\frac{\sigma-\xi(t)}{2^{l_2}}))\phi(\frac{\sigma-\xi(t)}{2^{l_2+10}}) \psi_{l_1}(\eta-\sigma) \phi(\frac{\xi-\eta}{2^{l_1}}) ]d\sigma d\eta d\xi dydzdw\\
=:&\eqref{eq-y5.77}(a)+\eqref{eq-y5.77}(b).
\end{split}
\end{equation}The estimates of \eqref{eq-y5.77}(a) and \eqref{eq-y5.77}(b) are the same, so we only estimate \eqref{eq-y5.77}(a).
\begin{equation*}
\begin{split}
&\eqref{eq-y5.77}(a)\\
=:&\sum_{l_1\leq l_2}\sum_{\mathcal{R}(j')}\iiint e^{-iy\xi(t)}P_{\xi(t), l_1}u^l_{j'_1}(y) e^{iz\xi(t)}P_{\xi(t),\leq l_1}\bar{u}^l_{j'_2}(z) u^h_{j'_3}(w)K(t;x-y,y-z,z-w)dydzdw\\
=&\sum_{l_1\leq l_2}\sum_{\mathcal{R}(j')}\iiint K(t;y,z,w) \\
&\times e^{-i(x-y)\xi(t)}P_{\xi(t), l_1}u^l_{j'_1}(x-y) e^{i(x-y-z)\xi(t)}P_{\xi(t),\leq l_1}\bar{u}^l_{j'_2}(x-y-z) u^h_{j'_3}(x-y-z-w)dydzdw,
\end{split}
\end{equation*}
where
\begin{equation*}
\begin{split}
&K(t;y,z,w)\\
=&\iiint e^{i\xi y+i\eta z+i\sigma w} [(\phi(\frac{\xi-\xi(t)}{2^{l_2}})-\phi(\frac{\sigma-\xi(t)}{2^{l_2}}))\phi(\frac{\sigma-\xi(t)}{2^{l_2+10}}) \psi_{l_1}(\xi-\eta) \phi(\frac{\eta-\sigma}{2^{l_1}}) ]d\xi d\eta d\sigma.
\end{split}
\end{equation*}
Notice that $|\xi-\eta|\sim 2^{l_1}$ and $|\eta-\sigma|\lesssim 2^{l_1}$, then by the fundamental theorem of calculus, we have
\begin{equation}\label{eq-y5.78'}
|\phi(\frac{\xi-\xi(t)}{2^{l_2}})-\phi(\frac{\sigma-\xi(t)}{2^{l_2}})|\lesssim 2^{-l_2}|\xi-\sigma|\lesssim 2^{-l_2}[|\xi-\eta|+|\eta-\sigma|]\lesssim 2^{l_1-l_2}.
\end{equation}
So by scaling we can easily see
\begin{equation}\label{eq-y5.78}
 \sup_t\int|K(t;y,z,w)|dydzdw\lesssim 2^{l_1-l_2}.
\end{equation}
Therefore, by Minkowski inequality, H\"{o}lder inequality, \eqref{eq-y5.78}, Lemma \ref{le-y5.7} and conservation of mass,
\begin{equation}\label{eq-y5.79}
\begin{split}
&\|P_{\xi(t),\leq l_2}\big(\sum^{\rightarrow}_{\mathcal{R}(j')}u^l_{j'_1}\bar{u}^l_{j'_2}u^h_{j'_3}\big)- \sum^{\rightarrow}_{\mathcal{R}(j')}P_{\xi(t),\leq l_2} u^l_{j'_1}P_{\xi(t),\leq l_2} \bar{u}^l_{j'_2} P_{\xi(t),\leq l_2} u^h_{j'_3}\|_{L^{4/3}_{t,x}l^{2}(G_{\beta}^{l_2}\times \mathbb{R}^2)}\\
\lesssim &\sum_{l_1\leq l_2}\Big\|\iiint |K(t;y,z,w)| \\
&\times \big\|\|P_{\xi(t), l_1}\vec{u}^l\|_{l^{2}}(x-y) \|P_{\xi(t),\leq l_1}\vec{u}^l\|_{l^{2}}(x-y-z) \|\vec{u}^h\|_{l^{2}}(x-y-z-w)\big\|_{L^{4/3}_x}dydzdw\Big\|_{L^{4/3}_t}\\
\lesssim &\sum_{l_1\leq l_2} 2^{l_1-l_2} \|P_{\xi(t),\leq l_1}\vec{u}^l\|_{L_t^{\infty}L^2_xl^2} \|P_{\xi(t), l_1}\vec{u}^l\|_{L_t^{8/3}L^{8}_xl^2} \|\vec{u}^h\|_{L_t^{8/3}L^{8}_xl^2}\\
\lesssim&\|\vec{u}\|^2_{\tilde{X}_{i}(G_{\alpha}^{i})}.
\end{split}
\end{equation}
We can similarly estimate \eqref{eq-y5.75}, there is no need to repeat the arguments here. \\
Finally let's take \eqref{eq-y5.76}.
\begin{equation*}
\begin{split}
&\Big\|O \big(\sum^{\rightarrow}_{\mathcal{R}(j')}P_{\xi(t),\leq l_2}(u^h_{j'_1}\bar{u}^h_{j'_2}u_{j'_3})- P_{\xi(t),\leq l_2} u^h_{j'_1}P_{\xi(t),\leq l_2} \bar{u}^h_{j'_2} P_{\xi(t),\leq l_2} u_{j'_3}\big)\Big\|_{L^{4/3}_{t,x}l^{2}(G_{\beta}^{l_2}\times \mathbb{R}^2)}\\
\lesssim&\|\vec{u}\|_{L_t^{\infty}L^2_xl^2} \|\vec{u}^h\|_{L_t^{8/3}L^{8}_xl^2} \|\vec{u}^h\|_{L_t^{8/3}L^{8}_xl^2}\\
\lesssim&\|\vec{u}\|^2_{\tilde{X}_{i}(G_{\alpha}^{i})}.
\end{split}
\end{equation*}Therefore,
\begin{equation}\label{eq-y5.80'}
\|\vec{N}_1\|_{L^{4/3}_{t,x}l^{2}(G_{\beta}^{l_2}\times \mathbb{R}^2)}\lesssim \|\vec{u}\|^2_{\tilde{X}_{i}(G_{\alpha}^{i})},
\end{equation} and
\begin{equation}\label{eq-y5.80}
\begin{split}
|\eqref{eq-y5.68}| \lesssim 2^{2l_2-2i}\|\vec{v}_0\|^2_{L^2_x l^{2}(\mathbb{R}^2\times \mathbb{Z})} (\|\vec{u}\|^3_{\tilde{X}_{i}(G_{\alpha}^{i})} +1).
\end{split}
\end{equation}
The right-hand side of this term is clearly bounded by the right-hand side of \eqref{eq-y5.67}.

Now we turn to \eqref{eq-y5.69}, integrating by parts in space, we have
\begin{equation}\label{eq-y5.81}
\begin{split}
 \eqref{eq-y5.69}\lesssim & \eqref{eq-y5.68} +2^{l_2-2i}\big|\sum_{j,j'\in \mathbb{Z}}   \int_{G_{\beta}^{l_2}}\int\int |v_{j}(t,y)|^2 \frac{1}{|x-y|} Re[\bar{w}_{j'}N_{j'}]dxdydt\big|.
\end{split}
\end{equation}
By  Strichartz estimate \eqref{eq-y2.1}, \eqref{eq-y5.80'}, \eqref{eq-y5.1}, \eqref{eq-y5.33}, Bernstein inequality and conservation of mass, we have
\begin{equation}\label{eq-y5.82}
\begin{split}
&2^{l_2-2i}\big|\sum_{j,j'\in \mathbb{Z}}   \int_{G_{\beta}^{l_2}}\int\int |v_{j}(t,y)|^2 \frac{1}{|x-y|} Re[\bar{w}_{j'}N_{j'}]dxdydt\big|\\
\lesssim& 2^{l_2-2i} \|\vec{v}_0\|^2_{L_x^2 l^{2}(\mathbb{R}^2\times \mathbb{Z})} \big(2^{-l_2/2}\int_{G_{\beta}^{l_2}} |\xi'(t)| \|P_{\xi(t),l_2-3 \leq \cdot\leq l_2+3} \vec{u}\|_{L_x^2l^{2}(\mathbb{R}^2\times \mathbb{Z})} \|(\nabla-i\xi(t))^{1/2}P_{\xi(t),\leq l_2} \vec{u}\|_{L_x^2l^{2}(\mathbb{R}^2\times \mathbb{Z})} dt\big)\\
&+2^{l_2-2i} \|\vec{v}\|^2_{L^4_t L^4_x l^{2}(G_{\beta}^{l_2}\times\mathbb{R}^2)} \|\vec{N}_1\|_{L^{4/3}_tL^{4/3}_xl^2(G_{\beta}^{l_2}\times\mathbb{R}^2)} \|\vec{w}\|_{L^{\infty}_{t,x}l^{2}(G_{\beta}^{l_2}\times\mathbb{R}^2)}\\
\lesssim& 2^{2l_2-2i}\|\vec{v}_0\|^2_{L^2_x l^{2}(\mathbb{R}^2\times \mathbb{Z})} (\|\vec{u}\|^3_{\tilde{X}_{i}(G_{\alpha}^{i})} +1).
\end{split}
\end{equation}
Therefore,
\begin{equation}\label{eq-y5.82'}
\eqref{eq-y5.69}\lesssim 2^{2l_2-2i}\|\vec{v}_0\|^2_{L^2_x l^{2}(\mathbb{R}^2\times \mathbb{Z})} (\|\vec{u}\|^3_{\tilde{X}_{i}(G_{\alpha}^{i})} +1).
\end{equation}
The term \eqref{eq-y5.70} is much more difficult to estimate, we give its estimate though the following lemma.
\begin{lemma}\label{le-y5.18}
\begin{equation}\label{eq-y5.83}
\eqref{eq-y5.70}\lesssim \|\vec{v}_0\|^2_{L^2_x l^{2}(\mathbb{R}^2\times \mathbb{Z})} (\|\vec{u}\|^4_{\tilde{X}_{i}(G_{\alpha}^{i})} +1).
\end{equation}
\end{lemma}
\begin{proof} Recall
\begin{equation*}
\begin{split}
&\eqref{eq-y5.70}\\
=& 2^{l_2-2i}\big| \sum_{j,j'\in \mathbb{Z}}   \int_{G_{\beta}^{l_2}} \int_{\mathbb{R}^2} \int_{\mathbb{R}^2} Im[\bar{v}_{j} (\nabla- i\xi(t))v_{j}](t,x) \frac{(x-y)}{|x-y|}\cdot Im[\bar{w}_{j'}N_{j'}](t,y)dxdydt \big|,
\end{split}
\end{equation*}
and
$$N_{j'}=N_{1,j'}+N_{2,j'}.$$
First for $N_{2,j'}$ part, by  Bernstein inequality and conservation of mass, we have
\begin{equation}\label{eq-y5.84}
\begin{split}
 & 2^{l_2-2i} \big| \sum_{j,j'\in \mathbb{Z}}   \int_{G_{\beta}^{l_2}} \int_{\mathbb{R}^2} \int_{\mathbb{R}^2} Im[\bar{v}_{j} (\nabla- i\xi(t))v_{j}](t,x) \frac{(x-y)}{|x-y|}\cdot Im[\bar{w}_{j'}N_{2,j'}](t,y)dxdydt \big|\\
 \lesssim& \|\vec{v}_0\|^2_{L_x^2 l^{2}(\mathbb{R}^2\times \mathbb{Z})} \big(2^{-i-l_2}\int_{G_{\beta}^{l_2}} |\xi'(t)| \|P_{\xi(t),l_2-3 \leq \cdot\leq l_2+3} \vec{u}\|_{L_x^2l^{2}(\mathbb{R}^2\times \mathbb{Z})} \|(\nabla-i\xi(t))P_{\xi(t),\leq l_2} \vec{u}\|_{L_x^2l^{2}(\mathbb{R}^2\times \mathbb{Z})} dt\big)\\
 &+ \|\vec{v}_0\|^2_{L_x^2 l^{2}(\mathbb{R}^2\times \mathbb{Z})} \big(2^{-i-\frac{l_2}{2}}\int_{G_{\beta}^{l_2}} |\xi'(t)| \|P_{\xi(t),l_2-3 \leq \cdot\leq l_2+3} \vec{u}\|_{L_x^2l^{2}(\mathbb{R}^2\times \mathbb{Z})} \|(\nabla-i\xi(t))^{1/2}P_{\xi(t),\leq l_2} \vec{u}\|_{L_x^2l^{2}(\mathbb{R}^2\times \mathbb{Z})} dt\big)\\
 \lesssim& 2^{l_2-i}\|\vec{v}_0\|^2_{L_x^2 l^{2}(\mathbb{R}^2\times \mathbb{Z})}\lesssim \|\vec{v}_0\|^2_{L_x^2 l^{2}(\mathbb{R}^2\times \mathbb{Z})}.
\end{split}
\end{equation}
Next for $N_{1,j'}$ part, note that
\begin{equation}\label{eq-y5.85'}
\begin{split}
&Im\Big[\sum_{j'} \bar{w}_{j'} \sum_{\mathcal{R}(j')}w_{j'_1}\bar{w}_{j'_2}w_{j'_3}\Big]\\
&=\frac{1}{2}\sum\limits_{\substack{ j'_1+j'_3 = j'+j'_2,\\ |j'_1|^2+|j'_3|^2 = |j'|^2+|j'_2|^2.}} Im\big[\bar{w}_{j'} w_{j'_1}\bar{w}_{j'_2}w_{j'_3}+w_{j'} \bar{w}_{j'_1} w_{j'_2} \bar{w}_{j'_3}\big]\\
&=0.
\end{split}
\end{equation}Therefore,
\begin{equation*}
\begin{split}
\sum_{j'} Im\big[ \bar{w}_{j'} N_{1,j'}\big]&=Im\Big[\sum_{j'} \bar{w}_{j'} P_{\xi(t),\leq l_2}\big(\sum_{\mathcal{R}(j')}u_{j'_1}\bar{u}_{j'_2}u_{j'_3}\big)\Big]-Im\Big[ \sum_{j'} \bar{w}_{j'} \sum_{\mathcal{R}(j')}w_{j'_1}\bar{w}_{j'_2}w_{j'_3}\Big]\\
&=Im\Big[\sum_{j'} \bar{w}_{j'} P_{\xi(t),\leq l_2}\big(\sum_{\mathcal{R}(j')}u_{j'_1}\bar{u}_{j'_2}u_{j'_3}\big)\Big].
\end{split}
\end{equation*}
We partition $u_{j'} = u_{j'}^h + u_{j'}^l$, where $u_{j'}^l = P_{\xi(t),\leq l_2 -5} u_{j'}$. Decompose
\begin{equation}\label{eq-y5.85}
\begin{split}
\sum_{j'} Im\big[ \bar{w}_{j'} N_{1,j'}\big]=&\sum_{j'} Im\Big[\bar{w}_{j'} P_{\xi(t),\leq l_2}\big(\sum_{\mathcal{R}(j')}u_{j'_1}\bar{u}_{j'_2}u_{j'_3}\big)\Big]\\
=&\sum_{j'}\big[ F_{0,j'}(t,y)+F_{1,j'}(t,y)+F_{2,j'}(t,y)+F_{3,j'}(t,y)+F_{4,j'}(t,y)\big].
\end{split}
\end{equation}
where $F_{0,j'}$ consists of the terms in $Im[\bar{w}_{j'} P_{\xi(t),\leq l_2}\big(\sum_{\mathcal{R}(j')}u_{j'_1}\bar{u}_{j'_2}u_{j'_3}\big)]$ with no $u_{j'}^h$ terms and four $u_{j'}^l$ terms, $F_{1,j'}$ has one $u_{j'}^h$ term and three $u_{j'}^l$ terms, and so on.

It follows from \eqref{eq-y5.85'}, $\sum_{j'} F_{0,j'}(t,y)=0$. Next by Bernstein inequality, \eqref{eq-y2.6} and  Lemma \ref{le-y5.7}, we have
\begin{equation}\label{eq-y5.86}
\begin{split}
& 2^{l_2-2i} \big| \sum_{j,j'\in \mathbb{Z}}   \int_{G_{\beta}^{l_2}} \int_{\mathbb{R}^2} \int_{\mathbb{R}^2} Im[\bar{v}_{j} (\nabla- i\xi(t))v_{j}](t,x) \frac{(x-y)}{|x-y|}\cdot (F_{3,j'}+F_{4,j'})(t,y)dxdydt \big|\\
\lesssim&2^{l_2-i}\|\vec{v}_0\|^2_{L^2_xl^2} \|\vec{u}^h\|^3_{L^3_{t}L^6_{x}l^{2}(G_{\beta}^{l_2}\times \mathbb{R}^2)} \|\vec{u}\|_{L^{\infty}_{t}L^2_{x}l^{2}(G_{\beta}^{l_2}\times \mathbb{R}^2)}\\
\lesssim&2^{l_2-i}\|\vec{v}_0\|^2_{L^2_xl^2} \|\vec{u}\|^3_{\tilde{X}_i(G^i_{\alpha})}.
\end{split}
\end{equation}Now we turn to $F_{1,j'}$,
\begin{equation}\label{eq-y5.86'}
\begin{split}
  &\sum_{j'}F_{1,j'}\\
  =&\sum_{j'}Im\Big[(P_{\xi(t),\leq l_2}\bar{u}^h_{j'}) P_{\xi(t),\leq l_2} \big(\sum_{\mathcal{R}(j')}u^l_{j'_1}\bar{u}^l_{j'_2}u^l_{j'_3}\big) +\bar{u}^l_{j'}P_{\xi(t),\leq l_2}\big(\sum_{\mathcal{R}(j')}u^l_{j'_1}\bar{u}^h_{j'_2}u^l_{j'_3} + 2\sum_{\mathcal{R}(j')}u^l_{j'_1}\bar{u}^l_{j'_2}u^h_{j'_3}\big)\Big] \\
  =& \sum_{j'} Im\big[(P_{\xi(t),\leq l_2}P_{\xi(t),\leq l_2-2}\bar{u}^h_{j'}) P_{\xi(t),\leq l_2} \big(\sum_{\mathcal{R}(j')}u^l_{j'_1}\bar{u}^l_{j'_2}u^l_{j'_3}\big)\\
  &+\bar{u}^l_{j'}P_{\xi(t),\leq l_2}\big(\sum_{\mathcal{R}(j')}u^l_{j'_1}(P_{\xi(t),\leq l_2-2}\bar{u}^h_{j'_2})u^l_{j'_3} + 2\sum_{\mathcal{R}(j')}u^l_{j'_1}\bar{u}^l_{j'_2}(P_{\xi(t),\leq l_2-2}u^h_{j'_3})\big)\big]\\
  &+\sum_{j'} Im\big[(P_{\xi(t),\leq l_2}P_{\xi(t),\geq l_2-2}\bar{u}^h_{j'}) P_{\xi(t),\leq l_2} \big(\sum_{\mathcal{R}(j')}u^l_{j'_1}\bar{u}^l_{j'_2}u^l_{j'_3}\big)\\
  &+\bar{u}^l_{j'}P_{\xi(t),\leq l_2}\big(\sum_{\mathcal{R}(j')}u^l_{j'_1} (P_{\xi(t),\geq l_2-2}\bar{u}^h_{j'_2}) u^l_{j'_3} + 2\sum_{\mathcal{R}(j')}u^l_{j'_1} \bar{u}^l_{j'_2} (P_{\xi(t),\geq l_2-2}u^h_{j'_3})\big)\big]\\
  =&2\sum\limits_{\substack{ j'_1+j'_3 = j'+j'_2,\\ |j'_1|^2+|j'_3|^2 = |j'|^2+|j'_2|^2.}} Im\Big[(P_{\xi(t),\leq l_2}\bar{u}^h_{j'}) u^l_{j'_1} \bar{u}^l_{j'_2} u^l_{j'_3}+(P_{\xi(t),\leq l_2}u^h_{j'})
  \bar{u}^l_{j'_1} u^l_{j'_2} \bar{u}^l_{j'_3} \Big]\\
  &+\sum_{j'} Im\big[(P_{\xi(t),\leq l_2}P_{\xi(t),\geq l_2-2}\bar{u}^h_{j'})P_{\xi(t),\leq l_2} \big(\sum_{\mathcal{R}(j')}u^l_{j'_1}\bar{u}^l_{j'_2}u^l_{j'_3}\big)\\
  &+\bar{u}^l_{j'}P_{\xi(t),\leq l_2}\big(\sum_{\mathcal{R}(j')}u^l_{j'_1}(P_{\xi(t),\geq l_2-2}\bar{u}^h_{j'_2})u^l_{j'_3} + 2\sum_{\mathcal{R}(j')}u^l_{j'_1}\bar{u}^l_{j'_2}(P_{\xi(t),\geq l_2-2}u^h_{j'_3})\big)\big]\\
  =&\sum_{j'} Im\big[(P_{\xi(t),\leq l_2}P_{\xi(t),\geq l_2-2}\bar{u}^h_{j'})P_{\xi(t),\leq l_2} \big(\sum_{\mathcal{R}(j')}u^l_{j'_1}\bar{u}^l_{j'_2}u^l_{j'_3}\big)\\
  &+\bar{u}^l_{j'}P_{\xi(t),\leq l_2}\big(\sum_{\mathcal{R}(j')}u^l_{j'_1}(P_{\xi(t),\geq l_2-2}\bar{u}^h_{j'_2})u^l_{j'_3} + 2\sum_{\mathcal{R}(j')}u^l_{j'_1}\bar{u}^l_{j'_2}(P_{\xi(t),\geq l_2-2}u^h_{j'_3})\big)\big].
\end{split}
\end{equation}
This implies that the Fourier support of $\sum_{j'}F_{1,j'}(t,y)$ is on $|\xi|\geq 2^{l_2-4}$ (notice again the Fourier support is not $|\xi-\xi(t)|\geq 2^{l_2-4}$), Integrating by parts about the spatial variable $y$, then using Hardy-Littlewood-Sobolev inequality,  Bernstein inequality,  Strichartz estimate and \eqref{eq-y5.9'}, we have
\begin{equation}\label{eq-y5.87}
\begin{split}
& 2^{l_2-2i} \big| \sum_{j\in \mathbb{Z}}   \int_{G_{\beta}^{l_2}} \int_{\mathbb{R}^2} \int_{\mathbb{R}^2} Im[\bar{v}_{j} (\nabla- i\xi(t))v_{j}](t,x) \frac{(x-y)}{|x-y|}\cdot \big[\frac{\triangle_y}{\triangle_y}(\sum_{j'}F_{1,j'})\big](t,y)dxdydt \big|\\
\lesssim& 2^{l_2-2i} \int_{G_{\beta}^{l_2}} \int_{\mathbb{R}^2} \int_{\mathbb{R}^2} \big|\sum_{j\in \mathbb{Z}} [\bar{v}_{j} (\nabla- i\xi(t))v_{j}]\big|(t,x) \frac{1}{|x-y|}\cdot \big|\frac{\partial_y}{\triangle_y}(\sum_{j'}F_{1,j'})\big|(t,y)dxdydt\\
\lesssim& 2^{l_2-2i}\|\vec{v}\|_{L^3_tL^{6}_xl^2(G_{\beta}^{l_2}\times\mathbb{R}^2)} \|(\nabla- i\xi(t))\vec{v}\|_{L^{\infty}_tL^{2}_xh^{0}(G_{\beta}^{l_2}\times\mathbb{R}^2)} \big\|\frac{\partial_y}{\triangle_y}(\sum_{j'}F_{1,j'})\big\|_{L^{3/2}_tL^{6/5}_x(G_{\beta}^{l_2}\times\mathbb{R}^2)}\\
\lesssim&2^{-i}\|\vec{v}_0\|^2_{L^2_xl^2} \|\vec{u}^l\|^3_{L^9_{t}L^{9/2}_{x}l^{2}(G_{\beta}^{l_2}\times \mathbb{R}^2)} \|\vec{u}^h\|_{L^{3}_{t}L^6_{x}l^{2}(G_{\beta}^{l_2}\times \mathbb{R}^2)}\\
\lesssim&2^{-i}\|\vec{v}_0\|^2_{L^2_xl^2} \|\vec{u}^l\|^3_{L^9_{t}L^{9/2}_{x}l^{2}(G_{\beta}^{l_2}\times \mathbb{R}^2)} \|\vec{u}\|_{\tilde{X}_i(G^i_{\alpha})}.
\end{split}
\end{equation}
However, by Bernstein inequality and \eqref{eq-y5.9},
\begin{equation}\label{eq-y5.87'}
\begin{split}
 \|\vec{u}^l\|_{L^9_{t}L^{9/2}_{x}l^{2}(G_{\beta}^{l_2}\times \mathbb{R}^2)} \lesssim\sum_{0\leq l_1\leq l_2} 2^{l_1/3} \|P_{\xi(t),l_1}\vec{u}\|_{L^9_{t}L^{18/7}_{x}l^{2}(G_{\beta}^{l_2}\times \mathbb{R}^2)}\lesssim  2^{l_2/3} \|\vec{u}\|_{\tilde{X}_i(G^i_{\alpha})}.
\end{split}
\end{equation}
Therefore,
\begin{equation}\label{eq-y5.88'}
  \eqref{eq-y5.87}\lesssim 2^{l_2-i}\|\vec{v}_0\|^2_{L^2_xl^2} \|\vec{u}\|^4_{\tilde{X}_i(G^i_{\alpha})}.
\end{equation}
Finally let's consider $F_{2,j'}$. We want to prove
\begin{equation}\label{eq-y5.88}
\begin{split}
&2^{l_2-2i} \big| \sum_{j,j'\in \mathbb{Z}}   \int_{G_{\beta}^{l_2}} \int_{\mathbb{R}^2} \int_{\mathbb{R}^2} Im[\bar{v}_{j} (\nabla- i\xi(t))v_{j}](t,x) \frac{(x-y)}{|x-y|}\cdot F_{2,j'}(t,y)dxdydt \big|\\
&\lesssim \|\vec{v}_0\|^2_{L^2_x l^{2}(\mathbb{R}^2\times \mathbb{Z})} (\|\vec{u}\|^4_{\tilde{X}_{i}(G_{\alpha}^{i})} +1).
\end{split}
\end{equation}
As in the analysis of $F_{1,j'}$, integrating by parts, by Bernstein inequality and \eqref{eq-y5.9'}, we have
\begin{equation}\label{eq-y5.89}
\begin{split}
& 2^{l_2-2i} \big| \sum_{j\in \mathbb{Z}}   \int_{G_{\beta}^{l_2}} \int_{\mathbb{R}^2} \int_{\mathbb{R}^2} Im[\bar{v}_{j} (\nabla- i\xi(t))v_{j}](t,x) \frac{(x-y)}{|x-y|}\cdot \big[\frac{\triangle_y}{\triangle_y} (\sum_{j'} P_{\geq l_2-10}F_{2,j'})\big](t,y)dxdydt \big|\\
\lesssim& 2^{l_2-2i}\|\vec{v}\|_{L^3_tL^{6}_xl^2(G_{\beta}^{l_2}\times\mathbb{R}^2)} \|(\nabla- i\xi(t))\vec{v}\|_{L^{\infty}_tL^{2}_xl^2(G_{\beta}^{l_2}\times\mathbb{R}^2)} \big\|\frac{\partial_y}{\triangle_y}(\sum_{j'} P_{\geq l_2-10}F_{2,j'})\big\|_{L^{3/2}_tL^{6/5}_x(G_{\beta}^{l_2}\times\mathbb{R}^2)}\\
\lesssim&2^{-i}\|\vec{v}_0\|^2_{L^2_xl^2} \|\vec{u}^l\|^2_{L^{\infty}_{t}L^{4}_{x}l^{2}(G_{\beta}^{l_2}\times \mathbb{R}^2)} \|\vec{u}^h\|^2_{L^{3}_{t}L^6_{x}l^{2}(G_{\beta}^{l_2}\times \mathbb{R}^2)}\\
\lesssim&2^{l_2-i}\|\vec{v}_0\|^2_{L^2_xl^2} \|\vec{u}\|^2_{\tilde{X}_i(G^i_{\alpha})}.
\end{split}
\end{equation}
Next, expanding out $F_{2,j'}$,
\begin{equation}\label{eq-y5.90}
\begin{split}
F_{2,j'}(t,y)=& Im\big[ 2\bar{u}^l_{j'} P_{\xi(t),\leq l_2} \big(\sum_{\mathcal{R}(j')}u^h_{j'_1}\bar{u}^h_{j'_2}u^l_{j'_3}\big) + \bar{u}^l_{j'} P_{\xi(t),\leq l_2} \big(\sum_{\mathcal{R}(j')}u^h_{j'_1}\bar{u}^l_{j'_2}u^h_{j'_3}\big)\\
&+ 2(P_{\xi(t),\leq l_2}\bar{u}^h_{j'}) P_{\xi(t),\leq l_2} \big(\sum_{\mathcal{R}(j')}u^l_{j'_1}\bar{u}^l_{j'_2}u^h_{j'_3}\big) + (P_{\xi(t),\leq l_2}\bar{u}^h_{j'}) P_{\xi(t),\leq l_2} \big(\sum_{\mathcal{R}(j')}u^l_{j'_1}\bar{u}^h_{j'_2}u^l_{j'_3}\big)\big].
\end{split}
\end{equation}
Observe that
\begin{equation*}
\begin{split}
&Im \Big[\sum_{j'}\sum_{\mathcal{R}(j')} (P_{\xi(t),\leq l_2}\bar{u}^h_{j'})  u^l_{j'_1} \bar{u}^l_{j'_2} (P_{\xi(t),\leq l_2}u^h_{j'_3})\Big]\\
=&Im \Big[\sum\limits_{\substack{ j'_1+j'_3 = j'+j'_2,\\ |j'_1|^2+|j'_3|^2 = |j'|^2+|j'_2|^2.}}  (P_{\xi(t),\leq l_2}\bar{u}^h_{j'})  u^l_{j'_1} \bar{u}^l_{j'_2} (P_{\xi(t),\leq l_2}u^h_{j'_3})\Big]\\
=&\frac{1}{2}\sum\limits_{\substack{ j'_1+j'_3 = j'+j'_2,\\ |j'_1|^2+|j'_3|^2 = |j'|^2+|j'_2|^2.}} Im\Big[(P_{\xi(t),\leq l_2}\bar{u}^h_{j'})  u^l_{j'_1} \bar{u}^l_{j'_2} (P_{\xi(t),\leq l_2}u^h_{j'_3}) + (P_{\xi(t),\leq l_2}u^h_{j'})  \bar{u}^l_{j'_1} u^l_{j'_2} (P_{\xi(t),\leq l_2}\bar{u}^h_{j'_3}) \Big]\\
=&0.
\end{split}
\end{equation*}
Similarly,
\begin{equation*}
 Im \Big[\sum_{j'}\sum_{\mathcal{R}(j')} \bar{u}^l_{j'} u^l_{j'_3} P_{\leq l_2}(u^h_{j'_1}\bar{u}^h_{j'_2})\Big]=0
\end{equation*}
Therefore,
\begin{equation}\label{eq-y5.91}
\begin{split}
&2Im\big[\sum_{j'}  \bar{u}^l_{j'} P_{\xi(t),\leq l_2} \big(\sum_{\mathcal{R}(j')}u^h_{j'_1}\bar{u}^h_{j'_2}u^l_{j'_3}\big) + \sum_{j'}  (P_{\xi(t),\leq l_2}\bar{u}^h_{j'}) P_{\xi(t),\leq l_2} \big(\sum_{\mathcal{R}(j')}u^l_{j'_1}\bar{u}^l_{j'_2}u^h_{j'_3}\big)\big]\\
=&2Im\big[\sum_{j'}\sum_{\mathcal{R}(j')} \big(\bar{u}^l_{j'}P_{\xi(t),\leq l_2}(u^h_{j'_1}\bar{u}^h_{j'_2}u^l_{j'_3}) -\bar{u}^l_{j'} u^l_{j'_3} (P_{\leq l_2} (u^h_{j'_1}\bar{u}^h_{j'_2}))\big)\big]\\
&+2Im\big[\sum_{j'}\sum_{\mathcal{R}(j')} \big((P_{\xi(t),\leq l_2}\bar{u}^h_{j'})P_{\xi(t),\leq l_2} (u^l_{j'_1}\bar{u}^l_{j'_2}u^h_{j'_3}) - (P_{\xi(t),\leq l_2}\bar{u}^h_{j'})  u^l_{j'_1} \bar{u}^l_{j'_2} (P_{\xi(t),\leq l_2}u^h_{j'_3}) \big)\big]\\
=:&\eqref{eq-y5.91}(a)+\eqref{eq-y5.91}(b).
\end{split}
\end{equation}
By \eqref{eq-y5.77}, \eqref{eq-y5.78}, Lemma \ref{le-y5.7} and conservation of mass, \eqref{eq-y5.91}(b) can be estimated as following
\begin{equation}\label{eq-y5.92}
\begin{split}
& 2^{l_2-i}\|\eqref{eq-y5.91}(b)\|_{L^1_{t,x}(G_{\beta}^{l_2}\times \mathbb{R}^2)}\\
\lesssim& 2^{l_2-i}\|\vec{u}^h\|_{L^4_{t,x}l^{2}(G_{\beta}^{l_2}\times \mathbb{R}^2)} \|\sum^{\rightarrow}_{\mathcal{R}(j')} P_{\xi(t),\leq l_2}(u^l_{j'_1}\bar{u}^l_{j'_2}u^h_{j'_3})- \sum^{\rightarrow}_{\mathcal{R}(j')} u^l_{j'_1}\bar{u}^l_{j'_2} P_{\xi(t),\leq l_2} u^h_{j'_3}\|_{L^{4/3}_{t,x}l^{2}(G_{\beta}^{l_2}\times \mathbb{R}^2)}\\
\lesssim &2^{l_2-i}\|\vec{u}\|_{\tilde{X}_{i}(G_{\alpha}^{i})} \sum_{l_1\leq l_2}\Big\|\iiint |K(t;y,z,w)| \\
&\times \big\|\|P_{\xi(t), l_1}\vec{u}^l\|_{l^{2}}(x-y) \|P_{\xi(t),\leq l_1}\vec{u}^l\|_{l^{2}}(x-y-z) \|\vec{u}^h\|_{l^{2}}(x-y-z-w)\big\|_{L^{4/3}_x}dydzdw\Big\|_{L^{4/3}_t}\\
\lesssim &2^{l_2-i}\|\vec{u}\|_{\tilde{X}_{i}(G_{\alpha}^{i})} \sum_{l_1\leq l_2} 2^{l_1-l_2} \|P_{\xi(t),\leq l_1}\vec{u}^l\|_{L_t^{\infty}L^2_xl^2} \|P_{\xi(t), l_1}\vec{u}^l\|_{L_t^{8/3}L^{8}_xl^2} \|\vec{u}^h\|_{L_t^{8/3}L^{8}_xl^2}\\
\lesssim&2^{l_2-i}\|\vec{u}\|^3_{\tilde{X}_{i}(G_{\alpha}^{i})}.
\end{split}
\end{equation}

Next we turn to \eqref{eq-y5.91}(a). As \eqref{eq-y5.77} and \eqref{eq-y5.78'}, the fundamental Theorem of calculus gives
\begin{equation}\label{eq-y5.92'}
|\phi(\frac{\xi_1+\xi_2-\xi(t)}{2^{l_2}})-\phi(\frac{\xi_1}{2^{l_2}})|\lesssim 2^{-l_2}|\xi_2-\xi(t)|.
\end{equation}
So by Lemma \ref{le-y5.7} and conservation of mass, we have
\begin{equation}\label{eq-y5.92b}
\begin{split}
& 2^{l_2-i}\|\eqref{eq-y5.91}(a)\|_{L^1_{t,x}(G_{\beta}^{l_2}\times \mathbb{R}^2)}\\
\lesssim& 2^{l_2-i}\|\vec{u}^l\|_{L^{\infty}_{x}L^2_xl^2(G_{\beta}^{l_2}\times \mathbb{R}^2)} \|\sum^{\rightarrow}_{\mathcal{R}(j')}P_{\xi(t),\leq l_2}(u^h_{j'_1}\bar{u}^h_{j'_2}u^l_{j'_3}) - \sum^{\rightarrow}_{\mathcal{R}(j')} (P_{\leq l_2} (u^h_{j'_1}\bar{u}^h_{j'_2})u^l_{j'_3})\|_{L^{1}_{t} L^2_x l^{2}(G_{\beta}^{l_2}\times \mathbb{R}^2)}\\
\lesssim &2^{-i}\|\vec{u}^h\|^2_{L_t^{3}L^{6}_xl^2(G_{\beta}^{l_2}\times \mathbb{R}^2)} \sum_{l_1\leq l_2} 2^{l_1} \|P_{\xi(t), l_1}\vec{u}^l\|_{L_t^{3}L^6_xl^2(G_{\beta}^{l_2}\times \mathbb{R}^2)}\\
\lesssim&2^{l_2-i}\|\vec{u}\|^3_{\tilde{X}_{i}(G_{\alpha}^{i})}.
\end{split}
\end{equation}
Now let's deal with the rest terms in \eqref{eq-y5.90}.
\begin{equation}\label{eq-y5.92c}
\begin{split}
&Im\big[\sum_{j'}\sum_{\mathcal{R}(j')}  \bar{u}^l_{j'} P_{\xi(t),\leq l_2} \big(u^h_{j'_1}\bar{u}^l_{j'_2}u^h_{j'_3}\big)+ \sum_{j'}\sum_{\mathcal{R}(j')} (P_{\xi(t),\leq l_2}\bar{u}^h_{j'}) P_{\xi(t),\leq l_2} \big(u^l_{j'_1}\bar{u}^h_{j'_2}u^l_{j'_3}\big)\big]\\
=&\sum_{j'}\sum_{\mathcal{R}(j')} Im\big[\bar{u}^l_{j'} P_{\xi(t),\leq l_2} \big(u^h_{j'_1}\bar{u}^l_{j'_2}u^h_{j'_3}\big) + (P_{\xi(t),\leq l_2}\bar{u}^h_{j'}) u^l_{j'_1}(P_{\xi(t),\leq l_2}\bar{u}^h_{j'_2})u^l_{j'_3}\big]\\
&+\sum_{j'} \sum_{\mathcal{R}(j')} Im\big[(P_{\xi(t),\leq l_2}\bar{u}^h_{j'}) P_{\xi(t),\leq l_2} (u^l_{j'_1}\bar{u}^h_{j'_2}u^l_{j'_3})- (P_{\xi(t),\leq l_2}\bar{u}^h_{j'}) u^l_{j'_1}(P_{\xi(t),\leq l_2}\bar{u}^h_{j'_2})u^l_{j'_3} \big]\\
=:&\eqref{eq-y5.92c}(a)+\eqref{eq-y5.92c}(b).
\end{split}
\end{equation}
Similarly to \eqref{eq-y5.92}, we can get
\begin{equation}\label{eq-y5.92d}
2^{l_2-i}\|\eqref{eq-y5.92c}(b)\|_{L^1_{t,x}(G_{\beta}^{l_2}\times \mathbb{R}^2)}\lesssim2^{l_2-i}\|\vec{u}\|^3_{\tilde{X}_{i}(G_{\alpha}^{i})}.
\end{equation}
Therefore, in order to obtain \eqref{eq-y5.88}, which in turn accomplishes the proof of Lemma \ref{le-y5.18}, it remains to estimate \eqref{eq-y5.92c}(a).
\begin{equation}\label{eq-y5.115}
\begin{split}
&2^{l_2-2i}\int_{G^{l_2}_{\beta}}\int\int \sum_jIm[\bar{v}_j(\nabla-i\xi(t))v_j](t,x) \frac{(x-y)}{|x-y|}\\
&\times\sum_{j'}\sum_{\mathcal{R}(j')}P_{\leq l_2-10} Im\big[\bar{u}^l_{j'} P_{\xi(t),\leq l_2} (u^h_{j'_1}\bar{u}^l_{j'_2}u^h_{j'_3}) + (P_{\xi(t),\leq l_2}\bar{u}^h_{j'}) u^l_{j'_1}(P_{\xi(t),\leq l_2}\bar{u}^h_{j'_2})u^l_{j'_3}\big](t,y)dxdydt\\
=&2^{l_2-2i}\int_{G^{l_2}_{\beta}}\int\int \sum_jIm[\bar{v}_j(\nabla-i\xi(t))v_j](t,x) \frac{(x-y)}{|x-y|}\\
&\times\sum_{j'}\sum_{\mathcal{R}(j')}P_{\leq l_2-10} Im\big[\bar{u}^l_{j'} u^h_{j'_1}\bar{u}^l_{j'_2}u^h_{j'_3} + (P_{\xi(t),\leq l_2}\bar{u}^h_{j'}) u^l_{j'_1}(P_{\xi(t),\leq l_2}\bar{u}^h_{j'_2})u^l_{j'_3}\big](t,y)dxdydt\\
=&2^{l_2-2i}\int_{G^{l_2}_{\beta}}\int\int \sum_jIm[\bar{v}_j(\nabla-i\xi(t))v_j](t,x) \frac{(x-y)}{|x-y|}\\
&\times\sum_{j'}\sum_{\mathcal{R}(j')}P_{\leq l_2-10} Im\big[(\bar{u}^l_{j'_1}\bar{u}^l_{j'_3}) (u^h_{j'}u^h_{j'_2} - (P_{\xi(t),\leq l_2}u^h_{j'}) (P_{\xi(t),\leq l_2}u^h_{j'_2}))\big](t,y)dxdydt\\
\end{split}
\end{equation}
After Galilean transformation,
\begin{equation*}\label{eq-y5.118}
\begin{split}
\eqref{eq-y5.115}=&2^{l_2-2i}\int_{G^{l_2}_{\beta}}\int\int \sum_jIm[\bar{v}_j(\nabla-i\xi(t))v_j](t,x+2t\xi(G^{l_2}_{\beta})) \frac{(x-y)}{|x-y|}\\
\times&\sum_{j'}\sum_{\mathcal{R}(j')}P_{\leq l_2-10} Im\big[(\bar{u}^l_{j'_1}\bar{u}^l_{j'_3}) (u^h_{j'}u^h_{j'_2} - (P_{\xi(t),\leq l_2}u^h_{j'}) (P_{\xi(t),\leq l_2}u^h_{j'_2}))\big](t,y+2t\xi(G^{l_2}_{\beta}))dxdydt
\end{split}
\end{equation*}
Next observe that
\begin{equation}\label{eq-y5.116}
Im[\bar{v}_j(\nabla-i\xi(t))v_j](t,x+2t\xi(G^{l_2}_{\beta}))= Im[(\overline{e^{-ix\cdot\xi(G^{l_2}_{\beta})}v_j}) (\nabla-i\xi(t)+i\xi(G^{l_2}_{\beta})) e^{-ix\cdot\xi(G^{l_2}_{\beta})}v_j](t,x+2t\xi(G^{l_2}_{\beta})),
\end{equation}
and
\begin{equation}\label{eq-y5.117}
\begin{split}
&Im[(\bar{u}^l_{j'_1}\bar{u}^l_{j'_3})u^h_{j'}u^h_{j'_2}](t,y+2t\xi(G^{l_2}_{\beta}))\\
=& Im\big[(\overline{P_{\xi(t)-\xi(G^{l_2}_{\beta}),\leq l_2-5}e^{-iy\cdot\xi(G^{l_2}_{\beta})}u_{j'_1}}) (\overline{P_{\xi(t)-\xi(G^{l_2}_{\beta}),\leq l_2-5}e^{-iy\cdot\xi(G^{l_2}_{\beta})}u_{j'_3}})\\
\times &(P_{\xi(t)-\xi(G^{l_2}_{\beta}),\geq l_2-5}e^{-iy\cdot\xi(G^{l_2}_{\beta})}u_{j'}) (P_{\xi(t)-\xi(G^{l_2}_{\beta}),\geq l_2-5}e^{-iy\cdot\xi(G^{l_2}_{\beta})}u_{j'_2})\big](t,y+2t\xi(G^{l_2}_{\beta})).
\end{split}
\end{equation}
Similarly we can perform computation with the $(\bar{u}^l_{j'_1}\bar{u}^l_{j'_3})(P_{\xi(t),\leq l_2}u^h_{j'}) (P_{\xi(t),\leq l_2}u^h_{j'_2})$ term. Combining \eqref{eq-y5.118}, \eqref{eq-y5.116} and \eqref{eq-y5.117}, we can assume that $\xi(G^{l_2}_{\beta})=0$ in \eqref{eq-y5.115}. According to \eqref{eq-y5.1}, $|\xi(t)|\ll 2^{l_2}$ for all $t\in G^{l_2}_{\beta}$.

We will use the method of space-time resonances to estimate \eqref{eq-y5.115}. First by inverse Fourier transform,
\begin{equation}\label{eq-y5.119}
\begin{split}
&\sum_{j,j'}\sum_{\mathcal{R}(j')}\int_{G^{l_2}_{\beta}}\int\int Im[\bar{v}_j(\nabla-i\xi(t))v_j](t,x) \frac{(x-y)}{|x-y|}\\
\times&P_{\leq l_2-10} Im\big[(\bar{u}^l_{j'_1}\bar{u}^l_{j'_3}) (u^h_{j'}u^h_{j'_2} - (P_{\xi(t),\leq l_2}u^h_{j'}) (P_{\xi(t),\leq l_2}u^h_{j'_2}))\big](t,y)dxdydt\\
=&\sum_{j,j'}\sum_{\mathcal{R}(j')}\int_{G^{l_2}_{\beta}}\int\int Im[\bar{v}_j(\nabla-i\xi(t))v_j](t,x) \frac{(x-y)}{|x-y|}\\
&\times\big[\iiiint \phi(\frac{\eta_1+\eta_2+\eta_3+\eta_4}{2^{l_2-10}}) e^{iy\cdot(\eta_1+\eta_2+\eta_3+\eta_4)} \hat{u}^l_{j'_1}(t,\eta_3)\hat{u}^l_{j'_3}(t,\eta_4) \hat{\bar{u}}^h_{j'}(t,\eta_1) \hat{\bar{u}}^h_{j'_2}(t,\eta_2)\\
&\times \big(1-\phi(\frac{\eta_1-\xi(t)}{2^{l_2}})\phi(\frac{\eta_2-\xi(t)}{2^{l_2}})\big)d\eta_1 d\eta_2 d\eta_3 d\eta_4\big]dxdydt.
\end{split}
\end{equation}
Let
\begin{equation}\label{eq-y5.120}
q(\eta)=|\eta_1|^2 + |\eta_2|^2 - |\eta_3|^2 - |\eta_4|^2.
\end{equation}
Since $|\eta_3|,|\eta_4|\leq 2^{l_2-4}$, $|\eta_3|^2+|\eta_4|^2\leq 2^{2l_2-8}$, $|\eta_1|^2+|\eta_2|^2> 2^{2l_2}$ and $q(\eta)> 2^{2l_2-2}$ on the support of $\big(1-\phi(\frac{\eta_1-\xi(t)}{2^{l_2}})\phi(\frac{\eta_2-\xi(t)}{2^{l_2}})\big)\phi(\frac{\eta_1+\eta_2+\eta_3+\eta_4}{2^{l_2-10}})$. Therefore, on the support of $\big(1-\phi(\frac{\eta_1-\xi(t)}{2^{l_2}})\phi(\frac{\eta_2-\xi(t)}{2^{l_2}})\big) \phi(\frac{\eta_3}{2^{l_2-4}}) \phi(\frac{\eta_4}{2^{l_2-4}})$,
\begin{equation}\label{eq-y5.121}
\frac{1}{q(\eta)}=\frac{1}{|\eta_1|^2 + |\eta_2|^2}\sum^{\infty}_{k=0}(-1)^k \left( \frac{|\eta_3|^2 + |\eta_4|^2}{|\eta_1|^2 + |\eta_2|^2}\right)^k
\end{equation}
converges, and by Theorem C-M of \cite{GM}, $\frac{1}{q(\eta)}$ is a convergent sum of terms whose operator norm is $\lesssim \frac{1}{|\eta_1|^2 + |\eta_2|^2}\sim \frac{1}{|\eta_1||\eta_2|}$ (since $|\eta_3+\eta_4|\leq 2^{l_2-3}$ and $|\eta_1|,|\eta_2|\gtrsim 2^{l_2}$).  Integrating by parts in time, if $G^{l_2}_{\beta}= [a,b]$, then
\begin{align}\nonumber
&\int_{G^{l_2}_{\beta}}\iiint\iiint  \frac{1}{iq(\eta)}\big( \frac{d}{dt}e^{itq(\eta)}\big) Im[\bar{v}_j(\nabla-i\xi(t))v_j](t,x) \frac{(x-y)}{|x-y|}\\ \label{eq-y5.122}
&\times\phi(\frac{\eta_1+\eta_2+\eta_3+\eta_4}{2^{l_2-10}}) e^{iy\cdot(\eta_1+\eta_2+\eta_3+\eta_4)} \big(1-\phi(\frac{\eta_1-\xi(t)}{2^{l_2}})\phi(\frac{\eta_2-\xi(t)}{2^{l_2}})\big) \\ \nonumber
&\times \big[e^{-it|\eta_1|^2} \hat{\bar{u}}^h_{j'}(t,\eta_1) e^{-it|\eta_2|^2}\hat{\bar{u}}^h_{j'_2}(t,\eta_2) e^{it|\eta_3|^2}\hat{u}^l_{j'_1}(t,\eta_3) e^{it|\eta_4|^2}\hat{u}^l_{j'_3}(t,\eta_4)  d\eta_1 d\eta_2 d\eta_3 d\eta_4\big]dxdydt\\ \nonumber
=&\iiint\iiint  \frac{1}{iq(\eta)}e^{itq(\eta)} Im[\bar{v}_j(\nabla-i\xi(t))v_j](t,x) \frac{(x-y)}{|x-y|}\\ \label{eq-y5.123}
&\times\phi(\frac{\eta_1+\eta_2+\eta_3+\eta_4}{2^{l_2-10}}) e^{iy\cdot(\eta_1+\eta_2+\eta_3+\eta_4)} \big(1-\phi(\frac{\eta_1-\xi(t)}{2^{l_2}})\phi(\frac{\eta_2-\xi(t)}{2^{l_2}})\big) \\ \nonumber
&\times \big[e^{-it|\eta_1|^2} \hat{\bar{u}}^h_{j'}(t,\eta_1) e^{-it|\eta_2|^2}\hat{\bar{u}}^h_{j'_2}(t,\eta_2) e^{it|\eta_3|^2}\hat{u}^l_{j'_1}(t,\eta_3) e^{it|\eta_4|^2}\hat{u}^l_{j'_3}(t,\eta_4)  d\eta_1 d\eta_2 d\eta_3 d\eta_4\big]dxdy|^b_a\\
\nonumber
-&\int^b_a\iiint\iiint  \frac{1}{iq(\eta)}e^{itq(\eta)} \frac{\partial}{\partial t}Im[\bar{v}_j(\nabla-i\xi(t))v_j](t,x) \frac{(x-y)}{|x-y|}\\ \label{eq-y5.124}
&\times\phi(\frac{\eta_1+\eta_2+\eta_3+\eta_4}{2^{l_2-10}}) e^{iy\cdot(\eta_1+\eta_2+\eta_3+\eta_4)} \big(1-\phi(\frac{\eta_1-\xi(t)}{2^{l_2}})\phi(\frac{\eta_2-\xi(t)}{2^{l_2}})\big) \\ \nonumber
&\times \big[e^{-it|\eta_1|^2} \hat{\bar{u}}^h_{j'}(t,\eta_1) e^{-it|\eta_2|^2}\hat{\bar{u}}^h_{j'_2}(t,\eta_2) e^{it|\eta_3|^2}\hat{u}^l_{j'_1}(t,\eta_3) e^{it|\eta_4|^2}\hat{u}^l_{j'_3}(t,\eta_4)  d\eta_1 d\eta_2 d\eta_3 d\eta_4\big]dxdydt\\
\nonumber
-&\int^b_a\iiint\iiint  \frac{1}{iq(\eta)}e^{itq(\eta)} Im[\bar{v}_j(\nabla-i\xi(t))v_j](t,x) \frac{(x-y)}{|x-y|}\\ \label{eq-y5.125}
&\times\frac{\partial}{\partial t}\left[\phi(\frac{\eta_1+\eta_2+\eta_3+\eta_4}{2^{l_2-10}}) e^{iy\cdot(\eta_1+\eta_2+\eta_3+\eta_4)} \big(1-\phi(\frac{\eta_1-\xi(t)}{2^{l_2}})\phi(\frac{\eta_2-\xi(t)}{2^{l_2}})\big)\right] \\ \nonumber
&\times \big[e^{-it|\eta_1|^2} \hat{\bar{u}}^h_{j'}(t,\eta_1) e^{-it|\eta_2|^2}\hat{\bar{u}}^h_{j'_2}(t,\eta_2) e^{it|\eta_3|^2}\hat{u}^l_{j'_1}(t,\eta_3) e^{it|\eta_4|^2}\hat{u}^l_{j'_3}(t,\eta_4)  d\eta_1 d\eta_2 d\eta_3 d\eta_4\big]dxdydt\\
\nonumber
-&\int^b_a\iiint\iiint  \frac{1}{iq(\eta)}e^{itq(\eta)} Im[\bar{v}_j(\nabla-i\xi(t))v_j](t,x) \frac{(x-y)}{|x-y|}\\ \label{eq-y5.126}
&\times\phi(\frac{\eta_1+\eta_2+\eta_3+\eta_4}{2^{l_2-10}}) e^{iy\cdot(\eta_1+\eta_2+\eta_3+\eta_4)} \big(1-\phi(\frac{\eta_1-\xi(t)}{2^{l_2}})\phi(\frac{\eta_2-\xi(t)}{2^{l_2}})\big) \\ \nonumber
&\times \frac{\partial}{\partial t}\big[e^{-it|\eta_1|^2} \hat{\bar{u}}^h_{j'}(t,\eta_1) e^{-it|\eta_2|^2}\hat{\bar{u}}^h_{j'_2}(t,\eta_2) e^{it|\eta_3|^2}\hat{u}^l_{j'_1}(t,\eta_3) e^{it|\eta_4|^2}\hat{u}^l_{j'_3}(t,\eta_4)  d\eta_1 d\eta_2 d\eta_3 d\eta_4\big]dxdydt.
\end{align}
\begin{remark}
To simplify notation, $L^p_t L^q_xh^0$ will always indicate $L^p_t L^q_xh^0(G^{l_2}_{\beta}\times\mathbb{R}^2)$ unless otherwise
advised.
\end{remark}
We first consider \eqref{eq-y5.123}. Denote $m(t;\eta_1,\eta_2,\eta_3,\eta_4):=\frac{1}{q(\eta)}\phi(\frac{\eta_1+\eta_2+\eta_3+\eta_4}{2^{l_2-10}}) \big(1-\phi(\frac{\eta_1-\xi(t)}{2^{l_2}})\phi(\frac{\eta_2-\xi(t)}{2^{l_2}})\big)$, we can calculate
\begin{equation*}
\begin{split}
&\eqref{eq-y5.123}=-i\iiint\iiint  Im[\bar{v}_j(\nabla-i\xi(t))v_j](t,x) \frac{(x-y)}{|x-y|} m(t;\eta_1,\eta_2,\eta_3,\eta_4) e^{iy\cdot(\eta_1+\eta_2+\eta_3+\eta_4)}\\
&\times \big[ \hat{\bar{u}}^h_{j'}(t,\eta_1) \hat{\bar{u}}^h_{j'_2}(t,\eta_2) \hat{u}^l_{j'_1}(t,\eta_3) \hat{u}^l_{j'_3}(t,\eta_4)  d\eta_1 d\eta_2 d\eta_3 d\eta_4\big]dxdy|^b_a\\
=&-i\iiint\iiint  Im[\bar{v}_j(\nabla-i\xi(t))v_j](t,x) \frac{(x-y)}{|x-y|} m(t;\eta_1,\eta_2,\eta_3,\eta_4) e^{iy\cdot(\eta_1+\eta_2+\eta_3+\eta_4)}\\
&\times \big[ \iiiint e^{-iz_1\cdot\eta_1}\bar{u}^h_{j'}(t,z_1) e^{-iz_2\cdot\eta_2}\bar{u}^h_{j'_2}(t,z_2) e^{-iz_3\cdot\eta_3}u^l_{j'_1}(t,z_3) e^{-iz_4\cdot\eta_4}u^l_{j'_3}(t,z_4)dz_1 dz_2 dz_3 dz_4 \big]\\
&\qquad \qquad \qquad \qquad \qquad \qquad \qquad \qquad \qquad \qquad \qquad \qquad \qquad \qquad \qquad  d\eta_1 d\eta_2 d\eta_3 d\eta_4dxdy|^b_a\\
=&-i\iiint\iiint  Im[\bar{v}_j(\nabla-i\xi(t))v_j](t,x) \frac{(x-y)}{|x-y|} \\
&\times K(t;z_1,z_2,z_3,z_4) \bar{u}^h_{j'}(t,y-z_1) \bar{u}^h_{j'_2}(t,y-z_2) u^l_{j'_1}(t,y-z_3)u^l_{j'_3}(t,y-z_4)dz_1 dz_2 dz_3 dz_4 dxdy|^b_a \\
\end{split}
\end{equation*}
where
\begin{equation}\label{eq-y5.127}
K(t;z_1,z_2,z_3,z_4)=\iiiint  m(t;\eta_1,\eta_2,\eta_3,\eta_4)e^{iz_1\cdot\eta_1}e^{iz_2\cdot\eta_2}e^{iz_3\cdot\eta_3} e^{iz_4\cdot\eta_4} d\eta_1 d\eta_2 d\eta_3 d\eta_4,
\end{equation}
and by Theorem C-M of \cite{GM}, $K(t;z_1,z_2,z_3,z_4)$ satisfies
\begin{equation}\label{eq-y5.128}
\sup_{t}\int| K(t;z_1,z_2,z_3,z_4)| dz_1 dz_2 dz_3 dz_4 \lesssim 2^{-2l_2}.
\end{equation}
Therefore, by Bernstein inequality, \eqref{eq-y5.128} and conservation of mass,
\begin{equation*}
\begin{split}
2^{l_2-2i}|\sum_{j,j'}\sum_{\mathcal{R}(j')}\eqref{eq-y5.123}|&\lesssim 2^{l_2-2i}\|\vec{v}\|_{L^{\infty}_tL^2_xl^2} \|(\nabla-i\xi(t))\vec{v}\|_{L^{\infty}_tL^2_xl^2} \int\iiiint |K(t;z_1,z_2,z_3,z_4)|\\
\times &\|\vec{u}^h\|_{l^2}(t,y-z_1) \|\vec{u}^h\|_{l^2}(t,y-z_2) \|\vec{u}^l\|_{l^2}(t,y-z_3) \|\vec{u}^l\|_{l^2}(t,y-z_4)dz_1 dz_2 dz_3 dz_4 dy\\
&\lesssim 2^{l_2-i}\|\vec{v}_0\|^2_{L^2_xl^2}2^{-2l_2} \|\vec{u}^h\|^2_{L^{\infty}_tL^2_{x}l^2}\|\vec{u}^l\|^2_{L^{\infty}_tL^{\infty}_xl^2}\\
&\lesssim 2^{l_2-i}\|\vec{v}_0\|^2_{L^2_xl^2}\lesssim \|\vec{v}_0\|^2_{L^2_xl^2}.
\end{split}
\end{equation*}
Next we consider \eqref{eq-y5.124}.
\begin{equation}\label{eq-y5.129}
\begin{split}
&\frac{\partial}{\partial t}Im[\bar{v}_j(\partial_k-i\xi_k(t))v_j](t,x)\\
=&\xi'_k(t)|v_j|^2(t,x) + Re[\bar{v}_j(\partial_k-i\xi_k(t))\Delta v_j](t,x) - Re[\Delta\bar{v}_j(\partial_k-i\xi_k(t)) v_j](t,x)\\
=&\xi'_k(t)|v_j|^2(t,x) + \sum^2_{k'=1}\partial_{k'} Re[\bar{v}_j(\partial_k-i\xi_k(t))\partial_{k'} v_j](t,x) - \sum^2_{k'=1}\partial_{k'}Re[\partial_{k'}\bar{v}_j(\partial_k-i\xi_k(t)) v_j](t,x).
\end{split}
\end{equation}
Therefore,
\begin{align}\nonumber
&\eqref{eq-y5.124}\\ \nonumber
=&\int_{G^{l_2}_{\beta}}\iiint\iiint   \frac{(x-y)}{|x-y|}\cdot\xi'(t)|v_j|^2(t,x) \\ \label{eq-y5.130}
&\times K(t;z_1,z_2,z_3,z_4) \bar{u}^h_{j'}(t,y-z_1) \bar{u}^h_{j'_2}(t,y-z_2) u^l_{j'_1}(t,y-z_3)u^l_{j'_3}(t,y-z_4)dz_1 dz_2 dz_3 dz_4 dxdydt \\ \nonumber
+&\sum^2_{k,k'=1}\int_{G^{l_2}_{\beta}}\iiint\iiint   \frac{(x-y)_k}{|x-y|}\partial_{k'} Re[\bar{v}_j(\partial_k-i\xi_k(t))\partial_{k'} v_j](t,x) \\ \label{eq-y5.131}
&\times K(t;z_1,z_2,z_3,z_4) \bar{u}^h_{j'}(t,y-z_1) \bar{u}^h_{j'_2}(t,y-z_2) u^l_{j'_1}(t,y-z_3)u^l_{j'_3}(t,y-z_4)dz_1 dz_2 dz_3 dz_4 dxdydt \\ \nonumber
-&\sum^2_{k,k'=1}\int_{G^{l_2}_{\beta}}\iiint\iiint   \frac{(x-y)_k}{|x-y|}\partial_{k'}Re[\partial_{k'}\bar{v}_j(\partial_k-i\xi_k(t)) v_j](t,x) \\ \label{eq-y5.132}
&\times K(t;z_1,z_2,z_3,z_4) \bar{u}^h_{j'}(t,y-z_1) \bar{u}^h_{j'_2}(t,y-z_2) u^l_{j'_1}(t,y-z_3)u^l_{j'_3}(t,y-z_4)dz_1 dz_2 dz_3 dz_4 dxdydt
\end{align}
with $K(t;z_1,z_2,z_3,z_4)$ given in \eqref{eq-y5.127}.\\
By \eqref{eq-y5.128}, \eqref{eq-y5.1}, \eqref{eq-y5.34},  Bernstein inequality and conservation of mass,
\begin{equation}\label{eq-y5.133}
\begin{split}
2^{l_2-2i}|\sum_{j,j'}\sum_{\mathcal{R}(j')}\eqref{eq-y5.130}|\lesssim& 2^{l_2-2i} 2^{-2l_2} \|\vec{v}\|^2_{L^{\infty}_tL^2_xl^2} \|\vec{u}^h\|^2_{L^{\infty}_tL^2_{x}l^2}\|\vec{u}^l\|^2_{L^{\infty}_tL^{\infty}_xl^2}(\int_{G^{l_2}_{\beta}}|\xi'(t)|dt)\\
\lesssim& \ep_3\ep_1^{-1}2^{2l_2-2i}\|\vec{v}_0\|^2_{L^2_xl^2}\lesssim \|\vec{v}_0\|^2_{L^2_xl^2}.
\end{split}
\end{equation}
Integrating \eqref{eq-y5.131} and \eqref{eq-y5.132} by parts in space, we have
\begin{equation*}
\begin{split}
\eqref{eq-y5.131}=-&\sum^2_{k,k'=1}\int_{G^{l_2}_{\beta}}\iiint\iiint [\frac{\delta_{kk'}}{|x-y|} +\frac{(x-y)_{k'}(x-y)_k}{|x-y|^3}] Re[\bar{v}_j(\partial_k-i\xi_k(t))\partial_{k'} v_j](t,x) \\
\times& K(t;z_1,z_2,z_3,z_4) \bar{u}^h_{j'}(t,y-z_1) \bar{u}^h_{j'_2}(t,y-z_2) u^l_{j'_1}(t,y-z_3)u^l_{j'_3}(t,y-z_4)dz_1 dz_2 dz_3 dz_4 dxdydt,
\end{split}
\end{equation*}
so by the Hardy-Littlewood-Sobolev inequality, \eqref{eq-y5.128}, Lemma \ref{le-y5.7}, the Sobolev embedding theorem, the fact that $G^{l_2}_{\beta}\subset G^i_{\alpha}$, $|\xi(t)|\ll 2^{l_2}$ and $l_2\leq i$,
\begin{equation}\label{eq-y5.134}
\begin{split}
2^{l_2-2i}|\sum_{j,j'}\sum_{\mathcal{R}(j')}\eqref{eq-y5.131}|\lesssim& 2^{l_2-2i} 2^{2i}2^{-2l_2} \|\vec{v}\|^2_{L^{6}_tL^3_xl^2} \|\vec{u}^h\|^2_{L^{3}_tL^6_{x}l^2}\|\vec{u}^l\|^2_{L^{\infty}_tL^{4}_xl^2}\\
\lesssim& \|\vec{v}_0\|^2_{L^2_xl^2} \|\vec{u}\|^2_{\tilde{X}_i(G^i_{\alpha})}.
\end{split}
\end{equation}
A similar computation for \eqref{eq-y5.132} implies
\begin{equation}\label{eq-y5.135}
2^{l_2-2i}|\sum_{j,j'}\sum_{\mathcal{R}(j')}\eqref{eq-y5.132}|
\lesssim \|\vec{v}_0\|^2_{L^2_xl^2} \|\vec{u}\|^2_{\tilde{X}_i(G^i_{\alpha})}.
\end{equation}
Next we estimate \eqref{eq-y5.125}.
\begin{equation*}
\frac{\partial}{\partial t}\phi(\frac{\eta_1-\xi(t)}{2^{l_2}})=2^{-l_2}(\nabla\phi)(\frac{\eta_1-\xi(t)}{2^{l_2}})\xi'(t),
\end{equation*}
Similarly to \eqref{eq-y5.123}, we can calculate the corresponding kernel function $\tilde{K}(t;z_1,z_2,z_3,z_4)$,
\begin{equation}\label{eq-y5.136}
\tilde{K}(t;z_1,z_2,z_3,z_4)=\iiiint  \tilde{m}(t;\eta_1,\eta_2,\eta_3,\eta_4)e^{iz_1\cdot\eta_1}e^{iz_2\cdot\eta_2}e^{iz_3\cdot\eta_3} e^{iz_4\cdot\eta_4} d\eta_1 d\eta_2 d\eta_3 d\eta_4,
\end{equation}
where $$\tilde{m}(t;\eta_1,\eta_2,\eta_3,\eta_4)=-\frac{2^{-l_2}}{q(\eta)}\phi(\frac{\eta_1+\eta_2+\eta_3+\eta_4}{2^{l_2-10}}) \big((\nabla\phi)(\frac{\eta_1-\xi(t)}{2^{l_2}})\phi(\frac{\eta_2-\xi(t)}{2^{l_2}})+ \phi(\frac{\eta_1-\xi(t)}{2^{l_2}})(\nabla\phi)(\frac{\eta_2-\xi(t)}{2^{l_2}})\big).$$
It's easy to verify that $\tilde{K}(t;z_1,z_2,z_3,z_4)$ satisfies
\begin{equation}\label{eq-y5.137}
\sup_{t}\int| \tilde{K}(t;z_1,z_2,z_3,z_4)| dz_1 dz_2 dz_3 dz_4 \lesssim 2^{-3l_2},
\end{equation}
Hence,
\begin{equation}\label{eq-y5.138}
\begin{split}
2^{l_2-2i}|\sum_{j,j'}\sum_{\mathcal{R}(j')}\eqref{eq-y5.125}|\lesssim& 2^{l_2-2i} 2^{-3l_2} 2^i\|\vec{v}\|^2_{L^{\infty}_tL^2_xl^2} \|\vec{u}^h\|^2_{L^{\infty}_tL^2_{x}l^2}\|\vec{u}^l\|^2_{L^{\infty}_tL^{\infty}_xl^2}(\int_{G^{l_2}_{\beta}}|\xi'(t)|dt)\\
\lesssim& \ep_3\ep_1^{-1}2^{l_2-i}\|\vec{v}_0\|^2_{L^2_xl^2}\lesssim \|\vec{v}_0\|^2_{L^2_xl^2}.
\end{split}
\end{equation}
Finally we turn to \eqref{eq-y5.126}. Let's recall $\mathbf{F} (\vec{u}(t)): = \{ F_j(\vec{u}(t)) \}_{j\in\mathbb{Z}}: =\{\sum\limits_{\mathcal{R}(j)}u_{j_1}\bar{u}_{j_2}u_{j_3}\}_{j\in\mathbb{Z}} =:\sum^{\rightarrow}\limits_{\mathcal{R}(j)}u_{j_1}\bar{u}_{j_2}u_{j_3}$.\\
If $m(t,\xi)$ is a Fourier multiplier for every fixed $t$, then
\begin{equation}\label{1476}
\begin{split}
  \frac{\partial}{\partial t}(m(t,\xi)e^{it|\xi|^2} \hat{u}_j(t,\xi))= & -im(t,\xi)e^{it|\xi|^2}(\sum_{\mathcal{R}(j)} u_{j_1} \bar{u}_{j_2} u_{j_3})^{\wedge}(t,\xi)+ (e^{it|\xi|^2} \hat{u}_j(t,\xi))\frac{\partial}{\partial t}m(t,\xi), \\
   \frac{\partial}{\partial t}(m(t,\xi)e^{-it|\xi|^2} \hat{\bar{u}}_j(t,\xi))= & im(t,\xi)e^{-it|\xi|^2}(\sum_{\mathcal{R}(j)}\bar{u}_{j_1} u_{j_2} \bar{u}_{j_3})^{\wedge}(t,\xi)+ (e^{-it|\xi|^2} \hat{\bar{u}}_j(t,\xi))\frac{\partial}{\partial t}m(t,\xi).
\end{split}
\end{equation}
Take $m(t,\xi)=\phi(\frac{\xi-\xi(t)}{2^{l_2-5}})$ and $m(t,\xi)=1-\phi(\frac{\xi-\xi(t)}{2^{l_2-5}})$ respectively, then we have $\frac{\partial}{\partial t}\phi(\frac{\eta_1-\xi(t)}{2^{l_2}})=2^{-l_2}(\nabla\phi)(\frac{\eta_1-\xi(t)}{2^{l_2}})\xi'(t)$ and $\frac{\partial}{\partial t}(1-\phi(\frac{\eta_1-\xi(t)}{2^{l_2}}))=-2^{-l_2}(\nabla\phi)(\frac{\eta_1-\xi(t)}{2^{l_2}})\xi'(t)$.
Following the previous argument on \eqref{eq-y5.123} and \eqref{eq-y5.125}, we obtain by Bernstein inequality, conservation of mass and Lemma \ref{le-y5.7} that
\begin{equation}\label{eq-y5.139}
\begin{split}
2^{l_2-2i}|\sum_{j,j'}\sum_{\mathcal{R}(j')}\eqref{eq-y5.126}|\lesssim& 2^{l_2-2i} 2^{-3l_2} 2^i\|\vec{v}\|^2_{L^{\infty}_tL^2_xl^2} \|\vec{u}^h\|^2_{L^{\infty}_tL^2_{x}l^2}\|\vec{u}^l\|^2_{L^{\infty}_tL^{\infty}_xl^2}(\int_{G^{l_2}_{\beta}}|\xi'(t)|dt)\\
+& 2^{l_2-2i} 2^i 2^{-2l_2}\|\vec{v}\|^2_{L^{\infty}_tL^2_xl^2} \|\vec{u}^h\|^2_{L^{4}_tL^4_{x}l^2} \|\vec{u}^l\|_{L^{\infty}_tL^{\infty}_xl^2} \|P_{\xi(t),\leq l_2-5}\mathbf{F} (\vec{u}^h(t))\|_{L^2_{t,x}l^2}\\
+& 2^{l_2-2i} 2^i 2^{-2l_2}\|\vec{v}\|^2_{L^{\infty}_tL^2_xl^2} \|\vec{u}^h\|^2_{L^{6}_tL^3_{x}l^2} \|\vec{u}^l\|_{L^{6}_tL^{12}_xl^2} \|P_{\xi(t),\leq l_2-5}\mathbf{F} (\vec{u}^l(t))\|_{L^2_{t}L^4_x l^2}\\
+& 2^{l_2-2i} 2^i 2^{-2l_2}\|\vec{v}\|^2_{L^{\infty}_tL^2_xl^2} \|\vec{u}^h\|^2_{L^{3}_tL^6_{x}l^2} \|\vec{u}^l\|^4_{L^{6}_tL^{12}_xl^2} \\
+& 2^{l_2-2i} 2^i 2^{-2l_2}\|\vec{v}\|^2_{L^{\infty}_tL^2_xl^2} \|\vec{u}^h\|^4_{L^{4}_tL^4_{x}l^2} \|\vec{u}^l\|^2_{L^{\infty}_tL^{\infty}_xl^2}\\
\lesssim& 2^{l_2-i}\|\vec{v}_0\|^2_{L^2_xl^2}(1+\|\vec{u}\|^4_{\tilde{X}_i(G^i_{\alpha})}).
\end{split}
\end{equation}
This finally completes the proof of Lemma \ref{le-y5.18}.

\end{proof}
According to Lemma \ref{le-y5.18}, we see that the proof of Theorem \ref{th-y5.17} is complete.
\end{proof}

\section{Frequency localized interaction Morawetz estimate and the completion of Theorem \ref{th-y8.1'}}
In this section we will prove frequency localized interaction Morawetz estimate, which are used to complete the proof of Theorem \ref{th-y8.1'}.
\begin{theorem}[frequency localized interaction Morawetz estimate]\label{th-y9.1}
   Suppose $\vec{u}(t,x)$ is a minimal mass blowup solution to \eqref{eq-y1} on $[0,T]$ with
$\int_0^T N(t)^3 dt=K$.  Then
\begin{equation}\label{eq-y9.1}
 \|\sum_{j\in\mathbb{Z}}|\nabla|^{1/2}|P_{\leq \frac{10K}{\ep_1}}u_j(t,x)|^2\|^2_{L^2_{t,x}([0,T]\times \mathbb{R}^2)} \lesssim o(K),
\end{equation}
where $o(K)$ is a quantity such that $\frac{o(K)}{K}\rightarrow0 $ as $K\rightarrow \infty$.
\end{theorem}
\begin{proof}
Suppose $[0,T]$ is an interval such that, for some integer $k_0$,
$\|\vec{u}\|^4_{L^4_{t,x}l^2([0,T])}=2^{k_0}$. Rescale with $\lambda=\frac{\ep_3 2^{k_0}}{K}$, then by Theorem \ref{th-y6.2}, we have
\begin{equation}\label{eq-y6.1}
\|\vec{u}_{\lambda}\|_{\tilde{X}_{k_0}([0,\frac{T}{\lambda^2}]\times \mathbb{R}^2)}\lesssim 1.
\end{equation}
Let $\vec{w}=P_{\leq k_0}\vec{u}$, then $\vec{w}=\{w_j\}_{j\in\mathbb{Z}}$ satisfies the following infinite dimensional vector-valued equation
\begin{equation*}
  i\partial_t w_{j}+\triangle w_{j}=\sum_{\mathcal{R}(j)}w_{j_1}\bar{w}_{j_2}w_{j_3}+N_{j},
\end{equation*}
where $N_{j}=P_{\leq k_0}\big(\sum_{\mathcal{R}(j)}u_{j_1}\bar{u}_{j_2}u_{j_3}\big)- \sum_{\mathcal{R}(j)}w_{j_1}\bar{w}_{j_2}w_{j_3}$ and we denote $\vec{N}=\{N_{j}\}_{j\in \mathbb{Z}}$.

Let
\begin{equation}\label{eq-y5.52'''}
\begin{split}
M(t)=\sum_{j,j'\in \mathbb{Z}} \big(\int_{\mathbb{R}^2} \int_{\mathbb{R}^2} |w_{j'}(t,y)|^2 \frac{(x-y)}{|x-y|}\cdot Im[\bar{w}_j\nabla w_j](t,x)dxdy \big).
\end{split}
\end{equation}
Following the calculation of \cite{FL} and \cite{D}, we can show
\begin{equation}\label{eq-y9.2}
 \|\sum_{j\in\mathbb{Z}}|\nabla|^{1/2}|w_j(t,x)|^2\|^2_{L^2_{t,x}([0,\frac{T}{\lambda^2}]\times \mathbb{R}^2)} \lesssim \sup_{[0,\frac{T}{\lambda^2}]} |M(t)| + E,
\end{equation}
where $E$ is a Galilean invariant quantity. After Galilean transformation,
\begin{align}\label{eq-y9.3}
  E= & 2\big| \sum_{j,j'\in \mathbb{Z}}  \int^{\frac{T}{\lambda^2}}_{0} \int_{\mathbb{R}^2} \int_{\mathbb{R}^2} Im[\bar{w}_{j} (\nabla- i\xi(t))w_{j}](t,x) \frac{(x-y)}{|x-y|}\cdot Im[\bar{w}_{j'}N_{j'}](t,y)dxdydt \big| \\ \label{eq-y5.54''}
  + & \big| \sum_{j,j'\in \mathbb{Z}}  \int^{\frac{T}{\lambda^2}}_{0} \int_{\mathbb{R}^2} \int_{\mathbb{R}^2} |w_j(t,y)|^2 \frac{(x-y)}{|x-y|}\cdot Im[\bar{N}_{j'} (\nabla- i\xi(t))w_{j'}](t,x)dxdydt \big|\\ \label{eq-y5.55''}
  + & \big| \sum_{j,j'\in \mathbb{Z}}  \int^{\frac{T}{\lambda^2}}_{0} \int_{\mathbb{R}^2} \int_{\mathbb{R}^2} |w_j(t,y)|^2 \frac{(x-y)}{|x-y|}\cdot Im[\bar{w}_{j'} (\nabla- i\xi(t))N_{j'}](t,x)dxdydt \big| .
\end{align}
Since $N(t)\leq 1$, $N_{\lambda}(t)\leq \frac{\ep_3 2^{k_0}}{K}$, Therefore, by Corollary \ref{co-y4.9} and Bernstein's inequality, for any $\eta>0$, if $K(\eta)\geq R(\eta)$, then
\begin{equation}\label{eq-y9.3'}
\|(\nabla-i\xi(t))\vec{w}\|_{L^{\infty}_t L^2_x l^2([0,\frac{T}{\lambda^2}]\times \mathbb{R}^2)}\lesssim \eta 2^{k_0}.
\end{equation}
Hence, by Galilean transformation and conservation of mass, $\sup_{[0,\frac{T}{\lambda^2}]} |M(t)|\lesssim \eta 2^{k_0}$.

Now we estimate \eqref{eq-y9.3}. As in \eqref{eq-y5.85}, let $u_{j'}^l = P_{\leq k_0 -3} u_{j'}$, $u_{j'} = u_{j'}^h + u_{j'}^l$, and we decompose
\begin{equation}\label{eq-y9.4}
\begin{split}
\sum_{j'} Im\big[ \bar{w}_{j'} N_{j'}\big]
=\sum_{j'}\big[ F_{0,j'}(t,y)+F_{1,j'}(t,y)+F_{2,j'}(t,y)+F_{3,j'}(t,y)+F_{4,j'}(t,y)\big].
\end{split}
\end{equation}
At first, $\sum_{j'} F_{0,j'}(t,y)=0$. Then by the proof of Lemma \ref{le-y5.18},
\begin{equation*}
  \sum_{j'}\|F_{2,j'}(t,y)+F_{3,j'}(t,y)+F_{4,j'}\|_{L^1_{t,x}l^2 ([0,\frac{T}{\lambda^2}]\times \mathbb{R}^2)}\lesssim 1.
\end{equation*}
So by \eqref{eq-y9.3'} and conservation of mass,
\begin{equation}\label{eq-y9.4'}
\begin{split}
&\big| \sum_{j,j'\in \mathbb{Z}}  \int^{\frac{T}{\lambda^2}}_{0} \int_{\mathbb{R}^2} \int_{\mathbb{R}^2} Im[\bar{w}_{j} (\nabla- i\xi(t))w_{j}](t,x) \frac{(x-y)}{|x-y|}\cdot [F_{2,j'}(t,y)+F_{3,j'}(t,y)+F_{4,j'}](t,y)dxdydt \big|\\
\lesssim& \eta 2^{k_0}.
\end{split}
\end{equation}
As in \eqref{eq-y5.86'}, the Fourier support of $\sum_{j'}F_{1,j'}(t,y)$ is on $|\xi|\geq 2^{k_0-4}$, so integrating by parts about the spatial variable $y$, then using Hardy-Littlewood-Sobolev inequality,  Bernstein inequality, Lemma \ref{le-y5.7} and \eqref{eq-y6.1}, we have
\begin{equation}\label{eq-y9.4''}
\begin{split}
& \big| \sum_{j\in \mathbb{Z}} \int_{0}^{\frac{T}{\lambda^2}} \int_{\mathbb{R}^2} \int_{\mathbb{R}^2} Im[\bar{w}_{j} (\nabla- i\xi(t))w_{j}](t,x) \frac{(x-y)}{|x-y|}\cdot \big[\frac{\triangle_y}{\triangle_y}(\sum_{j'}F_{1,j'})\big](t,y)dxdydt \big|\\
\lesssim& \int_{0}^{\frac{T}{\lambda^2}} \int_{\mathbb{R}^2} \int_{\mathbb{R}^2} \big|\sum_{j\in \mathbb{Z}} [\bar{w}_{j} (\nabla- i\xi(t))w_{j}]\big|(t,x) \frac{1}{|x-y|}\cdot \big|\frac{\partial_y}{\triangle_y}(\sum_{j'}F_{1,j'})\big|(t,y)dxdydt\\
\lesssim& \|\vec{w}\|_{L^{\infty}_tL^{2}_xh^0([0,\frac{T}{\lambda^2}]\times\mathbb{R}^2)} \|(\nabla- i\xi(t))\vec{w}\|_{L^{4}_{t,x}l^2([0,\frac{T}{\lambda^2}]\times\mathbb{R}^2)} \big\|\frac{\partial_y}{\triangle_y}(\sum_{j'}F_{1,j'})\big\|_{L^{4/3}_tL^{4/3}_x([0,\frac{T}{\lambda^2}]\times\mathbb{R}^2)}\\
\lesssim&2^{-k_0}\|(\nabla- i\xi(t))\vec{w}\|_{L^{4}_{t,x}l^2([0,\frac{T}{\lambda^2}]\times\mathbb{R}^2)} \|\vec{u}^l\|^3_{L^6_{t}L^{6}_{x}l^2([0,\frac{T}{\lambda^2}]\times\mathbb{R}^2)} \|\vec{u}^h\|_{L^{4}_{t}L^4_{x}l^2([0,\frac{T}{\lambda^2}]\times\mathbb{R}^2)}\\
\lesssim&2^{-k_0}\|(\nabla- i\xi(t))\vec{w}\|_{L^{4}_{t,x}l^2([0,\frac{T}{\lambda^2}]\times\mathbb{R}^2)} \|\vec{u}^l\|^3_{L^6_{t}L^{6}_{x}l^2([0,\frac{T}{\lambda^2}]\times\mathbb{R}^2)}.
\end{split}
\end{equation}
Again by Bernstein's inequality, Lemma \ref{le-y5.7} and \eqref{eq-y6.1},
\begin{equation*}
\|\vec{u}^l\|_{L^6_{t}L^{6}_{x}l^2([0,\frac{T}{\lambda^2}]\times\mathbb{R}^2)}\lesssim \sum_{0\leq l\leq k_0}2^{\frac{l}{3}} 2^{\frac{(k_0-l)}{6}}\lesssim 2^{k_0/3}.
\end{equation*}
Meanwhile,
\begin{equation}\label{eq-y9.5}
\|(\nabla-i\xi(t))\vec{w}\|_{L^{5/2}_tL^{10}_xh^0 ([0,\frac{T}{\lambda^2}]\times \mathbb{R}^2)}\lesssim \sum_{0\leq l\leq k_0}2^l 2^{\frac{2}{5}(k_0-l)}\lesssim 2^{k_0}.
\end{equation}
Interpolating \eqref{eq-y9.5} and \eqref{eq-y9.3'}, we have
\begin{equation}\label{eq-y9.6}
\|(\nabla-i\xi(t))\vec{w}\|_{L^4_{t,x}l^2 ([0,\frac{T}{\lambda^2}]\times \mathbb{R}^2)}\lesssim \eta^{3/8}2^{k_0}.
\end{equation}
Therefore,
\begin{equation}\label{eq-y9.6'}
|\eqref{eq-y9.3}| \lesssim \eta^{3/8}2^{k_0}+ \eta 2^{k_0}.
\end{equation}
Next we estimate \eqref{eq-y5.54''}. By \eqref{eq-y5.80'} and \eqref{eq-y6.1}, we have
\begin{align*}
|\eqref{eq-y5.54''}| \lesssim&  \|\vec{w}\|^2_{L^{\infty}_tL^2_xh^0([0,\frac{T}{\lambda^2}]\times \mathbb{R}^2)} \|\vec{N}\|_{L_{t,x}^{4/3}l^2([0,\frac{T}{\lambda^2}]\times \mathbb{R}^2)} \|(\nabla-i\xi(t))\vec{w}\|_{L^4_{t,x}l^2([0,\frac{T}{\lambda^2}]\times \mathbb{R}^2)} \\
\lesssim& \|(\nabla-i\xi(t))\vec{w}\|_{L^4_{t,x}l^2 ([0,\frac{T}{\lambda^2}]\times \mathbb{R}^2)}\\
\lesssim& \eta^{3/8}2^{k_0}.
\end{align*}
Finally we turn to \eqref{eq-y5.55''}. Integrating by parts, we have
\begin{equation*}
\eqref{eq-y5.55''}\leq |\eqref{eq-y5.54''}| +\big|\sum_{j,j'}\int_0^{\frac{T}{\lambda^2}}\int_{\mathbb{R}^2}\int_{\mathbb{R}^2} |w_j(t,y)|^2 \frac{1}{|x-y|} Re[\bar{w}_{j'}N_{j'}](t,x)dxdydt\big|.
\end{equation*}
By \eqref{eq-y2.6} and \eqref{eq-y5.85'},
\begin{align}\nonumber
&\big|\sum_{j,j'}\int_0^{\frac{T}{\lambda^2}}\int_{\mathbb{R}^2}\int_{\mathbb{R}^2} |w_j(t,y)|^2 \frac{1}{|x-y|} Re[\bar{w}_{j'}N_{j'}](t,x)dxdydt\big|\\ \label{eq-y9.7}
\lesssim& \int_0^{\frac{T}{\lambda^2}}\int_{\mathbb{R}^2}\int_{\mathbb{R}^2} \|\vec{w}\|_{l^2}^2(t,y) \frac{1}{|x-y|} \|\vec{w}\|_{l^2}(t,x) \|\vec{u}^h\|_{l^2}^3(t,x)dxdydt\\ \label{eq-y9.8}
&+ \int_0^{\frac{T}{\lambda^2}}\int_{\mathbb{R}^2}\int_{\mathbb{R}^2} \|\vec{w}\|_{l^2}^2(t,y) \frac{1}{|x-y|} \|\vec{w}\|_{l^2}(t,x) \|\vec{u}^l\|_{l^2}^2(t,x)\|\vec{u}^h\|_{l^2}(t,x)dxdydt.
\end{align}
By Hardy-Littlewood-Sobolev inequality, \eqref{eq-y9.3'}, \eqref{eq-y9.6}, Lemma \ref{le-y5.7}, the Sobolev embedding theorem, conservation of mass, and interpolation, we get
\begin{equation}\label{eq-y9.9}
\eqref{eq-y9.7}\lesssim \|\vec{u}^h\|^3_{L^4_{t,x}l^2([0,\frac{T}{\lambda^2}]\times \mathbb{R}^2)} \|\vec{w}\|^3_{L^4_{t,x}l^2([0,\frac{T}{\lambda^2}]\times \mathbb{R}^2)}\lesssim \eta^{3/8}2^{k_0}
\end{equation}
and
\begin{equation}\label{eq-y9.12}
\eqref{eq-y9.8}\lesssim \|\vec{u}^h\|_{L^3_tL^6_xl^2([0,\frac{T}{\lambda^2}]\times \mathbb{R}^2)} \|\vec{w}\|^3_{L^9_{t}L^{90/29}_xl^2([0,\frac{T}{\lambda^2}]\times \mathbb{R}^2)} \|\vec{u}^l\|^2_{L^6_{t}L^{60/11}_xl^2([0,\frac{T}{\lambda^2}]\times \mathbb{R}^2)}\lesssim \eta^{1/6}2^{k_0}.
\end{equation}
Therefore, by \eqref{eq-y9.2}
\begin{equation}\label{eq-y9.13}
 \|\sum_{j\in\mathbb{Z}}|\nabla|^{1/2}|w_j(t,x)|^2\|^2_{L^2_{t,x}([0,\frac{T}{\lambda^2}]\times \mathbb{R}^2)} \lesssim \eta^{1/6}2^{k_0},
\end{equation}
Undoing the scaling $u_j(t,x)\mapsto \lambda u_j(\lambda^2 t, \lambda x)$,  $\lambda=\frac{\ep_3 2^{k_0}}{K}$, we have
\begin{equation}\label{eq-y9.14}
 \|\sum_{j\in\mathbb{Z}}|\nabla|^{1/2}|P_{\leq \ep_1^{-1}K}u_j(t,x)|^2\|^2_{L^2_{t,x}([0,\frac{T}{\lambda^2}]\times \mathbb{R}^2)} \lesssim \ep_3^{-1}\eta(K)^{1/6}K.
\end{equation}
This proves Theorem \ref{th-y9.1}.
\end{proof}
Now we are ready to complete the proof of Theorem \ref{th-y8.1'}.
\begin{theorem}\label{th-y9.2}
   There does not exist a minimal mass blowup solution to \eqref{eq-y1}.
\end{theorem}
\begin{proof}It  suffices to exclude two scenarios separately.\\
\textbf{Case 1.} Rapid frequency cascade: $\int_0^{\infty} N(t)^3 dt <\infty$. In this case, we can repeat the process as section 5 of \cite{D} and follow the arguments in Section 6 in that paper to obtain an additional regularity of a minimal mass blowup solution to \eqref{eq-y1}, that is, $\|\vec{u}(t,x)\|_{L^{\infty}_t \dot{H}^{3}_xh^0([0,\infty)\times \mathbb{R}^2)} \lesssim_{m_0} (\int^{\infty}_0 N(t)^3 dt)^3$, which together with the definition of almost periodic solution yields
$$\|e^{-ix\cdot\xi(t)}\vec{u}\|_{\dot{H}^{1}_xh^0}\lesssim N(t)C(\eta(t))+\eta(t)^{1/2},\quad \eta(t)\rightarrow 0.$$
Since $\lim_{t\rightarrow\infty}N(t)=0$, this implies $\lim_{t\rightarrow\infty}\|e^{-ix\cdot\xi(t)}\vec{u}\|_{\dot{H}^{1}_xh^0}=0$. So for any $\ep>0$, there exists $t_0>0$ (by Galilean transformation we may take $t_0=0$), such that $\|e^{-ix\cdot\xi(t_0)}\vec{u}(t_0)\|_{\dot{H}^{1}_xh^0}<\ep$. \\
Notice that
$$E(\vec{u}(t))=\frac{1}{2}\int_{\mathbb{R}^{2}}\sum\limits_{j\in\mathbb{Z}}|\nabla u_{j}(t,x)|^{2}\mathrm{d}x+\frac{1}{4}\int_{\mathbb{R}^{2}}\sum\limits_{\substack{j_0,j_1,j_2,j_3\in \mathbb{Z},\\ j_1-j_2+j_3 = j_0,\\ |j_1|^2-|j_2|^2 +|j_3|^2 = |j_0|^2.}} \bar{u}_{j_0}u_{j_1} \bar{u}_{j_2} u_{j_3}\mathrm{d}x=E(\vec{u}(0)),$$
by Minkowski inequality and sharp Gagliardo-Nirenberg inequality, we can calculate
\begin{equation*}
\begin{split}
\int_{\mathbb{R}^{2}}\sum\limits_{\substack{j_0,j_1,j_2,j_3\in \mathbb{Z},\\ j_1-j_2+j_3 = j_0,\\ |j_1|^2-|j_2|^2 +|j_3|^2 = |j_0|^2.}} \bar{u}_{j_0}u_{j_1} \bar{u}_{j_2} u_{j_3}\mathrm{d}x \sim& \int_{\mathbb{R}^{2}}\big( \sum_{j}|u_j|^2\big)^2\mathrm{d}x\lesssim  \big(\sum_{j} \|u_j\|^2_{L^4_{x}(\mathbb{R}^{2})}\big)^2\\
\lesssim & \big(\sum_{j} \|u_j\|_{L^2_{x}(\mathbb{R}^{2})} \|\nabla u_j\|_{L^2_{x}(\mathbb{R}^{2})}\big)^2\\
\lesssim & (\sum_{j} \|u_j\|^2_{L^2_{x}(\mathbb{R}^{2})}) (\sum_{j}\|\nabla u_j\|^2_{L^2_{x}(\mathbb{R}^{2})}).
\end{split}
\end{equation*}Therefore, $E(\vec{u}(t))=E(\vec{u}(0))\lesssim \|\vec{u}(0)\|^2_{\dot{H}^{1}_xh^0}< \ep^2$.\\
However, by H\"{o}lder inequality,
\begin{align*}
  \sum_{j\in\mathbb{Z}}\int|u_j(0,x)|^2dx \leq& \sum_{j\in\mathbb{Z}}\int_{|x-x(0)|\leq \frac{C\big(\frac{\sum_j\|u_j(0)\|^2_{L^2}}{1000}\big)}{N(0)}}|u_j(0,x)|^2dx +\frac{\sum_j\|u_j(0)\|^2_{L^2}}{1000}\\
   \leq& C[\int_{\mathbb{R}^{2}}\big( \sum_{j}|u_j(0,x)|^2\big)^2dx]^{1/2} \frac{C\big(\frac{\sum_j\|u_j(0)\|^2_{L^2}}{1000}\big)}{N(0)} +\frac{\sum_j\|u_j(0)\|^2_{L^2}}{1000} \\
   \leq& CE(\vec{u}(0))^{1/2}\frac{C\big(\frac{\sum_j\|u_j(0)\|^2_{L^2}}{1000}\big)}{N(0)} +\frac{\sum_j\|u_j(0)\|^2_{L^2}}{1000}.
\end{align*}
Choose $\ep$ sufficiently small such that
$$C \ep \frac{C\big(\frac{\sum_j\|u_j(0)\|^2_{L^2}}{1000}\big)}{N(0)}<\frac{\sum_j\|u_j(0)\|^2_{L^2}}{100}.$$
This implies $\sum_{j\in\mathbb{Z}}\int|u_j(0,x)|^2dx<\frac{\sum_{j\in\mathbb{Z}}\int|u_j(0,x)|^2dx}{100}$, which can't happen unless $\sum_{j\in\mathbb{Z}}\int|u_j(0,x)|^2dx=0$, so $\int|u_j(0,x)|^2dx=0$ for every $j\in\mathbb{Z}$, this infers $\|\vec{u}(0)\|_{L^2_xh^1}=0$. This excludes rapid frequency cascade scenario.

\noindent\textbf{Case 2.} Quasi-soliton: $\int_0^{\infty} N(t)^3 dt =\infty$. In this case, we denote $Iu_j=P_{\leq10\ep_1^{-1}K}u_j$. By frequency localized interaction Morawetz estimate,
$$\|\sum_{j\in\mathbb{Z}}|\nabla|^{1/2}|Iu_j(t,x)|^2\|^2_{L^2_{t,x}([0,T]\times \mathbb{R}^2)} \lesssim o(K),$$
where recalling $\int^T_0 N(t)^3 dt=K$.
By H\"{o}lder inequality and the Sobolev embedding, we have
\begin{align*}
&\sum_{j\in\mathbb{Z}}\int_{|x-x(t)|\leq \frac{C\big(\frac{\sum_j\|u_j\|^2_{L^2}}{1000}\big)}{N(t)}}|Iu_j(t,x)|^2dx \\
\lesssim& \big(\frac{C\big(\frac{\sum_j\|u_j\|^2_{L^2}}{1000}\big)}{N(t)}\big)^{3/2} \|\sum_j|Iu_j(t)|^2\|_{L^4_x}\\
\lesssim&\big(\frac{C\big(\frac{\sum_j\|u_j\|^2_{L^2}}{1000}\big)}{N(t)}\big)^{3/2} \|\sum_{j\in\mathbb{Z}}|\nabla|^{1/2}|Iu_j(t,x)|^2\|_{L^2_{x}}.
\end{align*}
Now for $K>C\big(\frac{\sum_j\|u_j\|^2_{L^2}}{1000}\big)$, by Proposition \ref{pr-y4.8}, we have
$$\frac{\sum_j\|u_j\|^2_{L^2}}{2}<\sum_{j\in\mathbb{Z}}\int_{|x-x(t)|\leq \frac{C\big(\frac{\sum_j\|u_j\|^2_{L^2}}{1000}\big)}{N(t)}}|Iu_j(t,x)|^2dx.$$
Therefore,
\begin{align*}
&\big(\sum_j\|u_j\|^2_{L^2}\big)^2 K\sim \big(\sum_j\|u_j\|^2_{L^2}\big)^2\int_0^T N(t)^3dt\\
\lesssim &\int_0^T N(t)^3 \left(\sum_{j\in\mathbb{Z}}\int_{|x-x(t)|\leq \frac{C\big(\frac{\sum_j\|u_j\|^2_{L^2}}{1000}\big)}{N(t)}}|Iu_j(t,x)|^2dx\right)^2 dt\\ \lesssim&\|\sum_{j\in\mathbb{Z}}|\nabla|^{1/2}|Iu_j(t,x)|^2\|^2_{L^2_{t,x}([0,T]\times \mathbb{R}^2)}\lesssim o(K).
\end{align*}
Combined with the mass conservation, this gives a contradiction for $K$ sufficiently large. Therefore, the proof of theorem
\ref{th-y9.2} is complete.
\end{proof}

\section*{Acknowledgments}
The first author thanks Dr. Ze Li and Dr. Xing Cheng for helpful discussions.


\begin{thebibliography}{999}
\bibitem{APT} M. J. Ablowitz, B. Prinari,  A. D. Trubatch. \emph{Discrete and continuous nonlinear Schr\"{o}dinger
systems}. London Mathematical Society Lecture Note Series, 302. Cambridge University Press, Cambridge, 2004.

\bibitem{B} J. Bourgain. \emph{Refinements of Strichartz' inequality and applications to 2D-NLS with critical nonlinearity}. International Mathematical Research Notices (1998), no. 5, 253-283.

\bibitem{VRT} V. Banica,  R. Carles, T. Duyckaerts. \emph{On scattering for NLS: from Euclidean to hyperbolic space}. DCDS.{\bf 24} (2009), no. 4, 1113-1127.

\bibitem{BL} H. Berestycki and P.-L. Lions. \emph{Existence d'ondes solitaires dans des probl\`{e}mes nonlin\'{e}aires
du type Klein-Gordon}. C. R. Acad. Sci. Paris S\'{e}r. A-B {\bf 288} (1979), no. 7, A395-A398.


\bibitem{BC} H. Brezis and J. M. Coron. \emph{Convergence of solutions of H-systems or how to blow bubbles},
 Arch. Rational Mech. Anal. {\bf 89} (1985), no. 1, 21-56.

\bibitem{CW} T. Cazenave and F. B. Weissler. \emph{The Cauchy problem for the critical nonlinear Schr\"odinger equation in $H^{s}$}. Nonlinear Anal. {\bf 14} (1990), no. 10, 807-836.


\bibitem{CHENG} X. Cheng, Z. Guo, K. Yang and L. Zhao. \emph{On scattering for the cubic defocusing nonlinear Schr\"odinger equation on wave-guide  $\mathbb{R}^2 \times \mathbb{T}$}. prertint.

\bibitem{JH}J. H. Coleman.  \emph{Blowup phenomena for the vector nonlinear Schroedinger equation}. Thesis (Ph.D.)-University of Toronto (Canada). 2001. 114 pp. ISBN: 978-0612-63694-1

\bibitem{CKSTT} J. Colliander,  M. Keel, G. Staffilani, H. Takaoka, T. Tao.  \emph{Global well-posedness and scattering for the energy-critical nonlinear Schr\"{o}dinger equation in $\mathbb{R}^3$}. Ann. of Math. (2) {\bf 167} (2008), no. 3, 767-865.

\bibitem{D1} B. Dodson. \emph{Global well-posedness and scattering for the defocusing, $L^2$ -critical,
nonlinear Schr\"{o}dinger equation when $d=1$}.	Amer. J. Math. {\bf 138} (2016), no. 2, 531-569.

\bibitem{D} B.Dodson. \emph{Global well-posedness and scattering for the defocusing, $L^2$-critical, nonlinear Schr\"{o}dinger equation when d=2}. Duke Math. J. {\bf 165} (2016), no. 18, 3435-3516.
\bibitem{D3} B. Dodson. \emph{Global well-posedness and scattering for the defocusing, $L^2$-critical, nonlinear Schr\"odinger equation when $d\geq3$}, J. Amer. Math. Soc. {\bf 25} (2012), no. 2, 429-463.

\bibitem{D4} B. Dodson. \emph{Global well-posedness and scattering for the mass critical nonlinear Schr\"odinger equation with mass below the mass of the ground state}. Adv. Math. {\bf 285} (2015), 1589-1618.
\bibitem{D5} B. Dodson. \emph{Global well-posedness and scattering for the focusing, energy-critical nonlinear Schr\"{o}dinger problem in dimension $d=4$ for initial data below a ground state threshold}. arXiv:1409.1950.

\bibitem{D6} B. Dodson. \emph{nonlinear Schr\"odinger equations}. preprint, 2013.

\bibitem{DML}D. Duchesne, M. Ferrera, L. Razzari, R. Morandotti, B. Little, S. T. Chu, D. J. Moss.  \emph{Nonlinear optics in doped silica glass integrated waveguide structures}. Frontiers in guided wave optics and optoelectronics. Edited by Bishnu Pal. InTech, Croatia, 2010.

\bibitem{VS} V. S. Gerdjikov.  \emph{On soliton interactions of vector nonlinear Schr\"{o}dinger equations}. Application of Mathematics in Technical and Natural Sciences: 3rd International Conference-AMiTaNS'11 AIP Publishing, {\bf 1404}  (2011), no. 1, 57-67.
\bibitem{JG} J. Ginibre and G. Velo. \emph{Scattering theory in the energy space for a class of nonlinear Schr\"odinger equations}. J. Math. Pures Appl.(9) {\bf 64} (1985), no. 4,  363-401.
\bibitem{LG1} L. Grafakos.  \emph{Classical fourier analysis}[M]. New York: Springer, 2008.
\bibitem{GM} P. Germain, N. Masmoudi, and J. Shatah. \emph{Global solutions for 2D quadratic Schr\"{o}dinger equations}. Journal de Math\'{e}matiques Pures et Appliqu\'{e}es. Neuvi\`{e}me S\'{e}rie {\bf 97}(2012), no. 5, 505-543.
\bibitem{HHK} M. Hadac, S. Herr and H. Koch. \emph{Well-posedness and scattering for the KP-II equation in a critical space}. Ann. Inst. H. Poincar\'{e} Anal. Non Lin\'{e}aire, {\bf 26} (2009), no. 3, 917-941.

\bibitem{HP} Z. Hani and B. Pausader, \emph{On scattering for the quintic defocusing nonlinear Schr\"odinger equation on $\mathbb{R}\times \mathbb{T}^2$}. Comm. Pure Appl. Math. {\bf67} (2014), no. 9, 1466-1542.

%
%
\bibitem{KT} M. Keel and T. Tao. \emph{Endpoint Strichartz estimates}.  Amer. J. Math. {\bf 120} (1998), no. 5, 955-980.

\bibitem{KM} C. E. Kenig and F. Merle. \emph{Global well-posedness, scattering and blow-up for the energy-critical, focusing, non-linear Schr\"odinger equation in the radial case}. Invent. Math. {\bf 166} (2006), no. 3, 645-675.

\bibitem{KM1} C. E. Kenig and F. Merle. \emph{Global well-posedness, scattering and blow-up for the energy-critical focusing nonlinear wave equation}. Acta Math. {\bf 201} (2008), no. 2, 147-212.

\bibitem{KTV} H. Koch, D. Tataru and M. Visan. \emph{Dispersive equations and nonlinear waves}. Birkh\"{a}user, 2014.

\bibitem{KV1} R. Killip, M.Visan.  \emph{The focusing energy-critical nonlinear Schr\"{o}dinger equation in dimensions five and higher}.  Amer. J. Math. {\bf 132} (2010), no. 2, 361-424.

\bibitem{KV2} R. Killip, M. Visan, X. Zhang. \emph{The mass-critical nonlinear Schr\"odinger equation with radial data in dimensions three and higher}.  Anal. PDE {\bf 1} (2008), no. 2, 229-266.

\bibitem{KTV2}R. Killip,  T. Tao and M. Visan. \emph{The cubic nonlinear Schr\"{o}dinger equation in two dimensions with radial data}. J. Eur. Math. Soc. {\bf 11} (2009), no. 6, 1203-1258.
%


%

\bibitem{QJP}  Q. Lin, O. J. Painter, G. P. Agrawal. \emph{Nonlinear optical phenomena in silicon wave guides: modeling and applications}. Optics Express {\bf 15} (2007), no. 25, 16604-16644.
\bibitem{JDA} J. Metcalfe,  C. D. Sogge and  A. Stewart.  \emph{Nonlinear hyperbolic equations in infinite homogeneous waveguides}. Comm. PDE {\bf 30} (2005), no. 4-6, 643-661.
\bibitem{FL}F. Planchon and L. Vega. \emph{Bilinear Virial Identities and Applications}. Ann. Sci. \'{E}c. Norm. Sup\'{e}r. (4) {\bf 42} (2009), no. 2, 261-290.

\bibitem{T1} T. Schneider.  \emph{Nonlinear optics in telecommunications}. Springer, Berlin, 2004.
\bibitem{EM1} E. M. Stein, \emph{Singular Integrals and Differentiability Properties of functions}. Princeton
University Press, Princeton, NJ, 1970.
\bibitem{SS} C. Sulem and P. L. Sulem.   \emph{The nonlinear Schr\"{o}dinger equation: self-focusing and wave collapse}.
Applied Mathematical Sciences, 139. Springer, New York, 1999.

\bibitem{TVZ2} T. Tao, M. Visan, X. Zhang. \emph{Global well-posedness and scattering for the defocusing mass-critical nonlinear Schr\"odinger equation for radial data in high dimensions}. Duke Math. J. {\bf 140} (2007), no. 1, 165-202.

\bibitem{TVZ1} T. Tao, M. Visan, X. Zhang. \emph{Minimal-mass blowup solutions of the mass-critical NLS}. Forum Math. {\bf 20} (2008), no. 5, 881-919.


\bibitem{NN} N. Tzvetkov and  N. Visciglia \emph{Small data scattering for the nonlinear Schr\"{o}dinger equation on product spaces}. Comm. PDE {\bf 37} (2012), no. 1, 125-135.








%

%











\end{thebibliography}

\end{document}